\documentclass[10pt]{amsart}
\usepackage{amssymb}
\usepackage{amsmath}
\usepackage{verbatim}
\usepackage{enumerate}
\usepackage{graphicx}
\usepackage{amscd}

\newcommand{\bZ}{\mathbb{Z}}
\newcommand{\bN}{\mathbb{N}}
\newcommand{\bC}{\mathbb{C}}
\newcommand{\bR}{\mathbb{R}}
\newcommand{\bQ}{\mathbb{Q}}
\newcommand{\mr}{\mathrm}
\newcommand{\ra}{\rightarrow}
\newcommand{\xra}{\xrightarrow}
\newcommand{\dd}{\partial}
\newcommand{\be}{\begin{equation}}
\newcommand{\ee}{\end{equation}}
\newcommand{\bt}{\bullet}
\newcommand{\tr}{\mathrm{tr}\;}
\newcommand{\hra}{\hookrightarrow}
\newcommand{\hla}{\hookleftarrow}
\newcommand{\mleft}{\left(\begin{array}}
\newcommand{\mright}{\end{array}\right)}

\newcommand{\Det}{\mathrm{Det}}
\newcommand{\bT}{\mathbb{T}}
\newcommand{\Dens}{\mathrm{Dens}}
\newcommand{\sk}{\mathrm{sk}}
\newcommand{\mc}{\mathcal}
\newcommand{\lec}[1]{\marginpar{\bf \Large #1}}

\newtheorem{thm}{Theorem}[section]
\newtheorem{prop}[thm]{Proposition}
\newtheorem{lemma}[thm]{Lemma}

\newtheorem{corollary}[thm]{Corollary}

\newtheorem{definition}[thm]{Definition}

\newtheorem{hypothesis}[thm]{Hypothesis}

\theoremstyle{remark}
\newtheorem{rem}[thm]{Remark}
\newtheorem{example}[thm]{Example}

\begin{document}
\title{
Lecture notes on torsions}
\author{Pavel Mnev}
\address{Institut f\"ur Mathematik,
Universit\"at Z\"urich,
Winterthurerstrasse 190,
CH-8057, Z\"urich, Switzerland}

\email{pmnev@pdmi.ras.ru}

\maketitle

\begin{abstract}
These are the lecture notes for the introductory course on
Whitehead, Reidemeister and Ray-Singer torsions,
given by the author at the University of Zurich in Spring semester 2014.
\end{abstract}

\tableofcontents

\section*{Preface}

Torsion is a far-reaching generalization of the notion of a (super-)determinant, associated to a chain complex with certain additional structure (e.g. a distinguished class of bases or an inner product). In the setting of topological spaces, a torsion-type invariant constructed by Reidemeister can (sometimes) distinguish between homotopically equivalent spaces. As the first application, this invariant was 
used by Reidemeister and Franz to classify lens spaces.

A conceptual explanation of Reidemeister torsion was found in Whitehead's theory of simple homotopy type. Simple homotopy equivalence for cell complexes is an equivalence relation generated by elementary expansions and collapses. Whitehead torsion is a (complete) obstruction for a homotopy equivalence between cell complexes to be homotopic to a simple homotopy equivalence.
Reidemeister torsion is related to Whitehead torsion via a simple ``ring-changing'' procedure.

Ray-Singer analytic torsion is constructed for a Riemannian manifold and is defined as a product of certain powers of zeta-regularized determinants of Hodge-de Rham Laplacians acting on $p$-forms on the manifold with coefficients in a local system (a fixed flat orthogonal bundle). Ray-Singer torsion can be seen as a ``de Rham counterpart'' of the Reidemeister torsion (cf. de Rham cohomology vs. cellular cohomology). The celebrated Cheeger-M\"uller theorem establishes the equality of these two torsions.

From the standpoint of quantum field theory, Ray-Singer torsion naturally appears as a Gaussian functional integral -- one-loop part of the perturbative partition function in topological field theories (e.g. Chern-Simons theory, $BF$ theory).

What follows is the lecture notes for the course titled ``Torsions'' given by the author at the University of Zurich in Spring semester 2014. The idea was to give a survey of the three incarnations of torsions, mentioned above (Reidemeister, Whitehead, analytic) and discuss some of the applications. Some parts of the course were necessarily sketchy due to time restrictions. In particular, proofs of the difficult 
theorems were omitted (notably, Chapman's theorem on topological invariance of Whitehead torsion and Cheeger-M\"uller theorem on equality of Reidemeister and analytic torsions). 
For the most part the material is standard, with main references being Milnor \cite{Milnor66}, Cohen \cite{Cohen}, Turaev \cite{Turaev01}, Ray-Singer \cite{Ray-Singer}.

\lec{Lecture 1, 20.02.2014}

\section{Introduction}\label{sec: intro}
In this course we will be talking (mostly) about Reidemeister, Whitehead and analytic (Ray-Singer) torsions.
\subsection{Classification of 3-dimensional lens spaces}
Reidemeister torsion appeared first (1935) in the context of classification of 3-dimensional lens spaces.
\begin{definition}[Tietze, 1908]
For $p,q\in \mathbb{N}$ coprime, $q$ defined modulo $p$, the {\it lens space} $L(p,q)$ is defined as the quotient of the 3-sphere by a $q$-dependent free action of $\bZ_p$:
$$L(p,q)=\frac{\{(z_1,z_2)\in \bC^2\;|\; |z_1|^2+|z_2|^2=1\}}{ (z_1,z_2)\sim (\zeta z_1, \zeta^q z_2)} $$
where $\zeta=e^{\frac{2\pi i}{p}}$ is a root of unity.
\end{definition}

Example: $L(2,1)=\bR P^3$ -- the real projective 3-space.

Lens spaces $L(p,q)$ are smooth orientable compact 3-manifolds. For $p$ fixed and $q$ varying, all spaces $L(p,q)$ have the same homology
$$H_0= \bZ,\; H_1 = \bZ_p,\; H_2=0,\; H_3=\bZ$$ same fundamental group $\pi_1=\bZ_p$ and same higher homotopy groups (same as their universal covering space $S^3$). However, they do depend on $q$.

\begin{example}
$L(5,1)$ and $L(5,2)$ are not homotopy equivalent (Alexander, 1919).
\end{example}

\begin{thm}[Classification theorem, Reidemeister, 1935]
\begin{enumerate}
\item \label{lens classification homotopy} $L(p,q_1)$ and $L(p,q_2)$ are homotopy equivalent iff $q_1 q_2\equiv \pm n^2 \mod p$ for some $n\in\bZ$.
\item \label{lens classification homeo} $L(p,q_1)$ and $L(p,q_2)$ are homeomorphic\footnote{Reidemeister's classification was up to PL equivalence; it was proven to be in fact a classification up to homeomorphism later by Brody, 1960.} iff either $q_1 q_2\equiv \pm 1\mod p$ or $q_1\equiv \pm q_2\mod p$.
\end{enumerate}
\end{thm}

\begin{example}
$L(7,1)$ and $L(7,2)$ are homotopy equivalent but not homeomorphic.
\end{example}

Classification up to homotopy (\ref{lens classification homotopy}) is given by the torsion linking form
$\mr{tor}(H_1)\otimes \mr{tor}(H_1)\ra \bQ/\bZ$. The finer classification up to homeomorphism is given by a different sort of invariant, the Reidemeister torsion.

Classification result above was extended to higher-dimensional lens spaces by Franz (1935) and de Rham (1936).

\subsection{R-torsion, the idea}\label{sec: intro, R-torsion}\footnote{``R-torsion'' is a shorthand for ``Reidemeister-Franz-de Rham real representation torsion'', coined by Milnor \cite{Milnor66}.}
\subsubsection{Algebraic torsion} Let
$$C_n\xra{\dd}C_{n-1}\xra{\dd}\cdots\xra{\dd}C_1\xra{\dd}C_0$$
be an acyclic chain complex of finite-dimensional $\bR$-vector spaces.

Morally: $\mr{torsion(C_\bullet)}=``\det (\dd)"$.
This expression has to be made sense of: $\dd$ is not an endomorphism of a vector space, rather it shoots between different vector spaces, also it has a big kernel.

More precisely, denote the image of $\dd$ in $C_k$ by $C_k^{ex}$ (``exact'') and choose some linear complement\footnote{The notation is inspired by the Hodge decomposition theorem for differential forms, although here we use it somewhat abusively for cellular chains.} $C_k^{coex}$ of $C_k^{ex}$ in $C_k$. Thus we have $C_k=C_k^{ex}\oplus C_k^{coex}$ and the boundary map is an isomorphism
\be \dd: C_k^{coex}\xra{\sim}C_{k-1}^{ex} \label{d: coex to ex iso}\ee
Then
\be \mr{torsion}=\prod_{k=1}^n \left({\det}_{C_k^{coex}\ra C_{k-1}^{ex}}\dd\right)^{(-1)^k} \label{torsion = sdet}\ee
where we are calculating determinants of {\it matrices} of isomorphisms (\ref{d: coex to ex iso}) with respect to some chosen bases in $C_\bt^{ex,coex}$. Expression (\ref{torsion = sdet}) generally depends on the chosen bases.

For example, if spaces $C_k$ are equipped with positive-definite inner products $C_k\otimes C_k\ra \bR$, then we may choose $C_k^{coex}$ to be the orthogonal complement of $C_k^{ex}$ and choose any orthonormal basis in every $C_k^{ex,coex}$. Then the torsion (\ref{torsion = sdet}) is well-defined modulo sign (sign indeterminacy is due to the possibility to change the ordering of basis vectors of an orthonormal basis; well-definedness -- since changes between o/n bases are given by orthogonal matrices, which in turn have determinant $\pm 1$).

\begin{rem} Alternating product of determinants in (\ref{torsion = sdet}) should be compared to the alternating sum of traces in the Lefschetz fixed point theorem: for $f:X\ra X$ a diffeomorphism of a manifold $X$ with isolated fixed points, one has
$$\sum_k (-1)^k \tr_{H_k(X,\bR)}\;f_*= \sum_{x\in \{\mbox{fixed points of }f\}} \mr{ind}_x(f)$$
where $f_*$ is the pushforward of $f$ to real homology and on the r.h.s. one has the sum of indices of critical points of $f$.
\end{rem}

\subsubsection{R-torsion in topological setting}
Let $X$ be a finite simplicial (or CW-) complex and $\rho: \pi_1(X)\ra O(m)$ an orthogonal representation of the fundamental group. Let $\tilde{X}$ be the universal covering of $X$. Then set
\be C_\bt(X,\rho)=\bR^m\otimes_{\bZ[\pi_1(X)]}C_\bt(\widetilde{X},\bZ)\label{twisted chains}\ee
where $C_\bt(\widetilde{X},\bZ)$ are the cellular chains of $\widetilde{X}$; $\pi_1(X)$ acts on $\widetilde{X}$ by covering transformations and on $\bR^m$ via representation $\rho$; these actions are extended to actions of the group ring $\bZ[\pi_1(X)]$ by linearity.
Assuming that $\rho$ is such that the twisted chain complex (\ref{twisted chains}) is acyclic, one defines the R-torsion $T(X,\rho)$ as the torsion of the complex $C_\bt(X,\rho)$ with 
the inner product
induced from the cellular basis in $C_\bt(\widetilde{X})$ and the standard (or any orthonormal) basis in $\bR^m$.

Example: for the lens space $L(p,q)$, choose $\rho: \underbrace{\pi_1}_{\simeq \bZ_p}\ra O(2)\simeq U(1)$ sending $1\in \bZ_p$ to a root of unity $\eta=e^{2\pi i s/p}$ different from $1$ (since for $\eta=1$ the twisted chain complex is not acyclic). Then the R-torsion is
\be T(L(p,q),\eta)=|(1-\eta)(1-\eta^r)|^{-2}\in \bR_+\label{T for L(p,q)}\ee
where $r\in \bZ_p$ is the reciprocal of $q$, i.e. $q\cdot r\equiv 1\mod p$.
Thus the R-torsion in this case is a collection of positive numbers (one for each admissible $\rho$).

\subsection{Some properties of R-torsion}
\label{sec: properties of T}

\begin{enumerate}[(i)]
\item For a CW-complex represented as a (non-disjoint) union of two subcomplexes $X$ and $Y$, one has the Mayer-Vietoris-type gluing formula
    \be T(X\cup Y)=T(X)\cdot T(Y)\cdot T(X\cap Y)^{-1} \label{Mayer-Vietoris for torsions}\ee
    assuming everything is well-defined. (Instead of representations of $\pi_1$ in this case it is better to think of putting a flat Euclidean vector bundle over $X\cup Y$ and then restricting it to $X$, $Y$, $X\cap Y$).
    Note that (\ref{Mayer-Vietoris for torsions}) means that $\log T$ satisfies the same relation as the Euler characteristic,
    $$\chi(X\cup Y)=\chi(X)+\chi(Y)-\chi(X\cap Y)$$
    Formula (\ref{Mayer-Vietoris for torsions}) suggests that the torsion counts some local objects (like $\chi$ counts cells with signs); there is indeed an interpretation of this kind, due to David Fried \cite{Fried87}. We will return to this interpretation later.
\item \label{R-torsion product property} For $Y$ simply connected,
$$T(X\times Y)=T(X)^{\chi(Y)}$$
\item \label{R-torsion even-dim property} $X$ a cell decomposition of an even-dimensional compact orientable manifold,
$$T(X)=1$$
(Essentially a manifestation of Poincar\'e duality.)
\item Torsion is invariant under subdivisions of cell complexes. (It is this statement that justifies a posteriori the seemingly ad hoc definition of the R-torsion, with its puzzling alternating determinants and the preferred basis induced from the cellular basis).
\item A stronger result holds in fact: torsion is invariant under homeomorphisms (Chapman's theorem).
\item\label{sh invariance of T} R-torsion is invariant under simple homotopy equivalence of CW-complexes (Whitehead). We will elaborate on this point right away.
\end{enumerate}

\subsection{Simple homotopy equivalence and Whitehead torsion}
Whitehead's original idea was to describe homotopy equivalences of CW-complexes as sequences of elementary local moves (cf. Reidemeister's moves for knot diagrams, Tietze moves for group presentations, Pachner's moves for triangulations of manifolds). The suggested moves are: elementary collapses (for simplicial complexes, having a simplex $\sigma$ with a free face $\sigma'$, one can make a retraction to the horn comprised be the boundary of $\sigma$ minus $\sigma'$, not touching the rest of the complex), and the inverse moves -- elementary expansions (filling horns, in the simplicial setting). It turned out that these moves do not generate all homotopy equivalences, rather they generate a more refined relation -- the simple homotopy equivalence.
\begin{center}
  \includegraphics[scale=0.5]{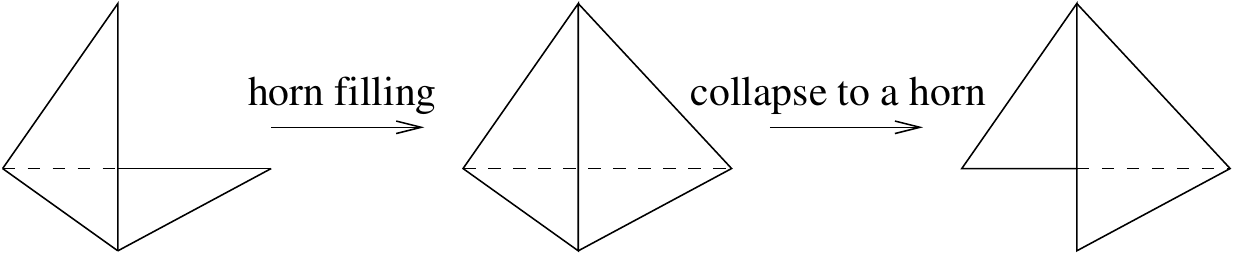}
\end{center}

For $f:X\ra Y$ a homotopy equivalence of CW-complexes, 
consider the
complex of relative chains $C_\bt(\widetilde{M}_f,\widetilde X)$, with $M_f$ the mapping cylinder for $f$ and tilde denoting the universal covering, as a complex of free $\bZ[\pi_1]$-modules. It can be prescribed a torsion $\tau(f)$, taking values in the Whitehead group of the fundamental group, $Wh(\pi_1)$, defined in terms of the algebraic $K_1$ functor evaluated on the group ring $\bZ[\pi_1]$. $\tau(f)$ is called the Whitehead torsion of the homotopy equivalence $f$. 
\begin{thm}[Whitehead]
$\tau(f)=1\in Wh(\pi_1)$ if and only if $f$ is homotopic to a simple homotopy equivalence.
\end{thm}

\subsubsection{Relation between Whitehead torsion and R-torsion}
For $f:X\ra Y$ a homotopy equivalence, $\rho:\pi_1\ra O(m)$ an orthogonal representation, one has
$$T(Y)=\rho_*(\tau(f))\cdot T(X)$$
where the map $Wh(\pi_1)\ra \bR_+$
comes from the induced map on reduced $K_1$-groups $\rho_*: \bar K_1(\bZ[\pi_1])\ra \bar K_1(M_m(\bR))\simeq \bR_+$
with $M_m(\bR)$ standing for the ring of matrices $m\times m$ over reals.

In particular, for $f$ a simple homotopy equivalence, $\tau(f)=1$ and hence $T(X)=T(Y)$ (thus the property (\ref{sh invariance of T}) of Section \ref{sec: properties of T} follows from Whitehead's theorem).

Example: for a homotopy equivalence between lens spaces $f: L(7,1)\ra L(7,2)$, the Whitehead torsion is nontrivial, since $T(L(7,1))\neq T(L(7,2))$ by the explicit formula (\ref{T for L(p,q)}).

\subsection{Some applications of torsions}

\begin{itemize}
\item Classification of lens spaces (Reidemeister 1935, Franz 1936).
\item The s-cobordism theorem (Mazur-Stallings-Barden). Smale's h-cobordism theorem \cite{Smale60} states than an $h$-cobordism (i.e. one where both inclusions $M_\mr{in}\hra N\hookleftarrow M_\mr{out}$ are homotopy equivalences) between simply connected manifolds of dimension $\geq 5$ is homeomorphic to a cylinder $N\simeq M_\mr{in}\times [0,1]$. An immediate consequence is the Poincar\'e conjecture in dimension $\geq 5$. The s-cobordism theorem states that an h-cobordism between not necessarily simply connected manifolds of dimension $\geq 5$ is homeomorphic to a cylinder if and only if the Whitehead torsion of the inclusion is trivial $\tau(M_\mr{in}\hra N)=1$. A refinement of this is the statement that homeomorphism classes of h-cobordisms form a torsor over the Whitehead group $Wh(\pi_1)$.
\item Milnor's counter-example to Hauptvermutung -- a conjecture by Steinitz and Tietze that a pair of homeomorphic simplicial complexes should have a common subdivision. Milnor's counter-example is a pair of complexes of dimension $\geq 6$; the proof of non-existence of common subdivision relies on the R-torsion \cite{MilnorHaupt}.\footnote{In the case of manifolds, Hauptvermutung is true in dimension $\leq 3$, i.e. every topological manifold admits a unique PL structure, and is false in higher dimensions. In dimension $\geq 5$, the complete obstruction for a topological manifold to have a PL structure is the Kirby-Siebenmann obstruction in $H^4(M,\bZ_2)$. If it vanishes, the PL structures form a torsor over $H^3(M,\bZ_2)$. In dimension 4, a topological compact simply-connected manifold can have infinitely many PL structures (examples due to Donaldson), or none (the $E_8$ manifold, Freedman).}
\item Relation of the R-torsion of the complement of a knot in $S^3$ to Alexander's polynomial of the knot.
\item Calculation of the symplectic volume of the moduli space of flat $G$-bundles over a surface $\Sigma$ (Witten, 1991 \cite{Witten91}).
\item R-torsion appears in certain topological quantum field theories as a (1-loop) part of the perturbative partition function (cf. e.g. \cite{Schwarz79,Witten89}), we will comment on this below.
\end{itemize}

\lec{Lecture 2, 27.02.2014}

\subsection{Analytic torsion}\label{sec: intro, RS torsion}
Let $(M,g)$ be a compact oriented Riemannian $n$-manifold without boundary and $E$ a rank $m$ Euclidean vector bundle over $M$ with a flat $O(m)$-connection $\nabla$. The connection defines the twisted de Rham differential $d_\nabla: \Omega^{\bt}(M,E)\ra \Omega^{\bt+1}(M,E)$.
\begin{definition}[Ray-Singer \cite{Ray-Singer}] Assuming that the twisted cohomology $H^\bt_{d_\nabla}(M,E)$ vanishes, one defines the analytic (Ray-Singer) torsion of $M$ with coefficients in the flat bundle $(E,\nabla)$ as
\be T_a(M,E)=\prod_{k=1}^n\left({\det}_{\Omega^k(M,E)}\Delta_{(k)}\right)^{-\frac{(-1)^k\cdot k}{2}} \label{RS torsion}\ee
where $\Delta_{(k)}=d_\nabla d^*_\nabla+d^*_\nabla d_\nabla: \Omega^k(M,E)\ra \Omega^k(M,E)$ is the Hodge-de Rham Laplacian twisted by the flat connection.
\end{definition}
Determinants of Laplacians in (\ref{RS torsion}) are understood in the sense of zeta-regularization: operator $\Delta_{(k)}$ has discrete positive eigenvalue spectrum $\{\lambda_j\}$ with spectral density behaving asymptotically as $\rho(\lambda)\sim\lambda^{\frac{n}{2}-1}$ at $\lambda\ra\infty$. One defines the zeta function
$\zeta_{\Delta_{(k)}}(s)=\sum_{j}\lambda_j^{-s}$ by analytic continuation from domain $\mr{Re}(s)>s_0=\frac{n}{2}$ where the sum on the right converges. The continuation is a meromorphic function regular at $s=0$. One defines the zeta-regularized determinants as
$$\det\Delta_{(k)}=e^{-\zeta'_{\Delta_{(k)}}(0)}$$

Powers of determinants in (\ref{RS torsion}) can be explained as follows. One has the Hodge decomposition $\Omega^k(M,E)=\Omega^k_{ex}\oplus \Omega^k_{coex}$ of $E$-valued $k$-forms on $M$ into $d_\nabla$-exact and $d^*_\nabla$-exact (=coexact). Spectrum of $\Delta$ on $\Omega^k$ splits into the spectrum on $\Omega^k_{ex}$ and the spectrum on $\Omega^k_{coex}$: $\{\lambda_j^{(k)}\}=\{\lambda_j^{(k),ex}\}\cup \{\lambda_j^{(k),coex}\}$. Moreover, $\{\lambda_j^{(k),coex}\}=\{\lambda_j^{(k+1),ex}\}$ since one can apply $d_\nabla$ to a coexact $k$-eigenform to obtain an exact $(k+1)$-eigenform with the same eigenvalue. Thus in (\ref{RS torsion}) every eigenvalue $\lambda$ appears twice: first time as a contribution to $\det\Delta_{(k)}$, coming from the coexact part, then as a contribution to $\det\Delta_{(k+1)}$, coming from the exact part. Therefore one can rearrange (\ref{RS torsion}) as
\be T_a(M,E)=\prod_k \left({\det}_{\Omega_{coex}^k(M,E)}\Delta_{(k)}\right)^{\frac{(-1)^k}{2}} \label{RS torsion in terms of coex forms}\ee
where the powers are $\frac{(-1)^k}{2}=\left(-\frac{(-1)^k\cdot k}{2}\right)+\left(-\frac{(-1)^{k+1}\cdot (k+1)}{2}\right)$. Finally, on coexact forms $\Delta_{(k)}=d^*_\nabla d_\nabla$ and thus one can view (\ref{RS torsion in terms of coex forms}) as a (way to make sense of) determinant of $d_\nabla: \Omega^\bt_{coex}\ra \Omega^{\bt+1}_{ex}$, as in (\ref{torsion = sdet}).

Ray and Singer proved in particular the following properties of $T_a$.
\begin{itemize}
\item $T_a(M,E)$ does not depend on the choice of Riemannian metric $g$ (cf. invariance of R-torsion w.r.t. subdivision).
\item For even-dimensional $M$, the analytic torsion $T_a=1$, when defined.
\item For $N$ simply connected, $T_a(M\times N,p^*E)=T_a(M)^{\chi(N)}$, where $p: M\times N\ra M$ is the projection to the first factor and $E$ is an acyclic flat $O(m)$-bundle over $M$.
\end{itemize}
The last two properties are direct analogues of properties (\ref{R-torsion even-dim property}), (\ref{R-torsion product property}) of R-torsion (Section \ref{sec: properties of T}).

Ray and Singer conjectured (and gave evidence for), and Cheeger and M\"uller proved (independently and by different methods) the following result.
\begin{thm}[Cheeger-M\"uller]
Analytic torsion coincides with R-torsion, i.e. for $M$ a compact oriented Riemannian manifold without boundary endowed with a cell decomposition $X$ and a flat $O(m)$-bundle $(E,\nabla)$ over $M$, one has
$$T_a(M,E)=T^c(X,\rho)$$
where the representation $\rho: \pi_1(M)\ra O(m)$ is given by the holonomy of $\nabla$.\footnote{One chooses the base point $x_0$ to be a vertex of $X$ and trivializes the fiber of $E$ over $x_0$ (fixes the isomorphism $E_{x_0}\simeq \bR^m$). Then for a closed path $\gamma\subset M$ starting and ending at $x_0$, one has the holonomy $\rho(\gamma):=\mr{hol}_\gamma(\nabla)\in \mr{Aut}(E_{x_0})\simeq O(m)$ depending only on the class of $\gamma$ in the fundamental group.}
\end{thm}
The torsion on the r.h.s. is the {\it cochain} version of the $R$-torsion, i.e. constructed from cellular cochains instead of chains. It is related to the chain $R$-torsion discussed above as $T^c(X,\rho)=T(X,\rho)^{-1}$.

Morally, one has two models for the torsion: combinatorial model (Reidemeister), based on cellular chains, and analytic model (Ray-Singer), based on differential forms. The situation is very similar to cellular vs. de Rham model for cohomology of a manifold.

\subsection{Torsion as a Gaussian integral}
For an acyclic chain complex $(C_\bt,\dd)$, one can write the torsion as a Gaussian integral over a vector superspace
\be T(C_\bt)=\alpha\int_{L\subset F}\mu\; e^{i S(A,B)} \label{torsion as BV integral}\ee
Where the even and odd parts of $F$ are:
\be F^{even}=(\oplus_k C_{2k}) \oplus (\oplus_k C_{2k+1}^*),\quad F^{odd}=(\oplus_k C_{2k+1}) \oplus (\oplus_k C_{2k}^*)  \label{torsion as BV integral space of fields}\ee
A vector $(A,B)\in F$ has a component $A=\sum_k A_{(k)}\in C_\bt$ (a non-homogeneous chain with components of alternating parity) and component $B=\sum_k B_{(k)}\in C_\bt^*$. The expression in the exponential in (\ref{torsion as BV integral}) is:
$$S=\langle B, \dd A \rangle=\sum_{k} \langle B_{(k-1)}, \dd A_{(k)}\rangle $$
where $\langle,\rangle$ stands for the pairing between the chain space $C_k$ and its dual.

The subspace $L\subset F$ over which the integral is taken in (\ref{torsion as BV integral}) is defined as $C_\bt^{coex}\oplus (C_\bt^{ex})^*$, with parities prescribed as in (\ref{torsion as BV integral space of fields}). The integration measure $\mu$ is 
the coordinate Berezin measure for the basis on $L$ induced from the preferred bases on $C_\bt^{ex,coex}$.

The proportionality coefficient in (\ref{torsion as BV integral}) is $\alpha=(2\pi)^{-\dim C^{coex}_{odd}}\cdot i^{-\dim C^{coex}_{even}}$.

Substituting in (\ref{torsion as BV integral}) the chain complex of a CW-complex twisted by a representation of $\pi_1$, we get an integral formula for the R-torsion.

The analog of (\ref{torsion as BV integral}) for the analytic torsion is:
\be T_a(M,E)=\int_{d_\nabla^*A=0,\; d_\nabla^*B=0} e^{i\int_M (B\stackrel{\wedge}{,}d_\nabla A)}\;\mathcal{D}A\;\mathcal{D}B \label{RS torsion as BV integral}\ee
The right hand side is a Gaussian functional integral over a subspace in $F=\Omega^\bt(M,E)\oplus \Omega^\bt (M,E^*)$ cut out by the Lorentz gauge condition $d_\nabla^* A=d_\nabla^* B=0$. Parity of the component $A_{(k)}$ in $A=\sum_k A_{(k)}\in \Omega^\bt(M,E)$ is opposite to parity of $k$ and the parity of $B_{(k)}$ in $B=\sum_{k} B_{(k)}\in \Omega^\bt(M,E^*)$ is same as parity of $n-k$.

The right hand side of (\ref{RS torsion as BV integral}) is the partition function of the abelian $BF$ theory on $M$, twisted by the flat bundle $E$. Proper way to understand this functional integral leads to the definition of the Ray-Singer torsion on the left hand side \cite{Schwarz79}.

By analogy, the right hand side of (\ref{torsion as BV integral}) in the case of R-torsion can be viewed as the partition function of a discrete abelian $BF$ theory living on (chains of) a CW-complex, twisted by a
local system.

\subsubsection{In Chern-Simons theory}
For $M$ a rational homology 3-sphere and $G$ a compact simply-connected simple Lie group, the semi-classical contribution of an acyclic flat connection $A_0$ to the Chern-Simons partition function
\begin{multline} Z_{CS}(M,G,\hbar)=\\
=\int_{G-\mbox{connections on }M}\mathcal{D}A\quad \exp
\frac{i}{\hbar}\underbrace{\int_M\tr\left(\frac{1}{2} A\wedge dA+\frac{1}{3}A\wedge A\wedge A\right)}_{S_{CS}(A)}
\label{CS path integral}
\end{multline}
is given by
\be [Z_{CS}(M,G,\hbar)]_{A_0} =e^{\frac{i}{\hbar}S_{CS}(A_{0})}\cdot T_a(M,A_0)^{\frac12}\cdot e^{i\phi}\cdot (1+\underbrace{O(\hbar)}_{\mbox{higher-loop corrections}}) \label{CS semi-classical}\ee
where the square root of the analytic torsion\footnote{$T_a(M,A_0)$ is evaluated for the adjoint bundle of the trivial $G$-bundle over $M$ with the connection $A_0$ in adjoint representation.} appears from the Gaussian part of (\ref{CS path integral}) (with integrand evaluated in the neighborhood of $A_0$). The phase $e^{i\phi}$ is complicated and depends on a choice of additional geometric structure -- the {\it framing} of $M$, see \cite{Witten89} for details.

\subsection{``Torsion counts circles''}
Recall the Poincar\'e-Hopf theorem that states that for a manifold $M$ and a vector field $v$ on it with isolated zeroes, the Euler characteristic counts zeroes of $v$:
\be \chi(M)=\sum_{\mbox{zeroes of }v}\mr{ind}_x(v) \label{Poincare-Hopf}\ee

D. Fried \cite{Fried87} proved that for certain classes of manifolds $M$ endowed with a local system $E$, there is an open subset in the space of flows $\phi$ on $M$ for which the following ``Lefschetz formula'' holds:
\be T(M,E)=\left|\prod_\gamma \det (\mr{Id}- \Delta(\gamma)\,\rho_E(\gamma))^{-(-1)^{u(\gamma)}} \right| \label{Lefschetz formula for T}\ee
where the product runs over closed orbits $\gamma$ of the flow $\phi$, $\rho_E(\gamma)$ is the holonomy of $E$ around $\gamma$, $u(\gamma)$ is the dimension of the unstable bundle over $\gamma$ (spanned by eigenvectors of the return map $T_x M/T_x \gamma\ra T_x M/T_x \gamma$ around $\gamma$ with eigenvalues of norm $>1$); $\Delta(\gamma)=\pm 1$ depending on whether or not the unstable bundle is orientable.

Formula (\ref{Lefschetz formula for T}) is not always true, however it was proven in \cite{Fried87} in a set of cases, including:
\begin{itemize}
\item Fiber bundles over $S^1$.
\item $S^1$-bundles.
\item Complement of a link in $S^3$.
\item Constant curvature spaces.
\item Thurston's geometric 3-manifolds.
\end{itemize}

Morally, (\ref{Lefschetz formula for T}) is very similar to the Poincar\'e-Hopf formula for the Euler characteristic (\ref{Poincare-Hopf}), but in case of torsion, the local contributions come from closed orbits of a generic vector field, instead of zeroes.

\lec{Lecture 3, 06.03.2014}

\section{Whitehead group}
(Reference: \cite{Milnor66}, \S\S2,6.)
For a unital ring $A$, and each $n\geq 1$, one has the group of invertible (=non-singular) matrices $n\times n$, $GL(n,A)$ and an inclusion $GL(n,A)\hra GL(n+1,A)$ by
$$M\mapsto \left(\begin{array}{ll}M & 0\\0& 1 \end{array}\right)$$
Direct limit (union) of the sequence of inclusions
$$GL(1,A)\subset GL(2,A)\subset\cdots$$
is called the infinite general linear group $GL(A)$.

A matrix is called {\it elementary} if it coincides with the identity matrix  except for one off-diagonal element.
Notation: $E_{ij}$ is the matrix with  entry $1$ in place $(i,j)$ and zeroes elsewhere.

\begin{lemma}[J. H. C. Whitehead]
The subgroup $E(A)\subset GL(A)$ generated by all elementary matrices coincides with the commutator subgroup of $GL(A)$.
\end{lemma}

\begin{proof}
(i) For every triple of pairwise distinct $i,j,k$, one has the identity
$$(I+aE_{ij})(I+E_{jk})(1-aE_{ij})(I-E_{jk})=(I+aE_{ik})$$
Therefore each elementary matrix in $GL(n,A)$ is a commutator if $n\geq 3$.

(ii) The following identitites show that each commutator $XYX^{-1}Y^{-1}$ in $GL(n,A)$ can be expressed as a product of elementary matrices in $GL(2n,A)$.
\be \left(\begin{array}{ll} XYX^{-1}Y^{-1} & 0 \\ 0 & 1\end{array}\right)= \left(\begin{array}{ll} X & 0 \\ 0 & X^{-1}\end{array}\right) \left(\begin{array}{ll} Y & 0 \\ 0 & Y^{-1}\end{array}\right) \left(\begin{array}{ll} (YX)^{-1} & 0 \\ 0 & YX\end{array}\right) \ee
\be
\left(\begin{array}{ll} X & 0 \\ 0 & X^{-1}\end{array}\right)=\left(\begin{array}{ll} I & X \\ 0 & I\end{array}\right) \left(\begin{array}{ll} I & 0 \\ I-X^{-1} & I\end{array}\right) \left(\begin{array}{ll} I & -I \\ 0 & I\end{array}\right) \left(\begin{array}{ll} I & 0 \\ I-X & I\end{array}\right) \label{Whitehead lemma diag(X, X_inverse)}\ee
\be \left(\begin{array}{ll} I & X \\ 0 & I\end{array}\right)=\prod_{i=1}^n\prod_{j=n+1}^{2n}(1+x_{ij}E_{ij})
\ee
\end{proof}

It follows that $E(A)\subset GL(A)$ is a normal subgroup.
\begin{definition}
The quotient
$$K_1(A)=GL(A)/E(A)$$
is called the Whitehead group of the ring $A$.
\end{definition}

$K_1$ is a covariant functor: a ring homomorphism $\phi:A\ra A'$ induces a homomorphism of abelian groups $\phi_*: K_1(A)\ra K_1(A')$.

For $A$ a  {\it commutative} ring, the determinant is a group homomorphism $\det: GL(A)\ra GL(1,A)=A^*$ with $A^*$ the group of units (=group of invertible elements) of $A$. This induces the splitting of $GL(A)\simeq A^*\times SL(A)$ where $SL(A)\subset GL(A)$ is the subgroup of matrices of determinant $1$. The quotient $SK_1(A)=SL(A)/E(A)$ is called the {\it special Whitehead group} and one has the splitting
$$K_1(A)=A^*\oplus SK_1(A)$$
For many important classes of rings, $SK_1$ is trivial, i.e. every matrix of determinant $1$ can be reduced to the identity matrix by elementary row operations (=left multiplication by elementary matrices).

Some examples:
\begin{example}For $A=F$ a field, $SK_1(F)=\{1\}$ and $K_1(F)=F^*=F-\{0\}$ the multiplicative group of $F$.\end{example}

\begin{example}For $A=\bZ$ the integers, $SK_1(\bZ)=\{1\}$, so $K_1(\bZ)=\bZ^*=\{\pm 1\}$.
More generally, if $A$ is a commutative ring with euclidean algorithm, $SK_1(A)=\{1\}$.\end{example}

\begin{example}[\cite{Cohen}, p.45] For $A=\bZ[\bZ_p]$, for $p$ an integer, $SK_1=\{1\}$.\end{example}

Some examples with $SK_1\neq \{1\}$:
\begin{example}[\cite{Milnor66}] $A=\bR[x,y]/(x^2+y^2-1)$, here the matrix
\be \mleft{rl}x & y\\-y & x\mright\in SL(A) \label{Milnor's 2x2 matrix}\ee represents a nontrivial element of $SK_1(A)$.\footnote{Proof (Milnor \cite{Milnor66}, p. 422):
matrix (\ref{Milnor's 2x2 matrix}) defines a map $m:S^1\ra SL(2,\bR)\subset SL(\bR)$ (which is homotopic to $O(\bR)$); elementary row operations on a matrix induce homotopy equivalences of such maps. On the other hand, $\pi_1 SL(\bR)\simeq \bZ_2$ and $m$ represents the generator of $\pi_1$ -- a non-contractible loop.
} More precisely, $SK_1(A)\simeq \bZ_2$.\footnote{\textbf{Exercise:} show that the square of the matrix (\ref{Milnor's 2x2 matrix}) represents the unit in $SK_1(A)$ by explicitly factorizing it into a product of elementary matrices in $GL(A)$.}
\end{example}
\begin{example}[\cite{Milnor66}]
$A=\bZ[\bZ\times \bZ_{23}]$. $SK_1 (A)=\bZ_3\oplus (?)$. The second term is an unknown (as of \cite{Milnor66}) $23$-primary group.\end{example}
\begin{example}[\cite{Cohen}, after R. C. Alperin, R. K. Dennis and M. R. Stein] Even for a finite abelian $G$, $SK_1(\bZ[G])$ is typically nonzero. E.g. for $A=\bZ[\bZ_2\oplus (\bZ_3)^3]$, $SK_1\simeq (\bZ_3)^6$.\end{example}

Examples with $A$ noncommutative:
\begin{example}
For $A=F$ a skew-field (= a division ring) and $U=F^*$ its group of units, $K_1(F)$ is the abelianization $U/[U,U]$. The natural homomorphism $GL(n,F)\ra K_1(F)$ is the ``noncommutative determinant'' of Dieudonn\'e.\end{example}
\begin{example}
For $M_n(A)$ the ring of all $n\times n$ matrices over a ring $A$,
\be K_1(M_n(A))=K_1(A) \label{K_1(M_n(A))}\ee
Indeed, since a $k\times k$ matrix with entries in $M_n(A)$ can be viewed as a $kn\times kn$ matrix with entries in $A$, we have $GL(k,M_n(A))\simeq GL(kn, A)$. Taking the limit $n\ra\infty$ and abelianizing, we get (\ref{K_1(M_n(A))}).
\end{example}

\begin{definition}
$K_1(A)$ contains a $\bZ_2$-subgroup as the image of $\{\pm 1\}\in GL(1,A)$. One calls $\bar K_1(A)=K_1(A)/\{\pm 1\}$ the {\it reduced} Whitehead group of $A$.
\end{definition}

\begin{example}$\bar K_1(\bZ)=\{1\}$.\end{example}

\begin{example}$\bar K_1(\bR)\simeq \bR_+$ with explicit isomorphism given by
\be (a_{ij})\mapsto |\det(a_{ij})|\label{bar K_1(R)}\ee
\end{example}

Motivation for the reduced Whitehead group: two matrices that differ by a permutation of rows are mapped to the same element of $\bar K_1(A)$.

\begin{definition}
For $A=\bZ[G]$ the group ring of group $G$, the Whitehead group of $G$ is defined as
$$Wh(G)=\bar K_1(\bZ[G])/\mr{im}(G)$$
where we are taking the quotient over the image of $G\subset GL(1,\bZ[G])$ in $\bar K_1$, i.e. we have an exact sequence
$$1\ra G/[G,G]\ra \bar K_1(\bZ[G])\ra Wh(G)\ra 1$$
\end{definition}

$Wh$ is a covariant functor from groups to abelian groups. The image of an inner automorphism $\phi:G\ra G$ is $\phi_*: Wh(G)\ra Wh(G)$ (since $\phi: g\mapsto h g h^{-1}$ induces a conjugation on $GL(n,\bZ[G])$ by the matrix $\mr{diag}(h,h,\ldots)$ which induces the identity automorphism on the abelianization).

Examples:
\begin{example}$Wh(\{1\})=\{1\}$ (since $K_1(\bZ)=\{\pm 1\}$).\end{example}
\begin{example}$Wh(\bZ)=\{1\}$ (Higman; see \cite{Cohen} pp. 42--43 for the proof).\end{example}
\begin{example} $Wh(\bZ\oplus\cdots \oplus \bZ)=\{1\}$ (Bass-Heller-Swan, a rather difficult theorem).\end{example}
\begin{example}
If $G$ is an abelian group which contains an element $x$ of order $q\not\in\{1,2,3,4,6\}$, then $Wh(G)\neq \{1\}$ (\cite{Cohen}, pp. 44--45). E.g. for $G=\bZ_5$ and $x$ the generator, $u=x+x^{-1}-1\in \bZ[\bZ_5]$ is a nontrivial unit (since $(x+x^{-1}-1)(x^2+x^{-2}-1)=1$).\end{example}
\begin{example}
$Wh(\bZ_p)$ is a free abelian group of rank $[p/2]+1-\delta(p)$ where $\delta(p)$ is the number of divisors of $p$. In particular, $Wh(\bZ_p)=\{1\}$ iff $p\in\{1,2,3,4,6\}$.
\item $Wh$ of any free group is trivial (Stallings, Gersten).
\end{example}

$Wh$ does not behave well w.r.t. direct products, e.g. $Wh(\bZ_3)=Wh(\bZ_4)=\{1\}$, but $Wh(\bZ_3\times \bZ_4)\simeq \bZ$.
However, for a free product one has
$$Wh(G\ast G')=Wh(G)\oplus Wh(G')$$
(Stallings).

\subsection{Aside}
\subsubsection{On $K_0$}
Related to $K_1$ is the functor $K_0$: for a ring $A$, $K_0(A)$ is the abelian group with one generator $\langle P\rangle$ for each finitely generated projective module $P$ over $A$ (considered modulo isomorphisms) and one relation $\langle P\oplus Q \rangle= \langle P \rangle + \langle Q \rangle$ for each pair of f. g. projective modules. (In other words, $K_0(A)$ is the Grothendieck group of the monoid of isomorphism classes of f. g. projective modules over $A$). For $h: A\ra A'$ a homomorphism of rings, one has a homomorphism of abelian groups $h_*:K_0(A)\ra K_0(A')$ sending
$$h_*:\quad \langle P\rangle\mapsto \langle A'\otimes_A P\rangle$$
In case when $A$ is commutative, $K_0(A)$ becomes a commutative ring, with
$$\langle P\rangle \langle Q\rangle= \langle P\otimes_A Q\rangle$$

For $A$ commutative, $K_1(A)$ is a module over $K_0(A)$, i.e. there is a product operation
$$K_0(A)\otimes K_1(A)\ra K_1(A)$$
sending $$\langle P\rangle \otimes \langle \underbrace{\alpha}_{\in \mr{Aut}(M)}\rangle \mapsto \langle \underbrace{\mr{id}_P\otimes \alpha}_{\in \mr{Aut}(P\otimes M)}\rangle$$
with $M$ a free module over $A$.

\subsubsection{Relation to topological K-theory}
For $X$ a compact space and $A=C^0_\bC(X)$ the ring of continuous complex-valued functions on $X$,
a f. g. projective module is isomorphic to the space of sections of a unique complex vector bundle over $X$. Hence
$$K_0(C^0_\bC(X))\simeq K^0(X)$$
-- the Atiyah-Hirzebruch group of virtual complex vector bundles over $X$.

Similarly, the special Whitehead group
$$SK_1(C^0_\bC(X))=[X,SL(\bC)]$$
-- the group of homotopy classes of maps from $X$ to $SL(\bC)$, which in turn is the group\footnote{Product of (equivalence classes of) matrices in $SK_1$ corresponds to the Whitney sum of vector bundles, because by relation (\ref{Whitehead lemma diag(X, X_inverse)}), modulo elementary matrices one has the equivalence $\mleft{cl}AB & 0 \\ 0 & 1\mright \sim \mleft{cl}AB & 0 \\ 0 & 1\mright\cdot \mleft{cl}B^{-1} & 0 \\ 0 & B\mright= \mleft{cl}A & 0 \\ 0 & B\mright$. On the r.h.s. we have the transition function for the Whitney sum of vector bundles defined by $A$ and $B$.} of complex vector bundles with structure group $SL(\bC)$ over the suspension $SX$ modulo stable equivalence of vector bundles. This group is closely related to $K^{-1}(X)$.

Similar remarks hold for real-valued functions $C^0_\bR(X)$ and the real topological $K$-theory $KO^0(X)$, $KO^{-1}(X)$.

\lec{Lecture 4, 13.03.2014}

\subsection{Conjugation}
In a group ring $\bZ[G]$ one has an involutive anti-automorphism\footnote{I.e. $\overline{x+y}=\bar{x}+\bar{y}$ and $\overline{xy}=\bar{y}\bar{x}$.}
$$x=\underbrace{\sum n_i g_i}_{\in\bZ[G]}\quad \mapsto \bar{x}=\underbrace{\sum n_i g_i^{-1}}_{\in\bZ[G]}$$
It induces an anti-automorphism on matrices
$$A=(a_{ij})\in GL(\bZ[G])\quad \mapsto A^\dagger=(\bar{a}_{ji})\in GL(\bZ[G])$$
which in turn induces involutive automorphisms on the abelianization $K_1(\bZ[G])$ and on the Whitehead group $Wh(G)$. One calls this involution on $Wh(G)$, sending $\omega\mapsto \bar\omega$, the {\it conjugation}.

\section{Torsion of a chain complex}
\subsection{Milnor's definition of torsion}
Let $A$ be a unital ring.\footnote{A mild technical assumption: free modules of ranks $r$ and $s$ over $A$ are not isomorphic for $r\neq s$. This is true e.g. if there is a ring homomorphism $A\ra A'$ for some commutative ring $A'$. An example of a ring where this assumption fails: take $A$ to be the ring of matrices over $\bR$ of infinite (countable) size with finitely many nonzero entries in every column. Then $A\oplus A\simeq A$ as left $A$-modules, where the isomorphism is given by interlacing the columns.}
By an $A$-module we mean a finitely generated left $A$-module.

For $F$ a free module over $A$ and $b=(b_1,\ldots,b_k)$ and $c=(c_1,\ldots,c_k)$ two bases in $F$, one has $c_i=\sum_j a_{ij} b_j$, $(a_{ij})\in GL(k,A)$ the transition matrix. Denote $[c/b]$ the image of $(a_{ij})$ in $\bar K_1(A)$.

\begin{definition}
We call bases $b$ and $c$ in a free module {\it equivalent} if $[c/b]=1$.
\end{definition}

Example: bases that differ by the ordering of basis vectors are equivalent.\footnote{It is important that we define the equivalence via the image of $(a_{ij})$ in the {\it reduced} Whitehead group.}

Example: bases that differ by a triangular change (i.e. the transition matrix $(a_{ij})$ is upper triangular with units on the diagonal) are equivalent.

Given a short exact sequence of free modules
$$0\ra E\ra F\ra G\ra 0$$
and bases $e=(e_1,\ldots, e_k)$, $g=(g_1,\ldots, g_l)$ in $E$ and $G$, one can construct a basis $``eg''$ in $F$ as
\be eg=(e_1,\ldots,e_k,g'_1,\ldots, g'_l) \label{product basis}\ee
for some lifts $g'_i$ of $g_i$ from $G$ to $F$. The class of $eg$ is well-defined (independent of the lifts).

If $\bar e, \bar g$ - another pair of bases in $E$ and $G$, one has
$$[\bar e \bar g/eg]=[\bar e/e]\cdot [\bar g/g]$$

Given a inclusions of free modules $F_0\subset F_1\subset F_2$ and bases $b_1$ for $F_1/F_0$ and $b_2$ for $F_2/F_1$, we obtain the basis $b_1 b_2$ for $F_2/F_0$ by applying the construction above to the short exact sequence
$$0\ra F_1/F_0\ra F_2/F_0\ra F_2/F_1 \ra 0$$

More generally, for a sequence of inclusions o free modules $F_0\subset F_1\subset \cdots\subset F_k$ and given bases $b_i$ for $F_i/F_{i-1}$, there is a basis $b_1 b_2\cdots b_k$ for $F_k/F_0$, well defined up to equivalence.

Let $$C_n\xra{\dd} C_{n-1}\xra{\dd}\cdots\xra{\dd} C_1\xra{\dd} C_0$$
be a chain complex of $A$-modules, such that each $C_i$ is free with a preferred basis $c_i$, and each homology group $H_i(C)$ is free with a preferred basis $h_i$.\footnote{We may have $H_i=0$, then by convention the zero module comes with a unique basis.} Denote $B_i,Z_i\subset C_i$ the image of $\dd: C_{i+1}\ra C_i$ and the kernel of $\dd: C_{i}\ra C_{i-1}$ respectively (``boundaries'' and ``cycles'').

\begin{hypothesis}\footnote{We will show (Section \ref{sec: stably free modules} below) that this hypothesis can in fact be dropped.}\label{hypothesis: B_i are free} Each $B_i$ is a free module.
\end{hypothesis}

Choose some bases $b_i$ for $B_i$. We have inclusions of free modules $0\subset B_i\subset Z_i\subset C_i$ with bases $b_i$ for $B_i/0$, $h_i$ for $Z_i/B_i=H_i$, $b_{i-1}$ for $\dd: C_i/Z_i\stackrel{\sim}{\ra} B_{i-1}$. This yields a combined basis $b_i h_i b_{i-1}$ on $C_i$.

\begin{definition}
We define the torsion\footnote{This is the inverse of Milnor's convention for torsion and agrees with 
Turaev \cite{Turaev01}. Our convention is consistent with (\ref{torsion = sdet}). 
} as
\be \tau(C)=\prod_i [c_i/(b_i h_i b_{i-1})]^{(-1)^i}\quad \in \bar K_1(A)\label{def: torsion}\ee
\end{definition}

For a different choice of bases $\bar b_i$ for $B_i$, the torsion changes by
$$\prod_i ([\bar b_i/b_i]\cdot [\bar b_{i-1}/ b_{i-1}])^{-(-1)^i}=1$$
(a telescopic product), hence $\tau(C)$ does not depend on the choice of bases $b_i$. On the other hand, it does depend on $c_i$ and $h_i$. E.g. if one makes a change $c_i\mapsto \bar c_i$ and $h_i\mapsto \bar h_i$, the torsion changes as
$$\tau\mapsto \tau\cdot \frac{\prod_i [\bar c_i/c_i]^{(-1)^i}}{\prod_i [\bar h_i/h_i]^{(-1)^i}}$$

\subsection{Stably free modules}
\label{sec: stably free modules}
\begin{definition}
An  $A$-module $M$ is called {\it stably free} if for some (finitely generated) free module $F$, $M\oplus F= F'$ is free.
\footnote{
Example: take $A=C^0(S^2)$ to be the ring of continuous real-valued functions on the $2$-sphere, which we view as the unit sphere in $\bR^3$. Sections of the tangent bundle $TS^2\ra S^2$ comprise a non-free, but stably free module, since $TS^2$ is a non-trivial bundle, but its sum with the normal bundle (which is trivial) is a trivial rank $3$ bundle. On the other hand, the module of sections of the $O(1)$ bundle over $\bC P^1$, realized as a rank $2$ real vector bundle $O(1)_\bR$ over $S^2$, is projective but not stably free (since the 2nd Stiefel-Whitney class of $O(1)_\bR$ is nonzero).
} The rank of $M$ is then defined as $rk(M)=rk(F')-rk(F)$.
\end{definition}

\begin{lemma}\label{Lm: stably free SES} Let $0\ra X\ra Y\ra Z\ra 0$ be a short exact sequence of $A$-modules. If $Y$ and $Z$ are stably free, then $X$ is also stably free.
\end{lemma}
\begin{proof} $Z$ is projective, hence the sequence splits, so that $Y\simeq X\oplus Z$. By assumption, there exist free modules $F, F', F''$ such that $Z\oplus F= F'$, $Y\oplus F= F''$. Hence, $X\oplus F'\simeq F''$.
\end{proof}

\begin{lemma} \label{Lm: B_i and Z_i are stably free}
Let $C_n\ra C_{n-1}\ra\cdots \ra C_1\ra C_0$ be a chain complex of free $A$-modules and assume that homology modules $H_i$ are free. Then all the boundary and cycle modules $B_i$, $Z_i$ are stably free.
\end{lemma}
\begin{proof}
Follows by induction from Lemma \ref{Lm: stably free SES} and the exact sequences
$$
\begin{array}{ccccccccc}
0&\ra& Z_i& \ra& C_i& \xra{\dd}& B_{i-1}&\ra& 0 \\
0&\ra& B_i&\ra& Z_i& \ra& H_i& \ra& 0
\end{array}
$$
\end{proof}

Let $F_i$ be a standard free module of rank $i$ with standard basis $f_1,\ldots, f_i$. One has the inclusion $F_i\subset F_{i+1}$ as a submodule generated by the first $i$ basis elements.

\begin{definition}
An {\it $s$-basis} $b$ for a stably free module $M$ is a basis $(b_1,\ldots,b_{r+t})$ for some $M\oplus F_t$, where $t\in \bN$.
\end{definition}

Given two $s$-bases $b=(b_1,\ldots,b_{r+t})$ and $c=(c_1,\ldots, c_{r+u})$ for $M$, the symbol
$$[c/b]\in \bar K_1(A)$$
is defined as follows. Choose an integer $v\geq \mr{max}(t,u)$ and extend $b$ and $c$ to bases
in $M\oplus F_v$ by setting $b_{r+i}=0\oplus f_i$ for $i>t$ and $c_{r+j}=0\oplus f_j$ for $j>u$.
Denoting the extended bases in $M\oplus F_v\simeq F_{r+v}$ by $\tilde b$, $\tilde c$, we set
$$[c/b]:=[\tilde c/\tilde b]\quad \in \bar K_1(A)$$
-- the class of the transition matrix from $\tilde b$ to $\tilde c$.\footnote{In other words: an $s$-basis is a choice of the isomorphism  $\alpha: M\oplus F_t\stackrel{\sim}{\ra} F_{r+t}$ to a free module with standard basis. Given two such isomorphisms $\alpha$ and $\beta$ and using the standard inclusions $F_{r+t}\subset F_N\supset F_{r+u}$, $[c/b]$ is the class of $\alpha\beta^{-1}\in \mr{Aut}(F_N)$ in $\bar K_1(A)$, for $N$ sufficiently large.}

Given a short exact sequence of modules
\be 0\ra X\ra Y\ra Z\ra 0 \label{SES of stable free modules}\ee
with $X$ and $Z$ stably free and equipped with $s$-bases $x$ (a basis in $X\oplus F_t$) and $z$ (a basis in $Z\oplus F_u$), one can define an $s$-basis $xz$ in $Y$, by applying the construction (\ref{product basis}) to the sequence
$$0\ra X\oplus F_t\ra Y\oplus F_{t+u}\ra Z\oplus F_u\ra 0$$
-- the sum of (\ref{SES of stable free modules}) and the sequence $0\ra F_t\ra F_{t+u}\ra F_u\ra 0$ with the standard inclusion and projection to the submodule generated by the last $u$ basis elements.

Using this construction, the Hypothesis \ref{hypothesis: B_i are free} can be dropped in the definition of torsion (\ref{def: torsion}): modules $B_i$ are stably free by Lemma \ref{Lm: B_i and Z_i are stably free}. In this case, we are using arbitrary $s$-bases $b_i$ for modules $B_i$ in the formula for torsion (\ref{def: torsion}).

\subsection{Alternative definition via a chain contraction (acyclic case)}
Assume that the chain complex $(C_\bt,\dd)$ is acyclic and let $\kappa: C_\bt\ra C_{\bt+1}$ be a map satisfying
$$\dd\kappa+\kappa \dd = \mr{id}:\quad C_\bt\ra C_\bt$$
and
\be\kappa^2=0\label{kappa^2=0}\ee
(the {\it chain contraction}). Note that choosing $\kappa$ is equivalent to choosing a right splitting of the short exact sequence
$$B_i\hra C_i \xra{\dd}
B_{i-1}$$
as $C_i=B_i\oplus \kappa(B_{i-1})$.
\footnote{A splitting $\kappa$ exists since $B_{i-1}$ is projective.}

Then we have mutually inverse maps
\begin{eqnarray}
\dd+\kappa:& C_{even}\ra C_{odd}, \\  \dd+\kappa:& C_{odd}\ra C_{even} \label{d+kappa odd to even}
\end{eqnarray}
where $C_{even}=\oplus_{k}C_{2k}$, $C_{odd}=\oplus_k C_{2k+1}$.

The torsion of $(C_\bt,\dd)$ with respect to the preferred bases $c_\bt$ can then be defined as
\be \tau(C)=\left[\frac{(\dd+\kappa)_{even\ra odd}c_{even}}{c_{odd}}\right]\quad\in \bar K_1(A) \label{def: torsion via chain contraction}\ee

\begin{rem} In the case of a chain complex of real vector spaces, $A=\bR$, with the isomorphism (\ref{bar K_1(R)}), the formula for the torsion becomes
\be \tau=|\det\;(\dd+\kappa)_{even\ra odd}| \label{torsion for real vector spaces via chain contraction}\ee
where we are taking the determinant of the matrix of the map $(\dd+\kappa)_{even\ra odd}$ with in the chosen bases in $C_{even}$, $C_{odd}$.
\end{rem}

\begin{lemma}
\begin{enumerate}[(i)]
\item \label{Lm: torsion via kappa, 1}In the case of an acyclic chain complex $C_\bt$, definition (\ref{def: torsion via chain contraction}) is equivalent to (\ref{def: torsion}).
\item \label{Lm: torsion via kappa, 2} In the case of an acyclic complex of vector spaces, $A=\bR$,  formula (\ref{torsion for real vector spaces via chain contraction}) is equivalent to the formula (\ref{torsion = sdet}) of Section \ref{sec: intro, R-torsion}.\footnote{Up to taking the absolute value, which we omitted in the initial exposition.}
\end{enumerate}
\label{Lm: torsion via kappa}
\end{lemma}
\begin{proof}
(\ref{Lm: torsion via kappa, 1}) Assume for simplicity that Hypothesis \ref{hypothesis: B_i are free} holds, i.e. modules $B_i$ are free (generalization to stably free case is straightforward). Choose some basis $b_i$ for each $B_i$. Then each chain space $C_i$ has a basis
$$b_i\cdot b_{i-1}=(b_{i,1},\ldots,b_{i,n_i},\kappa(b_{i-1,1}),\ldots, \kappa(b_{i-1,n_{i-1}}))$$
where we chose to lift the basis  $b_{i-1}$ to $C_i$ using the chain contraction $\kappa$ and $n_i$ is the rank of $B_i$. In this basis, operator $\dd+\kappa$ sends basis elements to basis elements. Thus the matrix of $(\dd+\kappa)_{even\ra odd}$ in this basis is a permutation matrix (since we know it is invertible). Hence,
$$\left[\frac{(\dd+\kappa)_{even\ra odd}\prod_{i\;even}b_i b_{i-1}}{\prod_{i\;odd}b_i b_{i-1}}\right]=1\quad \in \bar K_1(A)$$
Using this, we can calculate the r.h.s. of (\ref{def: torsion via chain contraction}):
\begin{multline*}\left[\frac{(\dd+\kappa)_{even\ra odd}c_{even}}{c_{odd}}\right] = \\
=\left[\frac{\prod_{i\;odd}b_i b_{i-1}}{c_{odd}}\right]\cdot
\underbrace{\left[\frac{(\dd+\kappa)_{even\ra odd}\prod_{i\;even}b_i b_{i-1}}{\prod_{i\;odd}b_i b_{i-1}}\right]}_{=1}\cdot
\left[\frac{c_{even}}{\prod_{i\;even}b_i b_{i-1}}\right] =\\
=\prod_i [c_i/(b_i b_{i-1})]^{(-1)^i}
\end{multline*}
Which got us back to the definition (\ref{def: torsion}) specialized to the acyclic case. Here we exploited the obvious property $[x/z]=[x/y]\cdot [y/z]$ for a triple of bases $x,y,z$ in a free module $F$.

(\ref{Lm: torsion via kappa, 2}) To have the setting of formula (\ref{torsion = sdet}), fix a splitting $C_i=B_i\oplus C^{coex}_i$ for each $i$, and fix bases $b_i$ in $B_i$ and $f_i$ in $C^{coex}_i$. Transition matrix from the basis $b_i b_{i-1}$ in $C_i$ to the basis $c_i=b_i f_{i}$ has the form
$$\mleft{cc} I & 0 \\ 0& D_i \mright$$
for $D_i$ the transpose of the matrix of the boundary operator $\dd: C^{coex}_i\ra B_{i-1}$ with respect to the bases $f_i$, $b_{i-1}$. Thus, using the definition (\ref{def: torsion}) of the torsion and the isomorphism (\ref{bar K_1(R)}), we have
$$\tau(C)=\prod_i  [c_i/(b_i b_{i-1})]^{(-1)^i}=\prod_i |\det D_i|^{(-1)^i}\quad \in \bR_+$$
Which is indeed the r.h.s. of (\ref{torsion = sdet}).
\end{proof}

\begin{rem}
Note that (\ref{Lm: torsion via kappa, 1}) of Lemma \ref{Lm: torsion via kappa} implies that the right hand side of (\ref{def: torsion via chain contraction}) is independent of the choice of chain contraction $\kappa$. In fact, it is possible to show that (\ref{def: torsion via chain contraction}) is true even without the assumption that $\kappa^2=0$.\footnote{In this case the inverse to $(\dd+\kappa)_{even\ra odd}$ is given by $(\mr{id}+\kappa^2)^{-1}(\dd+\kappa):\; C_{odd}\ra C_{even}$ instead of (\ref{d+kappa odd to even}), with $(\mr{id}+\kappa^2)^{-1}=\mr{id}-\kappa^2+\kappa^4-\cdots $ (only finitely many terms are nonzero, since $C_\bt$ has finite degree).}\footnote{
Proof: Let $\kappa$ be a chain contraction with $\kappa^2=0$ and $\tilde\kappa$ another one, possibly with $\tilde\kappa^2\neq 0$. Then $(\dd+\tilde\kappa)_{even\ra odd}=(\dd+\kappa)_{even\ra odd}\cdot M$ where
$M=\mr{id}+(\dd+\kappa)^{-1}\cdot (\tilde\kappa-\kappa)=\mr{id}+(\dd+\kappa)\cdot (\tilde\kappa-\kappa)=\mr{id}+\dd\cdot (\tilde\kappa-\kappa)+\kappa\tilde\kappa:\; C_{even}\ra C_{even}$.
Here the last term on the r.h.s. shifts the degree by $+2$ and the middle term vanishes on $B_i$ and maps $\mr{im}(\kappa)_i$ to $B_i$. Hence the matrix of $M$ is block lower-triangular with units on the diagonal with respect to the decomposition $C_{even}=B_0\oplus\mr{im}(\kappa)_2\oplus B_2\oplus \mr{im}(\kappa)_4\oplus B_4\oplus\cdots$. Therefore the class of $M$ in $\bar K_1(A)$ is $1$ and thus changing $\kappa\mapsto \tilde\kappa$ does not change the value of the r.h.s. in the definition of torsion (\ref{def: torsion via chain contraction}).
}
\end{rem}

\lec{Lecture 5, 20.03.2014}

\subsection{Multiplicativity with respect to short exact sequences}
Let
\be 0\ra C'_\bt \xra{\iota} C_\bt \xra{p} C''_\bt\ra 0 \label{SES of chain complexes}\ee
be a short exact sequence of chain complexes of free $A$-modules (maps $\iota$, $p$ are assumed to be chain maps). Let $c'_i$, $c''_i$ be some bases in $C'_i$, $C''_i$ and $c'_i c''_i$ the product basis in $C_i$.

\begin{lemma}\label{Lm: multiplicativity of torsions}
Assuming that complexes $C'_\bt$, $C_\bt$, $C''_\bt$ are acyclic, for the torsions with respect to bases $c'$, $c''$, $c'c''$ one has
$$\tau(C_\bt)=\tau(C'_\bt)\cdot \tau(C''_\bt)\quad\in \bar{K}_1(A)$$
\end{lemma}

\begin{proof}
Choose some bases $b'_i$, $b''_i$ in boundary modules $B'_i\subset C'_i$, $B''_i\subset C''_i$.
One can lift the basis $b''_i$ to a collection of linearly independent elements $p^{-1}(b''_i)\subset B_i$.\footnote{Arbitrary lifting $p^{-1}(x)$ of a boundary in $C''_i$ is not guaranteed to be a boundary in $C_i$.
However by assumption $x=\dd y$ for some $y\in C''_{i+1}$. Choose an arbitrary lifting $p^{-1}(y)$ and set $p^{-1}(x):=\dd p^{-1}(y)\in B_i$. 
}
Restricting the sequence (\ref{SES of chain complexes}) to the boundaries in degree  $i$, one gets a short exact\footnote{\label{footnote: exactness of SES of boundaries}Exactness is trivial in the left term and follows from existence of the lifting $p^{-1}(b''_i)$ for the right term. For the middle term, note that for $x\in \ker(p)\cap B_i$, one has $x=\iota(y)$ for some $y\in C'_i$ and $0=\dd\iota(y)=\iota(\dd y)$, hence $y$ is a cycle and, by acyclicity of $C'_\bt$, boundary. This proves exactness of (\ref{SES of boundaries}).
}
sequence
\be 0\ra B'_i\xra{\iota} B_i \xra{p} B''_i\ra 0 \label{SES of boundaries}\ee
The basis for $B_i$ is constructed as $b_i=\iota(b'_i)\cup p^{-1}(b''_i)=:b'_i b''_i$. Now we can calculate the torsion of $C_\bt$:
\begin{multline*}
\tau(C_\bt)=\prod_i\left[\frac{c_i}{b_i b_{i-1}}\right]_{C_i}^{(-1)^i}=\prod_i \left[\frac{c'_i c''_i}{b'_i b''_i b'_{i-1}b''_{i-1}}\right]_{C_i}^{(-1)^i}=\\
=\prod_i \left(\left[\frac{c'_i}{b'_i b'_{i-1}}\right]_{C'_i}\cdot \left[\frac{c''_i}{b''_i b''_{i-1}}\right]_{C''_i}\right)^{(-1)^i}=
\tau(C'_\bt)\cdot \tau(C''_\bt)
\end{multline*}
\end{proof}

Dropping the assumption that complexes $C'_\bt$, $C_\bt$, $C''_\bt$ are acyclic, one has a long exact sequence in homology
\be \cdots\ra H'_\bt \xra{\iota_*} H_\bt \xra{p_*} H''_{\bt}\xra{\delta} H'_{\bt-1}\ra\cdots \label{LES in homology}\ee
Let us denote this sequence $\chi$.
\begin{lemma}\label{Lm: multiplicativity of torsions, non-acyclic case}
One has
$$\tau(C_\bt)=\tau(C'_\bt)\cdot \tau(C''_\bt)\cdot \tau(\chi)$$
where the torsions are evaluated with respect to bases $c'$, $c' c''$, $c''$ in chains and arbitrarily chosen bases $h'$, $h$, $h''$ in homology.
\end{lemma}

\begin{proof}
In the non-acyclic case, the sequence (\ref{SES of boundaries}) fails to be exact in the middle term and attains homology
\be \frac{ \ker (p:\; B_i\ra B''_{i})}{\iota(B'_{i})}= \ker(i_*:\; H'_i\ra H_i) \label{failure of exactness in sequence of boundaries}\ee
[Explanation: $\ker (p:\; B_i\ra B''_i)=\iota(C'_i)\cap B_i= \iota(Z'_i)\cap B_i$ (for the last equality, $p(x)=0$ and $x=\dd y$ implies $x=\iota(z)$ with $\dd z=0$, cf. footnote \ref{footnote: exactness of SES of boundaries}). Thus we have a sequence
$$0\ra \frac{ \ker (p:\; B_i\ra B''_{i})}{\iota(B'_{i})} \ra H'_i \xra{i_*} H_i\ra\cdots $$
where the first map sends $x\mapsto [z]\in H'_i$. The sequence is obviously exact in the first place, and also exact in the second place by definition of $\iota_*$. This implies (\ref{failure of exactness in sequence of boundaries}).]

Choose some bases $\beta'_i$, $\beta_i$, $\beta''_i$ in the ``boundaries'' of the long exact sequence in homology (\ref{LES in homology}), i.e. in $\mr{im}(\delta)\subset H'_i$, $\mr{im}(\iota_*)\subset H_i$, $\mr{im}(p_*)\subset H''_i$ respectively. Then we have the product bases $\beta'_i \beta_i$, $\beta_i \beta''_i$, $\beta''_i \beta'_{i-1}$ in the entire homology modules $H'_i$, $H_i$, $H''_i$.

Using (\ref{failure of exactness in sequence of boundaries}), we construct the basis $b_i=b'_i \beta'_i b''_i$ in $B_i$.

Now we can calculate the torsion of $C_\bt$:
\begin{multline*}
\tau(C_\bt)=\prod_i \left[\frac{c_i}{b_i h_i b_{i-1}}\right]_{C_i}^{(-1)^i}=\prod_i \left[\frac{c'_i c''_i}{(b'_i\beta'_i b''_i) h_i (b'_{i-1}\beta'_{i-1} b''_{i-1})}\right]_{C_i}^{(-1)^i}=\\
=\prod_i \left(\left[\frac{c'_i c''_i}{(b'_i\beta'_i b''_i) \beta_i \beta''_i (b'_{i-1}\beta'_{i-1} b''_{i-1})}\right]_{C_i}\cdot \left[\frac{\beta_i\beta''_i}{h_i}\right]_{H_i}\right)^{(-1)^i}=\\
=\prod_i \left(
\left[\frac{c'_i}{b'_i b'_{i-1}\beta'_{i}\beta_i}\right]_{C'_i}\cdot
\left[\frac{c''_i}{b''_i b''_{i-1}\beta''_{i}\beta'_{i-1}}\right]_{C''_i}\cdot
\left[\frac{\beta_i\beta''_i}{h_i}\right]_{H_i}
\right)^{(-1)^i}=\\
=\prod_i \left(
\left[\frac{c'_i}{b'_i b'_{i-1}h'_i}\right]_{C'_i}\cdot
\left[\frac{c''_i}{b''_i b''_{i-1}h''_i}\right]_{C''_i}\cdot
\left[\frac{h'_i}{\beta'_i \beta_i}\right]_{H'_i}\cdot
\left[\frac{\beta_i\beta''_i}{h_i}\right]_{H_i}\cdot
\left[\frac{h''_i}{\beta''_i \beta'_{i-1}}\right]_{H''_i}
\right)^{(-1)^i}=\\
=\tau(C'_\bt)\cdot\tau(C''_\bt)\cdot \tau(\chi)
\end{multline*}
\end{proof}

\subsection{Determinant lines}\label{sec: det lines}
Let us specialize to the case of finite-dimensional vector spaces over $A=\bR$. For $V$ a vector space denote
$$\Det\,V:=\wedge^{\dim V}V$$
For the dual vector space, we have
$$\Det\,V^*\cong (\Det\,V)^*$$
where the isomorphism comes from the pairing
$$
\begin{array}{ccccc}
\Det\,V^*&\otimes  &\Det\,V&\ra& \bR \\
(\alpha_{\dim V}\wedge\cdots\wedge \alpha_{1})&\otimes & (v_1\wedge\cdots \wedge v_{\dim V})&\mapsto &\det \langle\alpha_i,v_j\rangle
\end{array}
$$
We will denote $(\Det\,V)^*=:(\Det\,V)^{-1}$.\footnote{The motivation for this notation is that for a 1-dimensional vector space $L$, to a vector $v\in L$, one can canonically associate a vector $v^{-1}\in L^*$ characterized by the property $\langle v^{-1},v \rangle =1$.}

For an {\it isomorphism} of vector spaces $f:V\xra{\sim} W$, there is an associated isomorphism of the determinant lines:
\begin{eqnarray*}
\Det\,f: &\Det\,V  &\xra{\sim}  \Det\,W\\
& v_1\wedge\cdots\wedge v_{\dim V} &\mapsto  f(v_1)\wedge \cdots\wedge f(v_{\dim V})
\end{eqnarray*}
If $f:V\xra{\sim} V$ is an automorphism, then
$$\Det\,f=\det\,f\quad \in \bR-0$$
-- the usual determinant of a matrix, by the canonical identification $\mr{Aut}(\Det\,V)\cong \bR-0$.

For a short exact sequence of vector spaces
$$0\ra U\xra{\iota} V\xra{p} W\ra 0$$
one has
\be \Det\,V\cong \Det\,U\otimes \Det\,W \label{Det for a SES of vector spaces}\ee
with the isomorphism given by
\begin{multline}
\underbrace{(u_1\wedge\cdots\wedge u_{\dim U})}_{\in\;\Det\,U}\otimes \underbrace{(w_1\wedge\cdots\wedge w_{\dim W})}_{\in\; \Det\,W}\mapsto \\
\mapsto \underbrace{\iota(u_1)\wedge\cdots\wedge\iota(u_{\dim U})\wedge p^{-1}(w_1)\wedge\cdots\wedge p^{-1}(w_{\dim W})}_{\in\; \Det\,V}
\label{Det(V) = Det(U) otimes Det(W) isomorphism}
\end{multline}
for some liftings $p^{-1}(w_i)\in V$ of $w_i\in W$. The expression on the right does not depend on the choice of liftings.

For a chain complex of finite-dimensional vector spaces
$$C_n\ra C_{n-1}\ra\cdots \ra C_1\ra C_0$$
we define the determinant line as
$$\Det\,C_\bt=\bigotimes_{i=0}^n (\Det\,C_i)^{(-1)^i}$$

\begin{lemma}
For a short exact sequence of chain complexes of vector spaces
$$0\ra C'_\bt\ra C_\bt\ra C''_\bt\ra 0$$
one has
\be \Det\,C_\bt\cong \Det\,C'_\bt\otimes \Det\,C''_\bt \label{Det for a SES of complexes}\ee
\end{lemma}
\begin{proof}Follows immediately from (\ref{Det for a SES of vector spaces}).
\end{proof}

\begin{rem} Explicit isomorphism (\ref{Det(V) = Det(U) otimes Det(W) isomorphism}) implies that the isomorphism $$\Det\,U\otimes \Det\,W\cong \Det\,W\otimes \Det\,U$$ contains a ``Koszul sign'' $(-1)^{\dim U\cdot \dim W}$, i.e.
$$(\bigwedge_1^{\dim U} u_i)\otimes  (\bigwedge_1^{\dim W} w_j)\quad \mapsto\quad (-1)^{\dim U\cdot \dim W}\cdot(\bigwedge_1^{\dim W} w_j)\otimes (\bigwedge_1^{\dim U} u_i)$$
Since in (\ref{Det for a SES of complexes}) one has to reshuffle $\Det\,C'_i$ and $\Det\,C''_j$ multiple times, there is a complicated Koszul sign.
\end{rem}

Note that up to this moment we did not use the boundary maps.

\begin{lemma}\label{Lm: Det(C)=Det(H)}
Determinant line of a chain complex of vector spaces is canonically isomorphic to the determinant line of homology:
\be \Det\,C_\bt\cong \Det\,H_\bt \label{Det(C)=Det(H)}\ee
\end{lemma}
\begin{proof} Applying (\ref{Det for a SES of complexes}) to the short exact sequence of complexes
$$0\ra Z_\bt\ra C_\bt\xra{\dd} B_{\bt-1}\ra 0$$
we get
\be \Det\,C_\bt\cong \Det\,Z_\bt\otimes (\Det\,B_\bt)^{-1} \label{Lm: Det(C)=Det(H) eq1}\ee
Next, using the short exact sequence
$$ 0\ra B_\bt\ra Z_\bt\ra H_\bt\ra 0 $$
we get
\be \Det\,Z_\bt\cong \Det\,B_\bt\otimes \Det\,H_\bt \label{Lm: Det(C)=Det(H) eq2}\ee
Combining this with (\ref{Lm: Det(C)=Det(H) eq1}), we obtain
$$\Det\,C_\bt\cong \Det\,H_\bt\otimes \Det\,B_\bt\otimes (\Det\,B_\bt)^{-1}\cong \Det\,H_\bt$$
\end{proof}

Let us introduce the notation $\bT: \Det\,C_\bt\ra \Det\,H_\bt$ for the isomorphism (\ref{Det(C)=Det(H)}).

\begin{rem}
Note that 
rescaling $\dd\mapsto \alpha\dd$ by a factor $\alpha\in\bR-0$ does not change the homology, but results in rescaling $\bT$ by a factor $\alpha^{-\sum (-1)^i\dim B_i}$.
\end{rem}

We want to consider determinant lines modulo signs, so that we do not need to remember complicated Koszul signs in (\ref{Det for a SES of complexes}), (\ref{Det(C)=Det(H)}).
\begin{definition} \label{def: density}
A {\it density} on a vector space $V$ is a map $F$ from the set of bases $\{v_i\}$ in $V$ to $\bR$, such that for an automorphism $m\in GL(V)$, one has
$$F(m v)=|\det m|\cdot F(v)$$
\end{definition}

Denote the space of densities on $V$ by $\mr{Dens}(V)$. One has
$$\mr{Dens}(V)\cong (\Det\,V^*)/\{\pm 1\}
$$

\subsection{Torsion via determinant lines}

Given a complex $C_\bt$ of finite-dimensional vector spaces and an element $$\mu\in \Det\,C_\bt/\{\pm 1\}$$
one may call the torsion its image
$$\bT(\mu)\in \Det\,H_\bt/\{\pm 1\}$$

If $C_\bt$ is acyclic, $\bT(\mu)\in \Det\,H_\bt/\{\pm1\}\simeq \bR_+$ is a positive real number.

If $C_\bt$ is equipped with a basis $c_\bt$, one can construct the element $\mu$ as
\be\mu=\left(\bigotimes_{i\;even} c_{i,1}\wedge\cdots\wedge c_{i,\dim C_i}\right)\otimes \left(\bigotimes_{i\;odd} c^*_{i,1}\wedge\cdots\wedge c^*_{i,\dim C_i}\right) \quad \in \Det\,C_\bt/\{\pm 1\}\label{mu from basis}\ee
with $c_i^*$ the basis in $C_i^*$ dual to $c_i$.
If furthermore $H_\bt$ is equipped with a basis $h_\bt$, one similarly constructs an element $\mu_{H}\in \Det\,H_\bt/\{\pm 1\}$.
\begin{lemma} \label{Lm: torsion via det line = Milnor's}
The ratio $$\tau(C_\bt)=\bT(\mu)/\mu_H\quad \in\bR_+$$
with $\mu$ and $\mu_H$ constructed from bases $c_i$, $h_i$ as above, coincides with the torsion, as in Milnor's definition (\ref{def: torsion}).
\end{lemma}

\lec{Lecture 6, 27.03.2014}

\subsubsection{Multiplicativity of torsions (Lemma \ref{Lm: multiplicativity of torsions, non-acyclic case}) in terms of determinant lines}
\label{sec: multiplicativity via det lines}
Given a short exact sequence of complexes of vector spaces
$$0\ra C'_\bt\ra C_\bt\ra C''_\bt\ra 0$$
the statement of Lemma \ref{Lm: multiplicativity of torsions, non-acyclic case} can be reinterpreted as commutativity of the following diagram
\be
\begin{CD} \Det\, C_\bt @<\cong<< \Det\, C'_\bt\otimes \Det\, C''_\bt\\
@V\bT VV @VV\bT'\otimes \bT'' V\\
\Det\, H_\bt @<\cong<< \Det\, H'_\bt\otimes \Det\, H''_\bt
\end{CD}
\label{Det square}\ee
where the map in the upper row is (\ref{Det for a SES of complexes}). The isomorphism in the lower row is obtained by applying (\ref{Det(C)=Det(H)}) to the long exact sequence in homology (\ref{LES in homology}):
$$\underbrace{\Det\,\chi}_{\cong \Det\,0=\bR}\cong \Det\, H''_\bt\otimes(\Det\, H_\bt)^{-1}\otimes \Det\, H'_\bt
\quad \Rightarrow \quad \Det\, H_\bt\cong \Det\, H'_\bt\otimes \Det\, H''_\bt$$
More specifically, to bases $c',c'',c=c'c'',h',h'',h$ in chains and homology, one associates elements in corresponding determinant lines, $\mu',\mu'',\mu,\mu_{H'},\mu_{H''},\mu_H$. Then we have a square of elements of (\ref{Det square}):
$$
\begin{CD} \mu @<<< \mu'\otimes \mu''\\
@V\bT VV @VV\bT'\otimes \bT'' V\\
\underbrace{\tau(C')\tau(C'')\tau(\chi)}_{=\tau(C)}\mu_H @<<< (\tau(C')\mu_{H'})\otimes (\tau(C'')\mu_{H''})
\end{CD}
$$

\subsubsection{``Euclidean complexes''}
If $C_\bt$ comes with a positive definite inner product on chain spaces, $(,): C_i\otimes C_i\ra \bR$,
one can construct $\mu\in \Det\,C_\bt/\{\pm1\}$ by formula (\ref{mu from basis}) for any collection of orthonormal bases $\{c_i\}$.\footnote{
The choice of bases does not matter, since transition to another set of orthonormal bases induces a change of $\mu$ by a product of determinants of transition matrices, which are in turn orthogonal, and have determinants $\pm 1$.
} In this situation one can define the induced inner product on the determinant line
\be (,)_{\Det\, C}:\;\Det\,C_\bt\otimes \Det\,C_\bt\ra \bR \label{norm on Det(C)}\ee
uniquely characterized by the property $(\mu,\mu)_{\Det\,C}=1$.

Let $\dd^*:C_\bt\ra C_{\bt+1}$ be the adjoint  of the boundary operator with respect to the inner product on $C_\bt$. Then one can define ``Laplacians''
$$\Delta_i=\dd \dd^*+\dd^*\dd:\quad C_i\ra C_i$$
by standard Hodge theory argument, one has the following decomposition of chain spaces
\be C_i=\underbrace{\underbrace{\ker\Delta_i}_{\simeq H_i}\oplus B_i}_{Z_i}\oplus \mr{im}\,\dd^* \label{Hodge decomp of Euclidean chain complex}\ee

Thus one gets two inner products on the determinant line of homology $\Det\, H_\bt$:
\begin{itemize}
\item The pushforward of $(,)_{\Det\,C}$ by $\bT$ yields the inner product
$$(\bt,\bt)_{\Det\, H}=(\bT^{-1}\bt,\bT^{-1}\bt)_{\Det\, C}:\,\Det\, H_\bt\otimes \Det\, H_\bt\ra \bR$$
\item Inner product on chains restricted to harmonic chains yields an inner product on homology
$(,)_H:H_i\otimes H_i \ra \bR$. Applying the construction (\ref{norm on Det(C)}) to homology with $(,)_H$, we get the inner product $(,)_{Harm}: \Det\,H_\bt\otimes \Det \, H_\bt\ra \bR$.
\end{itemize}

Denote $\nu\in \Det\,H_\bt/\{\pm1\}$ a vector of unit length with respect to $(,)_{Harm}$.
\begin{lemma}
\be \bT(\mu)/\nu=\prod_i ({\det}' \Delta_i)^{\frac{(-1)^i}{2}i}\quad\in\bR_+ \label{Lm: alg RS formula eq}\ee
where ${\det'}$ stands for the product of non-zero eigenvalues.
\end{lemma}
(Cf. the formula for Ray-Singer torsion (\ref{RS torsion}) in Section \ref{sec: intro, RS torsion} and explanations thereafter; different sign in the exponent is related to the fact that now we are dealing with chains instead of cochains).

Note that by Lemma \ref{Lm: torsion via det line = Milnor's} the l.h.s. of (\ref{Lm: alg RS formula eq}) is the torsion of the complex $C_\bt$ calculated with respect to any orthonormal w.r.t. $(,)$ basis in chains and any orthonormal w.r.t. $(,)_H$ basis in homology.

\subsubsection{Torsion as a pushforward of a density (the fiber Batalin-Vilkovisky integral)}
For a chain complex of finite-dimensional vector spaces $C_\bt$, the isomorphism
$$\bT: \Det\,C_\bt/\{\pm1\}\ra \Det\, H_\bt/\{\pm 1\}$$
can be understood as a pushforward of densities (fiber integral) $\bT=p_*$ with respect to a projection $p:C_\bt\ra H_\bt$ by the following construction.

Define a $\bZ$-graded vector space
\be F=C_\bt[1]\oplus C^\bt[-2]\label{F=chains+cochains}\ee
where $C^i:=(C_{-i})^*$ are the dual spaces to chains (=``cochains''). Convention for the degree shifts is: for $V_\bt$ a graded vector space, $(V_\bt[k])_i=V_{i+k}$. In particular, the component of $F$ in degree $i$ is $F_i=C_{i+1}\oplus C^{i-2}=C_{i+1}\oplus (C_{2-i})^*$. Coordinates on $F_i$ by convention have degree $-i$. $F$ is equipped with a canonical degree $-1$ (constant) symplectic form $\omega$, induced by the canonical pairing  $\langle,\rangle: \;C_i\otimes C^{-i}\ra \bR$.\footnote{
If $c_i$ and $c^*_i$ are chosen bases in $C_i$ and dual bases in $C^{-i}$, denote the associated coordinates $x^{(i)}_\alpha$, $y^{(i)\alpha}$ on $C_i[1]$ and $C^{-i}[-2]$ (as parts of $F$) respectively. The degrees are: $\deg(x^{(i)}_\alpha)=1-i$, $\deg(y^{(i)\alpha})=i-2$. The symplectic form is then written in coordinates as
$\omega=\sum_i \sum_\alpha (-1)^i\delta x^{(i)}_\alpha\wedge \delta y^{(i)\alpha}$. The (internal)
degree of $\omega$ is $-1$. In terms of the ``superfields'' (\ref{superfields A,B}), the symplectic form is simply $\omega = \langle \delta A , \delta B \rangle$. The corresponding map $\omega^\#: F\xra{\sim} F^*[-1]$ is the identity.
}

One introduces the {\it action}:
\be S=\langle B, \dd A \rangle \quad \in \mr{Fun}(F)=\left(S^\bt (F_{even})^*\right) \otimes \left(\wedge^\bt (F_{odd})^*\right) \label{superfields A,B}\ee
where the {\it superfields} $$A:\; C_\bt[1]\ra C_\bt,\quad B:\; C^\bt[-2]\ra C^\bt $$ are just shifted identity maps on chains/cochains, extended to the whole $F$ via decomposition (\ref{F=chains+cochains}). The action $S$ is a degree $0$ quadratic function on $F$.

Assume that we have a chain inclusion $\iota: H_\bt \ra C_\bt$, a chain projection $p: C_\bt \ra H_\bt$ which agrees with the canonical projection of cycles to homology: $p|_Z: Z_\bt \ra H_\bt$ (sending a cycle to its class in homology) and satisfies $p\circ \iota=\mr{id}_{H_\bt}$. Assume also that $\kappa: C_\bt\ra C_{\bt+1}$ is a chain homotopy satisfying
$$\dd\kappa+\kappa\dd=\mr{id}_{C_\bt}-\iota\circ p,\quad \kappa^2=0,\quad \kappa\circ\iota=0, \quad p\circ \kappa=0$$

The triple of maps $(\iota,p,\kappa)$ determines a ``Hodge decomposition'' of chains:
\be C_\bt=\rlap{$\overbrace{\phantom{\iota(H_\bt)\oplus B_\bt}}^{Z_\bt}$}\iota(H_\bt)\oplus \underbrace{B_\bt\oplus \mr{im}(\kappa)}_{\ker p} \label{Hodge decomp from ind data}\ee

If the chain spaces are endowed with positive definite inner product $(,)$, one can induce the triple $(\iota,p,\kappa)$ from the Hodge decomposition (\ref{Hodge decomp of Euclidean chain complex}). One sets $\iota$ to be the realization of homology classes by their harmonic representatives, $p$ subtracts from a chain its $\dd^*$-exact part and takes the homology class of the resulting cycle. The chain homotopy is constructed as
$$\kappa=\dd^*\circ (\Delta+\lambda\cdot\iota\circ p)^{-1}$$
where $\lambda\in\bR-0$ is an arbitrary nonzero number. 

Decomposition (\ref{Hodge decomp from ind data}) induces a decomposition of $F$ into a direct sum of symplectic subspaces
$$F=F'\oplus F''$$
where
$$F'=\iota(H_\bt)[1]\oplus \iota(H_\bt)^*[-2],\quad F''=(\ker p)[1]\oplus (\ker p)^*[-2]$$
and a Lagrangian subspace\footnote{Lagrangian in $F''$, not in the whole $F$.}
$$L''=\mr{im}(\kappa)[1]\oplus (B_\bt)^*[-2]\quad \subset F''$$

Now we would like to say that $\bT(\mu)$ is given by the following fiber Batalin-Vilkovisky integral:
\be \bT(\mu)=\beta\int_{L''\subset F''}\mu\cdot e^{i S}\label{fiber BV int}\ee
with $\beta=(2\pi)^{-\dim B_{even}}i^{-\dim B_{odd}}$.
Let us explain how the right hand side of (\ref{fiber BV int}) is understood.

By analogy with Definition \ref{def: density}, for a vector space $V$ and $\alpha\in\bR$ a number, one can introduce $\alpha$-densities on $V$ as functions $f$ from the set of bases $\{v\}$ on $V$ to $\bR_+$ such that
$$f(m\cdot v)=|{\det}(m)|^\alpha\cdot f(v)$$
for any $m\in GL(V)$. Denote $\mr{Dens}^\alpha(V)$ the set of $\alpha$-densities on $V$.

For $V_\bt$ a graded vector space,
an $\alpha$-density $f$ on $V_\bt$ associates to a choice of basis $v_i$ in each $V_i$ a number $f(\{v_i\})\in\bR$ in such a way that for a collection of automorphisms $m_i\in GL(V_i)$, one has $f(\{m_i\cdot v_i\})=\prod_i |\det m_i|^{(-1)^i\alpha}\cdot f(\{v_i\})$.

Note that $\mr{Dens}^1(V_\bt)\cong \Det(V_\bt[1])/\{\pm1\}$ can be viewed as the set of affine Berezin integration measures on $V$ (regarded as a superspace with parity prescribed according to the $\bZ$-grading reduced modulo $2$). For 1-densities we will suppress the superscript in $\Dens^1(V)=\Dens(V)$.

Fact \cite{Severa}: for $V\supset L$ a pair of an odd-symplectic graded vector space and a Lagrangian subspace, one has a canonical isomorphism (depending on the odd-symplectic structure on $V$)
$$\Dens^{1/2}(V)\cong \Dens(L)$$
(which comes from the symplectomorphism $V\simeq L\oplus L^*[-1]$ and the relation of determinant lines
$\Det(L\oplus L^*[-1])\cong \Det(L)^{\otimes 2}$).

Now, the r.h.s. of (\ref{fiber BV int}) is to be understood as the image of $\mu$ with respect to the following diagram:\footnote{Alternatively, one can altogether avoid using half-densities and work with determinant lines only, using the fact that for $L\subset V$ a Lagrangian subspace of an odd-symplectic graded vector space, there is a canonical quadratic map $\Det(L)\xra{\sim}\Det(V)\cong \Det(L)^{\otimes 2}$, sending $\nu\mapsto \nu^{\otimes 2}$. Then we have $\mu\in\Det(C_\bt)\cong \Det(C_\bt[1])^*\xra{\bt^{\otimes 2}} \Det(F)^*\simeq \Det(F')^*\otimes \Det(F'')^*\xra{\bt|_{H[1]}^{\otimes\frac{1}{2}}\otimes \bt|_L^{\otimes\frac{1}{2}}} \Det(H_\bt[1])^*\otimes \Det(L)^*\xra{\mr{id}\otimes\int e^{iS}\cdot \bt} \Det(H_\bt)\otimes\bC\ni \beta^{-1}\bT(\mu)$.}
\begin{multline*}
\Dens(C_\bt[1])\simeq \Dens^{1/2}(F)\simeq \Dens^{1/2}(F')\otimes \Dens^{1/2}(F'')\simeq \\ \simeq
\Dens(H_\bt[1])\otimes \Dens(L'')
\xra{\mr{id}\otimes\int_{L''}e^{iS}\cdot\bt} \Dens(H_\bt[1])\otimes\bC
\end{multline*}
Here in the last arrow we calculate the Berezin integral over the odd part of $L''$ and the ordinary (Lebesgue) integral over the even part of $L''$.

Consistency of (\ref{fiber BV int}) with the construction of $\bT$ given in the proof of Lemma \ref{Lm: Det(C)=Det(H)} can be checked by a direct calculation of the integral in the r.h.s. of (\ref{fiber BV int}) in coordinates on chain spaces adapted to the Hodge decomposition (\ref{Hodge decomp from ind data}).

\lec{Lecture 7, 03.04.2014}

\section{Whitehead torsion for CW-complexes}

\subsection{CW-complexes: reminder}
(Reference: \cite{Cohen}).
\subsubsection{Homotopy, a reminder of basic definitions}
A pair of maps\footnote{``Maps'' will always mean ``continuous maps''.} $f,g:X\ra Y$ between topological spaces are {\it homotopic} if there exists a map $F: X\times [0,1]\ra Y$ such that $F(x,0)=f(x)$ and $F(x,1)=g(x)$ for any $x\in X$.

A map $f:X\ra Y$ is called a {\it homotopy equivalence} of spaces $X$ and $Y$ if there exists a map ({\it homotopy inverse}) $g: Y\ra X$, such that $g\circ f: X\ra X$ and $f\circ g: Y\ra Y$ are homotopic to identity on $X$ and $Y$ respectively.

A particularly nice kind of homotopy equivalence is a deformation retraction.

If $Y\subset X$, then $D: X\ra Y$ is a {\it strong deformation retraction} of $X$ onto $Y$ if there exists a map $F:X\times [0,1]\ra X$ such that
$$F(x,0)=x, \quad F(x,1)=D(x), \quad F(y,t)=y$$
for any $x\in X$, $y\in Y$, $t\in [0,1]$. In this case, $Y$ is called a {\it deformation retract} of $X$.

\subsubsection{CW-complexes}
\begin{definition}
A {\it CW-complex} (={\it cell complex}) is a Hausdorff topological space $X$ and a collection of 
subsets $e_\alpha\subset X$ ({\it cells}) with the following properties.
\begin{enumerate}[(i)]
\item \label{CW complex def (i)} $X=\cup_\alpha e_\alpha$ and $e_\alpha\cap e_\beta=\varnothing$ for $\alpha\neq \beta$.
\item \label{CW complex def (ii)} Each cell $e_\alpha$ of is equipped with a {\it characteristic map} $\phi_\alpha: B^k\ra X$ where $B_k$ is a closed ball of dimension $k\geq 0$ -- the {\it dimension} of the cell, such that
    \begin{itemize}
    \item on the interior of the ball, $\phi_\alpha|_{\mr{int}(B^k)}$ is a homeomorphism onto $e_\alpha\subset X$,
    \item the boundary of the ball is sent to $(k-1)$-skeleton of $X$: $\phi_\alpha(\dd B^k)\subset \sk_{k-1}X$
    where one defines {\it skeleta} as $\sk_k X=\bigcup_{\alpha:\, \dim(e_\alpha)\leq k} e_\alpha$.
    \end{itemize}
\item \label{CW complex def (iii)} Each closure $\bar e_\alpha$ is contained in the union of finitely many cells.
\item \label{CW complex def (iv)} A set $Y\subset X$ is closed in $X$ iff $Y\cap \bar e_\alpha$ is closed in $\bar e_\alpha$ for all $e_\alpha$.
\end{enumerate}
The map $\phi_\alpha|_{\dd B^k}$ is called the {\it attaching map} for the cell $e_\alpha$.
\end{definition}

Skeleta form an increasing filtration of $X$,
$$\sk_0 X\subset \sk_1 X\subset \sk_2 X \cdots\subset X,\qquad X=\bigcup_k \sk_k X$$

Note that if $X$ has finitely many cells, properties (\ref{CW complex def (iii)},\ref{CW complex def (iv)}) are satisfied automatically.

When we need to distinguish the underlying topological space of a CW-complex from the CW-complex $X$ itself, we will denote the former $|X|$.

Union of a subcollection of cells $Y=\cup_\beta e_\beta$, $\{\beta\}\subset \{\alpha\}$ is called a ``subcomplex'' of $X$ if $\bar e_\beta\subset Y$ for every $e_\beta$.

A map $f:X\ra Y$ between CW-complexes is called {\it cellular} if $f(\sk_k(X))\subset \sk_k(Y)$ for every $k$.

If $f,g: X\ra Y$ are two homotopic maps and $g$ is cellular, then $g$ is called a {\it cellular approximation} to $f$.
\begin{thm}[Cellular approximation theorem]
Any map between CW-complexes $f: X\ra Y$ is homotopic to a cellular map.
\end{thm}

\subsubsection{Cellular homology}
{\it Chain complex} of a CW-complex $X$ is defined as a complex of free abelian groups arising as relative singular homology of pairs of consecutive skeleta:
$$C_k(X) = H_k(\sk_k X, \sk_{k-1}X)$$
The boundary map $C_k(X)\ra C_{k-1}(X)$ is constructed as the connecting homomorphism in the long exact sequence in homology of the triple $\sk_{k-2}\subset \sk_{k-1}\subset \sk_k$:
\begin{multline*}\underbrace{H_k(\sk_{k-1},\sk_{k-2})}_0\ra H_k(\sk_k,\sk_{k-2})\ra \underbrace{H_k(\sk_k,\sk_{k-1})}_{C_k(X)} \xra{\dd} \\
\xra{\dd} \underbrace{H_{k-1}(\sk_{k-1},\sk_{k-2})}_{C_{k-1}(X)}\ra H_{k-1}(\sk_{k},\sk_{k-2})\ra \underbrace{H_{k-1}(\sk_{k},\sk_{k-1})}_0
\end{multline*}

Each chain group $C_k(X)$ is a free abelian group with one generator for every cell of dimension $k$.

\begin{thm}
Homology of the complex $C_k(X)$ ({\it cellular homology} of $X$) is canonically isomorphic to the singular homology of the underlying topological space of $X$.
\end{thm}

For a pair $(X,Y)$ of a CW-complex $X$ and a subcomplex $Y\subset X$, the {\it chain complex of the pair} is defined via singular homology of pairs as
$$C_k(X,Y)=H_k(\sk_k X\cup |Y|, \sk_{k-1}\cup |Y|)$$
Chains $C_k(X,Y)$ form a free abelian group generated by cells of $X-Y$ of dimension $k$.
For the cellular homology of a pair one again has a canonical isomorphism with the singular homology,
$$H_k(X,Y)\cong H_k(|X|,|Y|)$$

\subsection{Whitehead torsion of a pair}
(Reference: \cite{Milnor66}).
Let $(X,Y)$ be a pair consisting of a finite connected CW-complex $X$ and a subcomplex $Y\subset X$, such that $Y$ is a deformation retract of $X$. Denote $\pi_1:=\pi_1(X)$ (of course, we have $\pi_1(X)\cong \pi_1(Y)$).

Let $\widetilde X \supset \widetilde Y$ be the universal covering complexes for $(X,Y)$. The group $\pi_1$ acts on $\widetilde X, \widetilde Y$ by covering transformations, i.e. an element $\sigma\in \pi_1$ determines a cellular
mapping $\sigma:(\widetilde X, \widetilde Y)\ra (\widetilde X, \widetilde Y)$.

Consider the chain complex of the pair $(\widetilde X,\widetilde Y)$:
$$C_n(\widetilde X,\widetilde Y)\ra C_{n-1}(\widetilde X,\widetilde Y)\ra\cdots\ra C_0(\widetilde X,\widetilde Y)$$
Each $C_i$ is a free abelian group generated by $i$-cells of $\widetilde X-\widetilde Y$. This complex is acyclic since cellular homology of a pair is isomorphic to the singular homology (of the underlying pair of topological spaces), and $H_\bt(|\widetilde X|,|\widetilde Y|)=0$ since $|\widetilde Y|$ is a deformation retract of $|\widetilde X|$.

Since $\pi_1$ acts freely on cells of $\widetilde X-\widetilde Y$ by covering transformations, cellular chains $C_\bt(\widetilde X,\widetilde Y)$ comprise in fact a complex of free $\bZ[\pi_1]$-modules. A basis in $C_i(\widetilde X,\widetilde Y)$, as a free $\bZ[\pi_1]$-module, is constructed by choosing some liftings $\widetilde e^{(i)}\subset \widetilde X-\widetilde Y$ of all $i$-cells $e^{(i)}$ of $X-Y$.

\begin{definition} For a pair of finite $CW$-complexes $(X,Y)$, with $Y$ a deformation retract of $X$, the Whitehead torsion
$$\tau(X,Y)\in Wh(\pi_1)$$
is the image of the torsion $\tau(C_\bt(\widetilde X, \widetilde Y))\in \bar K_1(\bZ[\pi_1])$ in the Whitehead group $Wh(\pi_1)=\bar K_1(\bZ[\pi_1])/\mr{im}(\pi_1)$. The torsion is evaluated with respect to the basis of cells of $X-Y$ lifted to the universal covering, $\{\widetilde e_i\}$.
\end{definition}

Choosing a different lifting for a single cell, $\bar e\mapsto \bar e'=\sigma\cdot \bar e$ for some $\sigma\in\pi_1$, results in the change of the torsion $\tau(C_\bt(\widetilde X,\widetilde Y))\in \bar K_1$ by the class of the matrix $\mr{diag}(1,\ldots,1,\sigma,1,\ldots)$ in $\bar K_1$, i.e. by the image of $\sigma$ in $\bar K_1$. But in projecting to $Wh(\pi_1)$ we are killing exactly such images. Hence, the Whitehead torsion of a pair, valued in $Wh(\pi_1)$, is well-defined independently of the choice of lifting of the cells of $X-Y$.

\begin{rem} Since inner automorphisms of $\pi_1$ induce identity in $Wh(\pi_1)$, the Whitehead torsion is independent of the choice of base point for $\pi_1$.\footnote{More precisely, the isomorphism $s:\pi_1(X,p)\xra{\sim} \pi_1(X,p')$ for $p,p'\in X$ two base points, depends on a choice of path $\gamma\subset X$ from $p$ to $p'$ up to homotopy.
However, on the level of Whitehead groups, $s_*: Wh(\pi_1(X,p))\xra{\sim} Wh(\pi_1(X,p'))$ is a unique isomorphism, independent on $\gamma$. Whitehead torsions of a pair $(X,Y)$ defined using base points $p,p'$ differ by application of $s_*$. 
}
\end{rem}

\begin{definition} Given two CW-complexes $X$ and $X'$ with the same underlying topological space $|X|=|X'|$, one says that $X'$ is a {\it subdivision} of $X$ if every cell of $X'$ is contained in a (possibly, higher dimensional) cell of $X$, so that the identity map $X\ra X'$ is {\it cellular}. Similarly, a pair of CW-complexes $(X',Y')$ is a subdivision of the pair $(X,Y)$ if $X'$ is a subdivision of $X$ and $Y'$ is a subdivision of $Y$.
\end{definition}

\begin{thm}\label{thm: invariance of Whitehead torsion wrt subdivision}
The Whitehead torsion $\tau(X,Y)$ is invariant under subdivision of the pair $(X,Y)$.
\end{thm}

\begin{lemma} \label{Lm: torsion of a pair with X-Y simply connected}
Let $(X,Y)$ be a pair of CW-complexes where $Y\subset X$ is a deformation retract and $X-Y$ is connected and simply connected. Then $\tau(X,Y)=1$.
\end{lemma}

\begin{proof}\footnote{Reference: \cite{Milnor66}, p.379.}
Denote $\Gamma=X-Y$ and choose a representative component $\widetilde \Gamma$ of 
$\widetilde X-\widetilde Y$. Clearly $\widetilde \Gamma$ projects homeomorphically to $\Gamma$. For each cell $e\subset \Gamma$ choose the representative cell $\widetilde e$ as the unique cell of  $\widetilde\Gamma$ that project to $e$. Notice that no representative cell $\widetilde e$ can be incident to a proper translate $\sigma\cdot \widetilde e$, $\sigma\neq 1$ (since $\sigma\cdot \widetilde e$ lies inside $\sigma\cdot\widetilde\Gamma$ which is disjoint from $\widetilde\Gamma$). Thus $\dd \widetilde e$ is a linear combination of representative cells with {\it integer} coefficients (as opposed to coefficients in $\bZ[\pi_1]$). Thus in calculating the torsion of the complex $C_\bt(\widetilde X,\widetilde Y)$ we are working with the subring $\bZ\subset \bZ[\pi_1]$. It follows that the torsion belongs to the subgroup
$$\{1\}=\bar K_1(\bZ)\subset \bar K_1(\bZ[\pi_1])$$
Therefore $\tau(X,Y)=1$.
\end{proof}

\begin{lemma} \label{Lm: multiplicativity of Whitehead torsions of pairs}
If $X\supset Y\supset Z$ is a triple of CW-complexes such that $Y$ is a deformation retract of $X$ and $Z$ is a deformation retract of $Y$, we have the following relation for torsions:
$$\tau(X,Z)=\tau(X,Y)\cdot \tau(Y,Z)$$
\end{lemma}

\begin{proof} Follows immediately from applying Lemma \ref{Lm: multiplicativity of torsions} to the short exact sequence of complexes of free $\bZ[\pi_1]$-modules
$$0\ra C_\bt(\widetilde Y,\widetilde Z)\ra C_\bt(\widetilde X,\widetilde Z)\ra C_\bt(\widetilde X,\widetilde Y)\ra 0$$
\end{proof}

\begin{lemma}\label{Lm: subdivision of one cell}
Let $(X,Y)$ be a pair of CW-complexes with $Y\subset X$ a deformation retract. Let $X'$ be a subdivision of $X$ where only one cell $e\subset X$ is subdivided (in particular, cells bounding $e$ are not subdivided). Then we have the following.
\begin{enumerate}[(i)]
\item \label{Lm: subdivision of one cell (i)} If $e\subset X-Y$, then
\be \tau(X',Y)=\tau(X,Y) \label{Lm: subdivision of one cell eq1}\ee
\item \label{Lm: subdivision of one cell (ii)} If $e\subset Y$ then
\be \tau(X',Y')=\tau(X,Y) \label{Lm: subdivision of one cell eq2}\ee
where $Y'=|Y|\cap X'$ is the induced subdivision of $Y$.
\end{enumerate}
\end{lemma}

\begin{proof}
Firs assume that $e\subset X-Y$.
Let $n$ be the dimension of the subdivided cell $e$ and denote $E\subset X'$ the subcomplex representing the subdivision of the cell $e\subset X$, $|E|=|e|$. Construct a new CW-complex $X''$ by adding two new cells $\epsilon^{(n+1)}$ and $\epsilon^{(n)}$ to $X'$ in such a way that
$$\dd \epsilon^{(n+1)}=\epsilon^{(n)}\cup E$$
Note that $X''$ contains a subcomplex $\tilde X$ where $e$ is replaced by $\epsilon^{(n)}$ (obviously $\tilde X$ is isomorphic to $X$ as a CW-complex). Note also that pairs $(X'',\tilde X)$ and $(X'',X')$ both satisfy the assumptions of Lemma \ref{Lm: torsion of a pair with X-Y simply connected}. Hence,
$$\tau(X'',\tilde X)=\tau(X'',X')=1$$
Now, using this and Lemma \ref{Lm: multiplicativity of Whitehead torsions of pairs}, we calculate
$$\tau(X,Y)=\tau(\tilde X,Y)=\underbrace{\tau(X'',\tilde X)^{-1}}_{1}\cdot\underbrace{\tau(X'',\tilde X)\cdot \tau(\tilde X,Y)}_{\tau(X'',Y)}=\underbrace{\tau(X'',X')}_{1}\cdot \tau(X',Y)$$
This proves (\ref{Lm: subdivision of one cell (i)}).
If $e\subset Y$, we immediately have (\ref{Lm: subdivision of one cell eq2}) since we have a canonical isomorphism of chains $C_\bt(\widetilde X',\widetilde Y')\cong C_\bt(\widetilde X,\widetilde Y)$ as based chain complexes.
\end{proof}


\begin{proof}[Proof of Theorem \ref{thm: invariance of Whitehead torsion wrt subdivision}]
Let $(X,Y)$ be a pair of CW-complexes with $Y\subset X$ a deformation retract and $(X',Y')$ a subdivision of $(X,Y)$. Choose some ordering $e_1,e_2,\cdot ,c_N$ of cells of $X$ such that $\dim e_i\leq \dim e_{i+1}$. Construct a sequence of complexes
$$X_k=\bigcup_{i\leq k} (|e_i|\cap X')\;\cup \bigcup_{i>k} e_i,\qquad k=0,1,\ldots,N$$
Note that each $X_{k}$ is obtained by subdividing a single cell $e_k$ of $X_{k-1}$.
Denote $Y_k=|Y|\cap X_k$ the induced subdivision of $Y$ (we have $Y_k=Y_{k-1}$ if the cell $e_k$ lies outside $Y$; otherwise $Y_{k}$ is a single cell subdivision of $Y_{k-1}$). Each pair $(X_k,Y_k)$ is a subdivision of a single cell of the pair $(X_{k-1},Y_{k-1})$, thus Lemma \ref{Lm: subdivision of one cell} applies and we obtain $\tau(X_k,Y_k)=\tau(X_{k-1},Y_{k-1})$. Observing that $(X_0,Y_0)=(X,Y)$ is the original pair and $(X_N,Y_N)=(X',Y')$ is the subdivided pair, we obtain
$$\tau(X',Y')=\tau(X,Y)$$
\end{proof}

\lec{Lecture 8, 10.04.2014}

\subsection{Whitehead torsion of a map}
\begin{definition} The {\it mapping cylinder} of a cellular mapping $f:X\ra Y$ is the CW-complex $M_f$ with underlying topological space
$$|M_f|=\frac{|X|\times [0,1]\quad\sqcup |Y|}{(x,1)\sim f(x)}$$
Cell structure on $M_f$ is induced from cell structures on $X$ and $Y$. In particular, $M_f$ contains $X=X\times\{0\}$ and $Y$ as disjoint subcomplexes.
\end{definition}

For any $f:X\ra Y$, $Y$ is a deformation retract of $M_f$ (with the retraction given by ``following the rays''\footnote{The formula for the retraction onto $Y$ is: $D: (x,s)\mapsto f(x),\; y\mapsto y$ and $D_t: (x,s)\mapsto (x,1-(1-s)(1-t)),\; y\mapsto y$ is a homotopy between the identity on $M_f$ and $D: M_f\ra Y$.}).

\begin{lemma}\label{Lm: tau(M_f,Y)}
For any $f:X\ra Y$, we have
$$\tau(M_f,Y)=1$$
\end{lemma}

\begin{proof} Let $f_p: \sk_p X\ra Y$ be the restriction of $f$ to the $p$-skeleton of $X$. We have an increasing filtration of mapping cylinders:
$$Y=M_{f_{-1}}\subset M_{f_0}\subset M_{f_1}\subset\cdots\subset M_{f_n}=M_f$$
By repeatedly using Lemma \ref{Lm: multiplicativity of Whitehead torsions of pairs}, we have
$$\tau(M_f,Y)=\prod_p \tau(M_{f_p},M_{f_{p-1}})$$
and by Lemma \ref{Lm: torsion of a pair with X-Y simply connected} each torsion in the product on the right hand side is trivial.
\end{proof}

If $f:X\ra Y$ is a homotopy equivalence, then $X=X\times\{0\}\subset M_f$ is too a deformation retract of $M_f$. \footnote{See e.g. \cite{Hatcher}, Corollary 0.21, p.16.}
However the torsion $\tau(M_f,X)$ is not always zero.

\begin{definition}
One defines the Whitehead torsion of a cellular homotopy equivalence $f: X\ra Y$ as the torsion of the pair
$$\tau(f):=\tau(M_f,X)\quad \in Wh(\pi_1)$$
\end{definition}

Here are some properties of Whitehead torsions of homotopy equivalences.
\begin{prop} \label{prop: tau of maps properties}
\begin{enumerate}[(i)]
\item \label{prop: tau of maps properties i}
For $\iota: X\ra Y$ an inclusion,
\be \tau(\iota)=\tau(Y,X) \label{tau(i)}\ee

\item  \label{prop: tau of maps properties ii}
For $f,g:X\ra Y$ two homotopic maps, one has
\be \tau(f)=\tau(g) \label{tau(f)=tau(g)}\ee

\item \label{prop: tau of maps properties iii}
Given two maps $f:X\ra Y$, $g:Y\ra Z$, one has
\be \tau(g\circ f)=\tau(g)\cdot\tau(f)\label{tau(gf)=tau(g)tau(f)}\ee
\end{enumerate}
(All maps $\iota,f,g$ above are assumed to be cellular homotopy equivalences.)
\end{prop}

\begin{proof} Property (\ref{prop: tau of maps properties i}) follows from applying Lemma \ref{Lm: multiplicativity of Whitehead torsions of pairs} to the triple
$$X\subset X\times [0,1] \subset M_f$$
and the fact that $\tau(\mr{id}_X)=1$ (following e.g. from Lemma \ref{Lm: tau(M_f,Y)}).

Properties (\ref{prop: tau of maps properties ii},\ref{prop: tau of maps properties iii}) follow as special cases from a more general one: for a triple of cellular homotopy equivalences
$$f: X\ra Y,\quad g: Y\ra Z,\quad h: X\ra Z$$
such that $g\circ f$ is homotopic to $h$, one has
\be \tau(g\circ f)=\tau(g)\cdot \tau(f) \label{prop: tau of maps properties tau(h)=tau(g)tau(f)}\ee
To prove this, choose a cellular\footnote{This can be done by the cellular approximation theorem.} homotopy between  $h$ and $g\circ f $, i.e. a cellular map
$$F: X\times [0,1]\ra Z$$
such that $F|_{X\times\{0\}}=h$, $F|_{X\times \{1\}}=g\circ f$. Map $F$ can be pieced together with the map $g: Y\ra Z$ to yield a map $H:M_f\ra Z$ such that
$$H_{X\times\{0\}\subset M_f}=h,\quad H_{Y\subset M_f}=g$$
Next we construct the mapping cylinder $M_H$ which contains mapping cylinders $M_f$, $M_g$, $M_h$ as CW-subcomplexes. More precisely, one has the following diagram of inclusions.
\be \label{prop: tau of maps properties CD}
\begin{CD}
X @>\tau(f)>> M_f @<\tau=1<< Y \\
@V\tau(h)VV @VViV @VV\tau(g)V \\
M_h @>k>> M_H @<j<< M_g \\
@A\tau=1AA @A\tau=1AA @A\tau=1AA \\
Z @= Z @= Z
\end{CD}
\ee
where we know by Lemma \ref{Lm: tau(M_f,Y)} that torsions of inclusions of $Z$ into the three mapping cylinders and the torsion of $Y\hra M_f$ are trivial. Next, applying Lemma \ref{Lm: multiplicativity of Whitehead torsions of pairs} to the triple $M_H\supset M_g\supset Z$ we infer that $\tau(j: M_g\hra M_H)=1$. Likewise, $\tau(k:M_h\hra M_H)=1$. Next, applying Lemma \ref{Lm: multiplicativity of Whitehead torsions of pairs} to the upper right square of the diagram (\ref{prop: tau of maps properties CD}), we get
$$\tau(M_H,M_f)\cdot 1= 1\cdot \tau(M_g,Y)$$
From the upper left square we get
$$\tau(M_H,M_f)\cdot \tau(M_f,X) = 1\cdot\tau(M_h,X)$$
Combining these two equations together, we have
$$\underbrace{\tau(M_g,Y)}_{\tau(g)}\cdot \underbrace{\tau(M_f,X)}_{\tau(f)}=\underbrace{\tau(M_h,X)}_{\tau(h)}$$
This finishes the proof of (\ref{prop: tau of maps properties tau(h)=tau(g)tau(f)}) and immediately implies (\ref{prop: tau of maps properties ii},\ref{prop: tau of maps properties iii}).
\end{proof}

\begin{rem}
Property (\ref{prop: tau of maps properties ii}) implies that one can also define $\tau(f)$ for a homotopy equivalence $f:X\ra Y$ which is not cellular. (One first approximates it by a cellular one $\tilde f$ and defines $\tau(f):=\tau(\tilde f)$, and this definition is independent on the particular choice of cellular approximation $\tilde f$ by virtue of (\ref{prop: tau of maps properties ii})).
\end{rem}

\subsection{Whitehead torsion and simple homotopy type}
\begin{definition}
A CW-complex $X'$ is called an {\it elementary expansion} of a CW-complex $X$ if $X'=X\cup e^{(n)}\cup e^{n-1}$ with $e^{(n)},e^{(n-1)}$ two new cells of dimensions $n,\, n-1$ respectively, such that $\dd e^{(n)}=e^{(n-1)}\cup \dd_X e^{(n)}$ with $\dd_X e^{(n)}\subset \sk_{n-1}X$.
\end{definition}
If $X'$ is an elementary expansion of $X$, one also says that $X$ is an {\it elementary collapse} of $X'$.
If
\be X=X_0\ra X_1\ra \cdots \ra X_{N-1}\ra X_N=X' \label{simple-homotopy}\ee
is a finite sequence of CW-complexes such that for each $k$, $X_{k+1}$ is either an elementary expansion or an elementary collapse of $X_k$, one says that complexes $X$ and $X'$ have same {\it simple homotopy type}.

Realizing every elementary expansion $X_k\ra X_{k+1}$ by the natural inclusion map and every elementary collapse $X_k\ra X_{k+1}$ by any cellular retraction, we obtain an actual cellular map from $X$ to $X'$ in (\ref{simple-homotopy}), the {\it deformation}.\footnote{This is the terminology of \cite{Cohen}. Elementary maps $X_k\ra X_{k+1}$ (either inclusions or cellular retractions) are called {\it formal deformations}.}

A map $f:X\ra Y$ homotopic to a deformation is called a {\it simple homotopy equivalence}.

We will say that a pair of CW-complexes $(X,Y)$ with $Y$ a deformation retract of $X$ is simple homotopic to the pair $(X',Y)$ (relative to Y), if one can bring $X$ to $X'$ by a sequence of elementary expansions/collapses, without ever collapsing cells of $Y$.

\begin{thm}\label{thm: simple homotopy pairs}
\begin{enumerate}[(i)]
\item \label{thm: simple homotopy pairs i} A pair of CW-complexes $(X,Y)$ is simple homotopic to $(X',Y)$ if and only if $\tau(X,Y)=\tau(X',Y)$.
\item \label{thm: simple homotopy pairs ii}For a CW-complex $X$ and any element of the Whitehead group $\alpha\in Wh(\pi_1(X))$, there exists a pair $(Y,X)$ such that $\tau(Y,X)=\alpha$.
\end{enumerate}
\end{thm}

Note that the ``only if'' part of (\ref{thm: simple homotopy pairs i}) follows immediately from Lemmas \ref{Lm: torsion of a pair with X-Y simply connected} and \ref{Lm: multiplicativity of Whitehead torsions of pairs}.

\begin{corollary}
Whitehead group $Wh(\pi_1(X))$ is canonically isomorphic to the set of equivalence classes of pairs of CW-complexes $(Y,X)$ with $Y$ retractable onto $X$, modulo simple homotopy of pairs (relative to $X$).\footnote{In \cite{Cohen} this is the {\it definition} of the Whitehead group of $X$.}
\end{corollary}

The following is an important special case of (\ref{thm: simple homotopy pairs i}).
\begin{corollary}
A complex $X$ is simple homotopic to $Y$ relative to $Y$ if and only if $\tau(X,Y)=1$.
\end{corollary}

The version of the theorem above for maps is the following.
\begin{thm}\label{thm: simple homotopy maps}
\begin{enumerate}[(i)]
\item \label{thm: simple homotopy maps i} A homotopy equivalence of CW-complexes $f:X\ra Y$ is a simple homotopy equivalence if and only if $\tau(f)=1 \in Wh(\pi_1(X))$.\footnote{In \cite{Milnor66}, simple homotopy equivalence is {\it defined} as a homotopy equivalence $f$ with $\tau(f)=1$; this theorem establishes the equivalence of definitions.}
\item \label{thm: simple homotopy maps ii} For a fixed $X$ and $\alpha\in Wh(\pi_1(X))$, one can find a CW-complex $Y$ and a homotopy equivalence $f:X\ra Y$ such that $\tau(f)=\alpha$.
\end{enumerate}
\end{thm}

The ``only if'' part of (\ref{thm: simple homotopy maps i}) again easily follows from (\ref{prop: tau of maps properties iii}) and (\ref{prop: tau of maps properties ii}) of Proposition \ref{prop: tau of maps properties} and the observation that for $\iota:X\hra X'$ an inclusion associated to an elementary expansion, $\tau(\iota)=1$ by Lemma \ref{Lm: torsion of a pair with X-Y simply connected} and (\ref{tau(i)}). For $p:X'\ra X$ a retraction associated to an elementary collapse, one has $p\circ\iota=\mr{id}_X$, hence $\tau(p)\tau(\iota)=1$, and thus $\tau(p)=1$.

\begin{corollary} If $X$ is simply connected (or more generally $Wh(\pi_1(X))=\{1\}$), then any homotopy equivalence $f:X\ra Y$ is a simple homotopy equivalence.
\end{corollary}

\lec{Lecture 9, 16.04.2014}

\subsubsection{Proof of (\ref{thm: simple homotopy pairs ii}) of Theorem \ref{thm: simple homotopy pairs} (realization property).} (Reference: \cite{Turaev01}, Proposition 7.1, p.37.) Let $X$ be a finite CW-complex with $x_0$ a chosen vertex (which we use as a base point for the fundamental group $\pi_1=\pi_1(X,x_0)$) and let $a=(a_{ij})$ be an invertible $k\times k$ matrix with entries in $\bZ[\pi_1]$. Let us also fix an integer $n\geq 2$.

Construct a CW-complex
$$X'=X\vee \bigvee_{j=1}^k S^{n}_j$$
by wedging $X$ with $k$ copies of an $n$-sphere, $S^{n}_1,\ldots,S^{n}_k$ at $x_0\in X$. Since $n\geq 2$, we have $\pi_1(X',x_0)\simeq \pi_1(X,x_0)$. This implies that the homotopy group $\pi_n(X',x_0)$ is acted on by $\pi_1$ (by sweeping $n$-spheres along loops) and hence has a structure of a $\bZ[\pi_1]$-module.
Denoting the class of $S^{n}_j$ in $\pi_n(X',x_0)$ by $[S^{n}_j]$, consider a collection of classes
$$\alpha_i=\sum_j a_{ij}[S^{n}_j]\quad \in \pi_n(X',x_0)$$
and let $f_i: S^n\ra X'$ be some map representing this class. We construct a new CW-complex $Y$ by attaching $k$ new $(n+1)$-balls to $X'$ along maps $f_i: \dd B^{n+1}_i\simeq S^n\ra X'$, i.e.
$$Y=\frac{X'\cup \bigcup_i B_i^{n+1}}{\dd B_i^{n+1}\ni y\sim f_i(y)}$$
Thus $Y-X$ consists of $k$ $n$-cells $e^{(n)}_j=S^n_j-x_0\subset Y$ and $k$ $(n+1)$-cells $e^{(n+1)}_i=\mr{int} B^{n+1}_i\subset Y$.

Passing to the universal covers $\widetilde X$, $\widetilde Y$, we have the relative chain complex concentrated in degrees $n+1$ and $n$:
$$C_\bt(\widetilde Y, \widetilde X)=\qquad 0\ra\cdots \ra 0\ra \underbrace{\bZ[\pi_1]^k}_{\mr{Span}_{\bZ[\pi_1]}\{\widetilde e^{(n+1)}_i\}}\xra{\dd} \underbrace{\bZ[\pi_1]^k}_{\mr{Span}_{\bZ[\pi_1]}\{\widetilde e^{(n)}_i\}}\ra 0\ra\cdots\ra 0$$

The matrix of the differential $\dd$ (with a natural choice of liftings of cells) is exactly $a=(a_{ij})\in GL(k,\bZ[\pi_1])$. Acyclicity of $C_\bt$ is equivalent to the assumption that $a$ is invertible. Moreover, $X$ is a deformation retract of $Y$.\footnote{This follows from the homotopy extension property for CW-complexes, see \cite{Cohen}, (3.2), p.6.}
Hence, the torsion $\tau(Y,X)=\tau(C_\bt(\widetilde Y,\widetilde X))$ is the class $[a]^{(-1)^{(n+1)}}\in Wh(\pi_1)$. Thus, using our freedom to choose any $a$, we can construct a pair $(Y,X)$ realizing any element of $Wh(\pi_1)$.

\section{Reidemeister torsion}

\subsection{Change of rings}
Let $C_\bt$ be a complex of free left $A$-modules, with each $C_k$ endowed with a preferred basis $c_k$. Let $h:A\ra A'$ be a ring homomorphism. We construct a new complex $C'_\bt$ of free  $A'$-modules by setting
$$C'_k=A'\otimes_A C_k$$
where the right $A$-module-structure on $A'$ is given by $h$: $a'\cdot a:=a'h(a)$. Furthermore, the basis $$c_{k1},\ldots,c_{kn}$$ in $C_k$ induces a basis
$$1\otimes c_{k1},\ldots, 1\otimes c_{kn}$$
in $C'_k$.

Assuming that $C'_\bt$ is acyclic, one can define the torsion $\tau(C'_\bt)\in \bar K_1(A')$. It is called the torsion of $C_\bt$ associated with the representation $h$ of $A$. We will use the notation $\tau_h(C_\bt):=\tau(C'_\bt)$.

If $C_\bt$ itself is acyclic, then $C'_\bt$ is acyclic automatically. In this case, one can define both torsions $\tau(C_\bt)$ and $\tau_h(C_\bt)$ which satisfy the relation
$$\tau_h(C_\bt)=h_*\tau(C_\bt)$$
(Thus $\tau_h(C_\bt)$ does not contain any new information.)
The case of interest is when $C'_\bt$ is acyclic whereas $C_\bt$ is not.

\subsection{Torsion with coefficients in $\bC$}
The simplest example is to take $A'=\bC$ --  the field of complex numbers.

Consider a pair of finite CW-complexes $(X,Y)$ without assuming that $Y$ is a deformation retract of $X$. Note that relative chains of the universal cover $C_\bt(\widetilde X,\widetilde Y)$ still form a complex of free modules over $\bZ[\pi_1(X)]$. (Now we should be specific about which $\pi_1$ we take; also, $\widetilde Y$ here is the preimage  of $Y\subset X$ in $\widetilde X$, rather than the universal cover of $Y$.)

Choosing a group homomorphism $h: \pi_1(X)\ra \bC^*$, we can extend it to a ring homomorphism
$$h: \bZ[\pi_1(X)]\ra \bC$$
Now we can form the complex of {\it $h$-twisted} chains
$$C'_\bt=\bC\otimes_{\bZ[\pi_1(X)]}C_\bt(\widetilde X,\widetilde Y)$$
-- a complex of $\bC$-vector spaces.

If homology of $C'_\bt$ vanishes, the torsion $$\tau(C'_\bt)\in \bar K_1(\bC)=\bC^*/\{\pm 1\}$$ is well-defined up to ambiguity arising from the possibility to use different liftings of cells of $X-Y$ to the universal cover for the preferred basis. It is cancelled, as in the case of Whitehead torsion, by passing to the further group quotient:
$$\bar K_1(\bC)/\mr{image}(\pi_1(X))=\bC^*/\;\pm h(\pi_1(X))$$

Thus, the Reidemeister-Franz torsion $\tau_h(X,Y)$ of a pair $(X,Y)$ associated to a homomorphism $h: \pi_1(X)\ra\bC^*$ is defined as the torsion $\tau(C'_\bt)\in \;\bC^*/\;\pm h(\pi_1(X))$.

The torsion $\tau_h(X,Y)$ is invariant with respect to subdivisions of $(X,Y)$ (when defined).

{\bf Example.} $X=S^1$ a circle (with a cell decomposition consisting of a single $0$-cell $e^{(0)}$ and single $1$-cell $e^{(1)}$), $Y=\varnothing$. Denote by $\sigma$ the class of the loop going around $S^1$ once (for convenience, we choose an orientation on $S^1$) in $\pi(S^1)$. Under the isomorphism $\pi_1(S^1)\sim \bZ$, $\sigma$ is identified with the unit $1\in\bZ$.
A homomorphism $h: \pi_1(S^1)=\bZ\ra \bC^*$ is determined by its value on $\sigma$,
$h(\sigma)=t\in\bC^*$.

Choosing some liftings $\widetilde e^{(0)}$, $\widetilde e^{(1)}$ of the cells $e^{(0)},e^{(1)}\subset X$ to the universal covering complex $\widetilde X$, the boundary map in $C_\bt(\widetilde X,\widetilde Y)=C_\bt(\widetilde X)$ is:
$$\dd: \widetilde e^{(1)}\mapsto \sigma^{N+1}\widetilde e^{(0)}-  \sigma^{N}\widetilde e^{(0)}$$
for some $N\in \bZ$ dependent on the particular choice of liftings.

For the complex of $h$-twisted chains, we have $C'_0$ and $C'_1$ $1$-dimensional complex vector spaces with basis vectors $1\otimes \widetilde e^{(0)}$, $1\otimes \widetilde e^{(1)}$. The boundary map sends
$$\dd: 1\otimes \widetilde e^{(1)}\mapsto (t^{N+1}-t^N)\; 1\otimes \widetilde e^{(0)}$$
In particular, the complex $C'_\bt$ is acyclic iff $t\neq 1$ (recall that we also assume $t\neq 0$ from the start).
Therefore, the torsion is:
$$\tau(C'_\bt)=\left[1\otimes \widetilde e^{(1)}/1\otimes \widetilde e^{(0)}\right]_{C'_1}^{-1}=\left[\frac{\dd(1\otimes \widetilde e^{(1)})}{1\otimes \widetilde e^{(0)}}\right]_{C'_0}^{-1}=\left[\frac{1}{t-1}\right]\quad \in \bC^*/\{\pm t^j\}_{j\in\bZ}$$

Note that the original chain complex $C_\bt(\widetilde X)$ is not acyclic and the Whitehead torsion $\tau(S^1,\varnothing)$ is not defined.

{\bf Example}\footnote{Cf. \cite{Milnor62}.} For $K\subset S^3$ an oriented knot in the 3-sphere, take $X$ to be the complement of a tubular neighborhood of $K$ in $S^3$. Choose a representation $h:\pi_1(X)\ra \bC^*$ by sending a loop $\gamma$ to $t^{\mr{lk}(\gamma,K)}$ where $t\in \bC^*$ is a fixed number and $\mr{lk}(,)$ is the linking number of two oriented knots.
Then the Reidemeister torsion is
\be \tau_h(X)=\left[\frac{A_K(t)}{t-1}\right]\quad \in \bC^*/\{\pm t^j\}_{j\in\bZ} \label{R-torsion of knot complement}\ee
where $A_K(t)$ is the Alexander polynomial of the knot $K$.\footnote{
Recall that for any knot $K\subset S^3$, one has $H_1(X)\simeq\bZ$ and the abelianization map $\pi_1(X)\ra \pi_1(X)/[\pi_1(X),\pi_1(X)]\cong H_1(X)\simeq \bZ$ is given by linking numbers with loops in the knot complement $\mr{lk}(K,\bt):\pi_1(X)\ra \bZ$. Let $\tilde X$ be the infinite cyclic (= maximal abelian) cover of $X$, with $\bZ$ acting by covering transformations. Denote by $\sigma$ the generator of $H_1(X)\simeq \bZ$. For any $\bZ$-equivariant cell decomposition of $\tilde X$, $C_\bt(\tilde X)$ is a chain complex of free module over $\bZ[\bZ]=\bZ[\sigma,\sigma^{-1}]$ (the ring of Laurent polynomials in $\sigma$). The homology $H_1(\tilde X)$ is a (non-free) $\bZ[\sigma,\sigma^{-1}]$-module (it is moreover a pure torsion module, i.e. consisting solely of torsion elements). For any presentation of $H_1(\tilde X)$ by $n$ generators $x_j$ with $m\geq n$ relations $\sum_j a_{ij}x_j=0$, one has the $m\times n$ presentation matrix $(a_{ij})$ with elements in $\bZ[\sigma,\sigma^{-1}]$.  The 0-th Fitting ideal of $H_1(\tilde X)$ is the ideal generated by all maximal ($n\times n$) minors of the presentation matrix (the ideal is independent of the particular presentation). Alexander polynomial $A_K(\sigma)\in \bZ[\sigma,\sigma^{-1}]$ is defined as a generator of the smallest principal ideal containing the 0-th Fitting ideal of $H_1(\tilde X)$. As such, it is defined up to multiplication by invertible elements of the ring $\bZ[\sigma,\sigma^{-1}]$, i.e. by $\pm \sigma^n$.
}\footnote{
Result (\ref{R-torsion of knot complement}) follows from the following general algebraic statement (for details and the proof see \cite{Turaev01}). For $A$ a Noetherian principal ideal domain, denote $\tilde A$ its field of fractions. Let $C_\bt$ be a chain complex of free $A$-modules and assume that homology $H_i(C)$ are $A$-modules of rank zero (i.e. $\tilde A\otimes_A H_i(C)=0$). Construct a new complex of vector spaces over $\tilde A$ as $C'_\bt=\tilde A\otimes_A C_\bt$. Then $C'$ is acyclic and its torsion is $\tau(C')=\prod_i (\mr{ord}\,H_i(C))^{(-1)^{i+1}}$ where the order $\mr{ord}\,M\in A$ of a rank zero $A$-module $M$ is defined as the greatest common divisor of (= generator of the smallest principal ideal containing) the 0-th Fitting ideal of $M$. Note that orders are defined modulo units of $A$. In (\ref{R-torsion of knot complement}) we set $A=\bZ[\sigma,\sigma^{-1}]$, $C_\bt=C_\bt(\tilde X)$ and map the field of fractions $A'$ to complex numbers by evaluating at $\sigma=t$.
}

\subsection{Real torsion} The following modification of the Reidemeister torsion is useful, especially in case of non-abelian $\pi_1$.
Let $h:\pi_1\ra O(n)$ be a representation of $\pi_1(X)$ by orthogonal matrices of some fixed size $n$.
Then $h$ uniquely extends to a ring homomorphism $\bZ[\pi_1]\ra M_n(\bR)$ to the ring of all real $n\times n$ matrices. Using $h$, we form the complex
$$C'_\bt=M_n(\bR)\otimes_{\bZ[\pi_1]} C_\bt(\widetilde X,\widetilde Y)$$
If the complex $C'_\bt$ is acyclic, one can define its torsion
$$\tau(C'_\bt)\in \bar K_1 (M_n(\bR))=\bar K_1(\bR)\cong \bR_+$$
where the identification of $\bar K_1(M_n(\bR))$ with $\bR_+$ is via the map $(a_{ij})\mapsto |\det(a_{ij})|$. Note that, since $\pi_1$ is represented by orthogonal matrices, which in particular have determinants $\pm 1$, the image  $\mr{im}(\pi_1)\subset \bar K_1(M_n(\bR))$ is trivial and one does not have to make an additional quotient to account for ambiguity in the choice of liftings of cells.

\begin{definition} The positive real number
$\tau_h(X,Y)=\tau(C'_\bt)\in \bR_+$ is called the {\it $R$-torsion} of the pair $(X,Y)$, associated to the orthogonal representation $h$.
\end{definition}

The $R$-torsion is also invariant with respect to subdivisions.

Representation $h$ defines a homomorphism $h_*: Wh(\pi_1)\ra \bar K_1(\bR)\cong \bR_+$. In case when the Whitehead torsion $\tau(X,Y)\in Wh(\pi_1)$ is defined, the $R$-torsion is its image under $h_*$,
$$\tau_h(X,Y)=h_* \tau(X,Y)$$

\begin{thm}[Bass]
Suppose the Whitehead torsion $\tau(X,Y)$ is defined and the group $\pi_1$ is finite. Then $\tau(X,Y)$ is an element of finite order in $Wh(\pi_1)$ iff $\tau_h(X,Y)=1$ for all possible orthogonal representations $h$ of $\pi_1$. If $\pi_1$ is finite abelian, then $\tau(X,Y)=1$ iff $\tau_h(X,Y)=1$ for all possible $h$.
\end{thm}

Another relation between Whitehead torsion and $R$-torsion is the following. We denote $\tau_h(X)=\tau_h(X,\varnothing)$.
If $f:X\ra Y$ is a homotopy equivalence, then\footnote{
Indeed, one has $\tau_h(Y,\varnothing)=\tau_h(M_f,\varnothing)=\tau_h(M_f,X)\tau_h(X,\varnothing)=h_*\tau(M_f,X)\tau_h(X)= h_*\tau(f)\tau_h(X)$. We have used the multiplicativity property of $R$-torsions $\tau_h(Z,X)=\tau_h(Z,Y)\tau_h(Y,X)$ where $Y$ is a retract of $Z$ (but not asking that $X$ is a retract of $Y$). It is proved similarly to Lemma \ref{Lm: multiplicativity of Whitehead torsions of pairs}.
}
$$\tau_h(Y)=h_*(f)\tau_h(X)$$

In particular, if
$f:X\ra Y$ is a homotopy equivalence and $\tau_h(X)\neq \tau_h(Y)$ for some $h$, then $\tau(f)\neq 1$ and $f$ is not a simple homotopy equivalence.

{\bf Exercise.}
\begin{enumerate}
\item Fix a pair $(X,Y)$ and  let $h:\pi_1\ra U(1)\subset \bC$ be a (complex) unitary representation of rank 1 and let
$\tilde h: \pi_1\ra O(2)\simeq U(1)$ be the corresponding representation in $2\times 2$ (real) orthogonal matrices. Show that the corresponding complex (Reidemeister) torsion and the $R$-torsion are related by
$$\tau_{\tilde h}(X,Y)=|\tau_h(X,Y)|^2$$
\item Show that the torsion of $M_n(\bR)\otimes_{\bZ[\pi_1]}C_\bt(\widetilde X,\widetilde Y)$ as a complex of free $M_n(\bR)$-modules is equal to the torsion of $\bR^n\otimes_{\bZ[\pi_1]}C_\bt(\widetilde X,\widetilde Y)$ as a complex of real vector spaces (for the basis we take the cellular basis tensored with any orthonormal basis in $\bR^n$).
\end{enumerate}

\lec{Lecture 10, 06.05.2014}

\section{Lens spaces}

\begin{definition} For $n,p\in \bN$ and $q_1,\ldots,q_n\in \bZ$ coprime with $p$, the lens space $L(p;q_1,\ldots,q_n)$ is defined as
$$L(p;q_1,\ldots,q_n)=\frac{\{(z_1,\ldots,z_n)\in \bC^n\,|\, \sum_k |z_k|^2=1\}}{(z_1,\ldots,z_k)\sim (\zeta^{q_1} z_1,\ldots,\zeta^{q_n}z_n)}$$
where $\zeta=e^{\frac{2\pi i}{p}}$ is the $p$-th primitive root of unity.
\end{definition}

The lens space $L(p;q_1,\ldots,q_n)$ is a quotient of the sphere $S^{2n-1}$ by a free\footnote{Note that freeness of the $\bZ_p$-action on $S^{2n-1}$ is equivalent to $q_1,\ldots,q_n$ being coprime with $p$.} action of $\bZ_p$ and thus is a smooth oriented manifold. Moreover, the lens space inherits a Riemannian metric of constant sectional curvature $+1$ from the round metric on $S^{2n-1}$. By construction it also is equipped with a preferred generator of the fundamental group $\pi_1(L)=\bZ_p$ -- the class of any path connecting a fixed point $(z_1^0,\ldots,z_n^0)$ with $(\zeta^{q_1}z_1^0,\ldots \zeta^{q_n}z_n^0)$ in the quotient $L=S^{2n-1}/\bZ_p$. We will denote this generator $\sigma$.

Obviously, $L(p;q_1,\ldots,q_n)$ only depends on the residues $q_1,\ldots,q_n\mod p$.

\subsection{Trivial homeomorphisms and classification up to simple homotopy equivalence}
\label{sec: lens spaces, trivial maps}
\begin{enumerate}[I]
\item For $\pi$ a permutation of $1,\ldots,n$, the map
$$\bC^n\ra \bC^n,\quad (z_1,\ldots,z_n)\mapsto (z_{\pi(1)},\ldots,z_{\pi(n)})$$
induces a map of lens spaces
$$\Phi_\pi^{I}:\quad L(p;q_1,\ldots,q_n)\ra L(p;q_{\pi(1)},\ldots, q_{\pi(n)})$$
It is an orientation preserving diffeomorphism, which also preserves the preferred generator of $\pi_1$.
\item For each $k=1,\ldots,n$, the map $(z_1,\ldots,z_k,\ldots,z_n)\mapsto (z_1,\ldots,\bar z_k,\ldots,z_n)$
(complex conjugation of $k$-th coordinate) induces a map of lens spaces
$$\Phi_k^{II}:\quad L(p;q_1,\ldots,q_k,\ldots,q_n)\ra L(p;q_1,\ldots,-q_k,\ldots,q_n)$$
where we use that $\overline{\zeta^{q_k}z_k}=\zeta^{-q_k}\bar z_k$.
The map $\Phi^{II}$ is an orientation reversing diffeomorphism preserving the preferred generator of $\pi_1$.
\item For $m\in \bZ_p$ coprime with $p$, the identity map on $\bC^n$ induces a map
$$\Phi^{III}_m:\quad \underbrace{L(p;q_1,\ldots,q_n)}_L\ra \underbrace{L(p;m q_1,\ldots, m q_n)}_{L'}$$
It is an orientation preserving diffeomorphism, but maps the generator $\sigma\in \pi_1(L)$ to $(\sigma')^r\in \pi_1(L')$ where $r$ is the inverse residue for $m$ defined by $mr\equiv 1\mod p$.
\end{enumerate}

Note that all the maps of lens spaces above are actually isometries of Riemannian manifolds.

Using maps $\Phi^I$, $\Phi^{II}$ any lens space\footnote{Except the real projective space $L(2;1,\ldots,1)=\bR P^n$ and the sphere $L(1;1,\ldots,1)=S^{2n-1}$.} can be brought to a unique {\it special} lens space $L(p;q_1,\ldots,q_n)$ with $$1\leq q_1\leq q_2\leq \cdots \leq q_n< p/2$$

\textbf{Terminology:}\cite{Turaev01} One calls a simple homotopy equivalence of lens spaces preserving the preferred generator of $\pi_1$ the {\it esh equivalence} (simple homotopy equivalence of {\it enriched} spaces).

\begin{thm} \label{thm: lens space sh classification}
\begin{enumerate}[(i)]
\item \label{thm: lens space sh classification (i)} Two lens spaces $L=L(p;q_1,\ldots,q_n)$ and $L'=L(p;q'_1,\ldots,q'_n)$ are simple homotopy equivalent as spaces with preferred generator of $\pi_1$ iff
$$q'_k\equiv \pm q_{\pi(k)}\mod p,\qquad\qquad k=1,\ldots,n$$
for some permutation $\pi$.
\item \label{thm: lens space sh classification (ii)}
$L$ and $L'$ are simple homotopy equivalent (ignoring the preferred generator of $\pi_1$) iff
$$q'_k\equiv \pm m\cdot q_{\pi(k)}\mod p,\qquad\qquad k=1,\ldots,n$$
for some permutation $\pi$ and some residue $m\in \bZ_p$ coprime with $p$.
\end{enumerate}
\end{thm}

Note that using maps $\Phi^{III}$ one can convert a simple homotopy equivalence to one preserving the generator of $\pi_1$. Thus  (\ref{thm: lens space sh classification (ii)}) immediately follows from (\ref{thm: lens space sh classification (i)}). The ``if'' part of (\ref{thm: lens space sh classification (i)}) is obvious: the simple homotopy equivalence is the composition of maps $\Phi^I$ and $\Phi^{II}$. The ``only if'' part relies on the explicit calculation of Reidemeister torsion for lens spaces. We will finish the proof of Theorem \ref{thm: lens space sh classification} in Section \ref{sec: lens spaces, end of proof of sh classification}.

The following is immediate from Theorem \ref{thm: lens space sh classification}.
\begin{corollary}
 A lens space is esh equivalent to a unique special lens space. In particular, two special lens spaces $L(p;q_1,\ldots,q_n)$ and $L(p;q'_1,\ldots,q'_n)$ are esh equivalent iff $q_k\equiv q'_k\mod p$ for all $k$.
\end{corollary}

\begin{corollary}
All simple homotopy equivalences between lens spaces can be represented by isometries.
\end{corollary}

\begin{rem} Theorem \ref{thm: lens space sh classification} gives in fact the classification of lens spaces up to homeomorphism. If a pair of lens spaces are simple homotopy equivalent, they are indeed homeomorphic by the discussion above. On the other hand if they are not simple homotopy equivalent, they are distinguished by one of the torsion invariants (cf. Section \ref{sec: lens spaces, end of proof of sh classification}). Chapman's theorem asserts that the Whitehead torsion is homeomorphism invariant, which implies that for $L,L'$ homeomorphic, Reidemester torsions of $L$ and $L'$ coincide. Thus non-simple homotopic lens spaces are not homeomorphic.
\end{rem}

Here is a more general result due to de Rham. For $G\subset SO(n+1)$ a finite group acting on the unit sphere $S^n$ by orthogonal rotations without fixed points, one calls $X=S^n/G$ a {\it spherical Clifford-Klein manifold} \cite{Milnor66}.

\begin{thm}[de Rham]
A spherical Clifford-Klein manifold is determined uniquely up to isometry by its fundamental group $G$ and the collection of torsion invariants $\tau_h(X)$.
\end{thm}

\subsection{Cellular structure}

Let us introduce a $\bZ_p$-equivariant cell decomposition of the unit circle $S^1=\{z\in \bC\;|\; |z|=1\}$ with $0$-cells given by $\zeta^k$ and $1$-cells $I_k=[\zeta^k,\zeta^{k+1}]$ for $k=0,\ldots, p-1$.

Next we introduce a $\bZ_p$-equivariant cell decomposition of $S^{2n-1}$ with closed cells defined as follows. For $i=1,\ldots,n$ and $k\in \bZ_p$ we set
\begin{multline*}
E^{2i-2}_k=\{(z_1,\ldots,z_n)\in S^{2n-1}\;|\; z_{i+1}=\cdots=z_n=0,\; z_i\in \zeta^k\cdot [0,1]\subset \bC\}\\
=\{(z_1,\ldots,z_n)\in S^{2n-1}\;|\; \sum_{j=1}^{i-1}|z_j|^2=1-|z_i|^2,\; z_i\in \zeta^k\cdot [0,1]\subset \bC \}
\end{multline*}
and
\begin{multline*}
E^{2i-1}_k=\{(z_1,\ldots,z_n)\in S^{2n-1}\;|\; z_{i+1}=\cdots=z_n=0,\; z_i\in I_k\cdot [0,1]\subset \bC\}\\
=\{(z_1,\ldots,z_n)\in S^{2n-1}\;|\; \sum_{j=1}^{i-1}|z_j|^2=1-|z_i|^2,\; z_i\in I_k\cdot [0,1]\subset \bC \}
\end{multline*}

E.g. for $n=2$, one has
\begin{eqnarray*}
E^0_k &=& (\zeta^k,0) \in S^3 \\
E^1_k &=& (I_k,0) \subset S^3 \\
E^2_k &=& \{(z_1,t\zeta^k)\;|\; 0\leq t\leq 1,\; |z_1|^2=1-t^2\} \\
E^3_k &=& \{(z_1,z_2)\in S^3 \; |\; z_2\in I_j\cdot [0,1]\subset \bC\}
\end{eqnarray*}

Closed balls $(E^{2i-2}_k,E^{2i-1}_k)_{1\leq i\leq n,\, k\in \bZ_p}$ form a $\bZ_p$-equivariant CW-decomposition of $S^{2n-1}$ with $p$ cells in each dimension $0,1,\ldots, 2n-1$ (which in turn induces a cell decomposition of any lens space $L(p;\cdots)$ with a {\it single} cell in every dimension $0,1,\ldots, 2n-1$). Topological boundaries of cells are:
\begin{eqnarray*}\dd E^{2i-2}_k &=& E^{2i-3}_0\cup\cdots\cup E^{2i-3}_{p-1} \\
\dd E^{2i-1}_k &=& E^{2i-2}_{k}\cup E^{2i-2}_{k+1}
\end{eqnarray*}
Odd dimensional skeleta of this cell complex are ``equatorial'' spheres
$$\mr{Sk}_{2i-1}S^{2n-1}=S^{2i-1}=\{(z_1,\ldots,z_i,0,\cdots,0)\in \bC^n\;|\; \sum_{j=1}^i |z_i|^2=1\}$$

Denoting $e^{2i-2}_k$, $e^{2i-1}_k$ the corresponding open cells of $S^{2n-1}$, the boundary map in cellular chains (with the natural orientation of cells) is:
$$\dd e^{2i-2}_k= e^{2i-3}_0+\cdots+e^{2i-3}_{p-1},\quad \dd e^{2i-1}_k= e^{2i-2}_{k+1}-e^{2i-2}_k$$

The distinguished generator $\sigma\in \pi_1$ acts on cells by
$$\sigma:\quad e^{2i-2}_k\mapsto e^{2i-2}_{k+q_i},\quad e^{2i-1}_k\mapsto e^{2i-1}_{k+q_i}$$

\subsubsection{Equivariant cell decomposition of $S^{2n-1}$ via the join of polygons}
\textbf{Reminder: joins.} For $X$, $Y$ topological spaces the {\it join} $X\ast Y$ is defined as
$$X\ast Y=\frac{X\times Y\times [0,1]}{(x,y,0)\sim (x,y',0),\; (x,y,1)\sim (x',y,1)\;\forall x,x'\in X,\; y,y'\in Y}$$
Some properties of joins:
\begin{itemize}
\item $X\ast \mathrm{pt}$ is a cone over $X$.
\item $X\ast S^0$ is the suspension $S X$.
\item For spheres $S^m$, $S^n$, the join $S^m\ast S^n$ is homeomorphic to $S^{m+n+1}$.
\item One has natural inclusions $X\hra X\ast Y\hookleftarrow Y$ where $X$ is included as $\{(x,\bt,0)\;|\; x\in X\}\subset X\ast Y$, and similarly for $Y$.
\end{itemize}

Realizing $S^{2n-1}$ as an $n$-fold join of circles
$$S^{2n-1}=\underbrace{S^1\ast\cdots\ast S^1}_n$$
one can induce from the $\bZ_p$-equivariant cell decomposition of $S^1$ (with cells $\{\zeta^k,I_k\}_{k\in\bZ_p}$) a $\bZ_p$-equivariant cell decomposition of $S^{2n-1}$ with cells
$$E^{2i-2}_k=\underbrace{S^1\ast\cdots\ast S^1}_{i-1}\ast \zeta^k\subset S^{2n-1},\qquad E^{2i-1}_k=\underbrace{S^1\ast\cdots\ast S^1}_{i-1}\ast I_k\subset S^{2n-1}$$

\subsection{Homology and Reidemeister torsion}
For a lens space $L(p;q_1,\ldots,q_n)$, the group ring $\bZ[\pi_1]=\bZ[\bZ_p]$ can be identified with the ring of polynomials in $\sigma$ with integer coefficients, modulo relation $\sigma^p=1$. Denote $\nu=1+\sigma+\sigma^2+\cdots+\sigma^{p-1}\in \bZ[\sigma]/(\sigma^p-1)$.
The chain complex of $S^{2n-1}$ with the $\bZ_p$-equivariant cell decomposition, regarded as a complex of free $\bZ[\pi_1]$-modules, is:
\begin{multline}\label{lens space chains of cover}
0\ra \mr{Span}(e_0^{2n-1})\xra{\cdot (\sigma^{r_n}-1)}\mr{Span}(e_0^{2n-2})\xra{\cdot \nu}
\mr{Span}(e_0^{2n-3})\xra{\cdot (\sigma^{r_{n-1}}-1)}
\cdots \\
\cdots\xra{\cdot\nu}\mr{Span}(e_0^1) \xra{\cdot (\sigma^{r_1}-1)} \mr{Span}(e_0^0) \ra 0
\end{multline}
Here $r_k\in\bZ_p$ is the inverse residue to $q_k$, i.e. $q_k r_k=1\mod p$.

Chains of the lens space with integral coefficients are obtained from (\ref{lens space chains of cover}) by the ring change $\bZ[\pi_1]\ra \bZ$ with $\sigma\mapsto 1$:
$$0\ra \mr{Span}_\bZ(e_0^{2n-1})\xra{0}\mr{Span}_\bZ(e_0^{2n-2})\xra{\cdot p}
\mr{Span}_\bZ(e_0^{2n-3})\xra{0}
\cdots \\
\cdots\xra{\cdot p}\mr{Span}_\bZ(e_0^1) \xra{0} \mr{Span}_\bZ(e_0^0)\ra 0
$$
Here in each degree we have a rank 1 free module over $\bZ$. The homology is:
$$H_0=H_{2n-1}=\bZ,\qquad H_{2i-1}=\bZ_p,\quad H_{2i}=0\quad\mbox{for}\;\; i=1,\ldots,n-1$$

For the Reidemeister torsion, choose a group homomorphism $h:\pi_1\ra \bC^*$ by choosing a root of unity $\eta=h(\sigma)\in\bC$, $\eta^p=1$. For $\eta\neq 1$ the twisted chain complex is acyclic\footnote{In case $\eta=1$ we have the complex $0\ra\bC\xra{0}\bC\xra{\cdot p}\cdots\xra{\cdot p}\bC\xra{0}\bC\ra 0$. Its homology is $H_0=H_{2n-1}=\bC$, $H_{i\not\in\{0,2n-1\}}=0$.} and has the form
$$\bC\otimes_{\bZ[\pi_1]}C_\bt(S^{2n-1},\bZ):\qquad \bC\xra{\cdot (\eta^{r_n}-1)}\bC\xra{0}\cdots\xra{0} \bC\xra{\cdot (\eta^{r_1}-1)}\bC$$
Its torsion is
\be \tau_\eta(L)=\prod_{i=1}^{n}(\eta^{r_i}-1)^{-1}\quad \in \bC^*/\{\pm \eta^j\}_{j\in \bZ_p} \label{lens space torsion}\ee
-- the Reidemeister torsion of the lens space $L(p;q_1,\ldots,q_n)$.

\subsubsection{End of proof of Theorem \ref{thm: lens space sh classification}}\label{sec: lens spaces, end of proof of sh classification}
The fact that Reidemeister torsion distinguishes between lens spaces not related by trivial homeomorphisms of Section \ref{sec: lens spaces, trivial maps}, is based on the following number theoretic result.

\begin{lemma}[Franz's independence lemma]\label{Lm: Franz lemma}
Denote $S\subset \bZ_p$ the group of residues coprime with $p$. Suppose $\{a_j\}_{j\in S}$ is a set of integers satisfying
\begin{enumerate}
\item $\sum_{j\in S} a_j = 0$,
\item $a_j = a_{-j}$,
\item $\prod_{j\in S} (\eta^j-1)^{a_j}=1$ for every $p$-th root of unity $\eta$.
\end{enumerate}
Then $a_j=0$ for all $j\in S$.
\end{lemma}

\begin{proof}[Proof of the ``only if'' part of Theorem \ref{thm: lens space sh classification}]
It suffices to prove that two different {\it special} lens spaces $L=L(p;q_1,\ldots,q_n)$, $L'=L(p;q'_1,\ldots,q'_n)$ at least for some root of unity $\eta$, the torsions (\ref{lens space torsion}) $\tau_\eta(L)\neq \tau_\eta(L')$ are different. Suppose the converse:
\be\prod_{i=1}^n(\eta^{r_i}-1)^{-1}=\pm\eta^k\prod_{i=1}^n(\eta^{r'_i}-1)\label{thm: lens space sh classification equal torsions}\ee
for any $p$-th root of unity $\eta$ and some $k\in \bZ_p$ depending on $\eta$. Denote
$$s_j=\#\{1\leq i\leq n\;|\; r_i=j\},\quad s'_j=\#\{1\leq i\leq n\;|\; r'_i=j\}$$
for $j\in S$. Taking the square of the absolute value in (\ref{thm: lens space sh classification equal torsions}) and dividing r.h.s. by l.h.s. we obtain
$$\prod_{j\in S}(\eta^j-1)^{s_j+s_{-j}-s'_j-s'_{-j}}=1$$
for all admissible $\eta$.  Thus we have a collection of integers $a_j=s_j+s_{-j}-s'_j-s'_{-j}$ satisfying all the assumptions of Lemma \ref{Lm: Franz lemma}. Hence $a_j=0$ for all $j$. Take $j$ an inverse residue to $1\leq l< p/2$, $l$ coprime with $p$. Then $s_{-j}=s'_{-j}=0$ since $L,L'$ were special, and hence $s_j=s'_j$. Thus $s_j=s'_j$ for all $j$. Therefore $L=L'$ and we came to a contradiction.
\end{proof}

\lec{Lecture 11, 08.05.2014}

\subsection{Maps between lens spaces and the homotopy classification}
\begin{lemma}\label{Lm: lens spaces homotopy Lm1}
Let $f,g: L(p;q_1,\ldots,q_n)\ra L(p;q'_1,\ldots,q'_n)$ be two maps between lens spaces 
which induce the same map of fundamental groups $f_*=g_*:\pi_1(L)\ra \pi_1(L')$.
Then $\deg f\equiv \deg g\mod p$.
If moreover $\deg f=\deg g$ (as integers, not residues $\mod p$), then $f$ is homotopic to $g$.
\end{lemma}
\begin{proof}[Sketch of proof] (For more details, see \cite{Cohen}, (29.3), p.92.)
\begin{enumerate}
\item First, lift $f,g$ to a pair of $\bZ_p$-equivariant maps $\tilde f,\tilde g: S^{2n-1}\ra S^{2n-1}$ sending $\sigma$ to $(\sigma')^m$ for some $m$.
\item Define a  $2n$-dimensional CW-complex $P=S^{2n-1}\times [0,1]$ with cell structure coming from the equivariant cell decomposition of the sphere $S^{2n-1}$. Denote $P_j=\mr{sk}_j P\;\cup\; (S^{2n-1}\times \{0,1\})$.
We construct inductively a sequence of $\bZ_p$-equivariant maps $F_j:P_j\ra S^{2n-1} $ such that $F_j|_{S^{2n-1}\times \{0\}}=\tilde f$, $F_j|_{S^{2n-1}\times \{1\}}=\tilde g$. For $j=0$, $F_0$ is given by $\tilde f\cup \tilde g: P_0=S^{2n-1}\times \{0,1\} \ra S^{2n-1}$. Assuming $F_{j}$ is defined, we can define $F_{j+1}: P_{j+1}\ra S^{2n-1}$ by extending $F_j: P_j\ra S^{2n-1}$ to the cell $E^j_0\times [0,1]\subset P_{j+1}$ from its boundary $\underbrace{\dd (E^j_0\times [0,1])}_{\simeq S^j}\subset P_j$. This extension is always possible for $j< 2n-1$ due to vanishing of homotopy groups $\pi_j(S^{2n-1})$. The map is extended to other top cells $E^j_k\times [0,1]$, $k\in \bZ_p$ of $P_{j+1}$ by imposing $\bZ_p$-equivariance. Thus we constructed an equivariant map $F_{2n-1}: P_{2n-1}\ra S^{2n-1}$ restricting to $\tilde f$, $\tilde g$ on $S^{2n-1}\times \{0,1\}\subset P_{2n-1}$.
\item Denote
$$\Phi_k=F_{2n-1}|_{\dd(E_k^{2n-1}\times [0,1])}: \quad \underbrace{\dd(E_k^{2n-1}\times [0,1])}_{\simeq S^{2n-1}} \ra S^{2n-1}$$
With appropriate choice of orientations of top cells we get, using $\bZ_p$-equivariance, a relation between degrees of maps
\be \underbrace{\sum_{k=1}^p \deg \Phi_k}_{p\cdot \deg \Phi_0}=\deg \tilde g-\deg\tilde f \label{deg Phi}\ee
This proves the first part of the lemma, that $\deg f \equiv \deg g\mod p$.
\item Assuming further that $\deg f=\deg g$, we obtain from (\ref{deg Phi}) that $\deg \Phi_k=0$, $k\in \bZ_p$. This implies, by Brouwer theorem, that the obstruction to extending the map $F_{2n-1}: P_{2n-1}\ra S^{2n-1}$ to the whole $P=P_{2n}$ vanishes. Hence we can construct $F_{2n}:P\ra S^{2n-1}$, which is an equivariant homotopy between $\tilde f$ and $\tilde g$ and thus induces a homotopy between $f$ and $g$ as maps of lens spaces.
\end{enumerate}
\end{proof}

\begin{lemma}\label{Lm: lens spaces homotopy Lm2} A map $f: L(p;q_1,\ldots,q_n)\ra L(p;q'_1,\ldots,q'_n)$ of degree $d$ preserving the preferred generator of $\pi_1$ exists iff
\be d\equiv r_1\cdots r_n\,q'_1\cdots q'_n\mod p \label{Lm: lens spaces homotopy Lm2 eq1}\ee
where $r_k$ are the inverse residues for $q_k$.
Some map of degree $d$ exists (ignoring the preferred generator of $\pi_1$) iff
\be d\equiv m^n\,r_1\cdots r_n\,q'_1\cdots q'_n\mod p \label{Lm: lens spaces homotopy Lm2 eq2}\ee
for some $m\in \bZ_p$ coprime with $p$ (and then this map sends $\sigma\in \pi_1(L)$ to $(\sigma')^m\in\pi_1(L')$).
\end{lemma}

\begin{proof} Fix $m\in \bZ$ coprime with $p$. The map
\be \phi:\bC^n\ra \bC^n,\qquad (z_1,\ldots, z_n)\mapsto (z_1^{m r_1 q_1'},\ldots, z_n^{m r_n q'_n}) \label{Lm: lens spaces homotopy Lm2 eq3}\ee
induces\footnote{
For $\iota: S^{2n-1}\hra \bC^n$ the inclusion of the unit sphere into $\bC^n$ and $\pi: \bC^n-0\ra S^{2n-1}$, $(z_1,\ldots,z_n)\mapsto \frac{1}{\sum_{k=1}^n |z_k|^2}(z_1,\ldots,z_n)$ the projection onto the sphere, we define $\tilde F_0:= \pi\circ \phi\circ \iota: S^{2n-1}\ra S^{2n-1}$.
} a $(\sigma,(\sigma')^m)$-equivariant map $\tilde F_0=S^{2n-1}\ra S^{2n-1}$ of degree $$\deg \tilde F_0=\prod_{k=1}^n (m\, r_k\, q'_k)$$
By equivariance it induces a map of lens spaces $F_0: L\ra L'$ of the same degree, sending $\sigma\in \pi_1(L)$ to $(\sigma')^m \in \pi_1(L')$.

Next, given an integer $N\in\bZ$ we construct a new $(\sigma,(\sigma')^m)$-equivariant map $\tilde F_N=S^{2n-1}\ra S^{2n-1}$ as follows. Fix a point in the top cell $x_0\in e^{2n-1}_0\subset S^{2n-1}$, a closed $(2n-1)$-ball $B_\epsilon\subset e^{2n-1}_0$ centered at $x_0$ which fits in the top cell and a closed $(2n-1)$-ball of twice smaller radius $B_{\epsilon/2}\subset B_\epsilon$, also centered at $x_0$. Denote $\phi: B_{\epsilon}-x_0 \ra B_{\epsilon}-B_{\epsilon/2}$ the homothety\footnote{
We mean the following: fix a homeomorphism $\chi$ between $B_\epsilon$ and the standard ball $\mathbf{B}_1$ of radius $1$ in Euclidean $\bR^{2n-1}$, such that $\chi$ maps $x_0$ to the origin and $B_{\epsilon/2}$ to the concentric Euclidean ball $\mathbf{B}_{1/2}\subset \mathbf{B}_1$ of radius $1/2$. Then define $\bar\phi: \mathbf{B}_1-0 \ra \mathbf{B}_1-\mathbf{B}_{1/2}$ by $r\cdot \vec{n}\mapsto \frac{1+r}{2}\cdot \vec{n}$ for $r\in (0,1]$ and $\vec{n}$ any vector of unit length. Then we set $\phi=\chi^{-1}\circ \bar\phi\circ\chi$.
} which fixes the bounding sphere $\dd B_\epsilon$. We set
\be \tilde F_N |_{e^{2n-1}_0}= \left\{\begin{array}{ll}\tilde F_0 & \mbox{ on } e^{2n-1}_0-B_\epsilon \\ \tilde F_0\circ \phi^{-1} & \mbox{ on } B_\epsilon-B_{\epsilon/2} \\
\Phi_N & \mbox{ on } B_{\epsilon/2} \end{array}\right. \label{lens spaces homotopy eq1}\ee
where $\Phi_N: B_{\epsilon/2}\ra S^{2n-1}$ is any degree $N$ map sending the bounding sphere $\dd B_{\epsilon/2}$ to $y_0=\tilde F_0(x_0)\in S^{2n-1}$. Next one extends (\ref{lens classification homotopy}) to the whole $S^{2n-1}$ by $(\sigma,(\sigma')^m)$-equivariance. The resulting map $\tilde F_N$ has degree
$$\deg \tilde F_N=\deg \tilde F_0+p\cdot N=m^n\,r_1\cdots r_n\,q'_1\cdots q'_n+p\cdot N$$

Thus we proved the existence of maps of degrees (\ref{Lm: lens spaces homotopy Lm2 eq1}) (for $m=1$) and (\ref{Lm: lens spaces homotopy Lm2 eq2}) (for $m$ general).

The ``only if'' part of (\ref{Lm: lens spaces homotopy Lm2 eq1},\ref{Lm: lens spaces homotopy Lm2 eq2}) follows immediately from Lemma \ref{Lm: lens spaces homotopy Lm1}.
\end{proof}

\begin{lemma} \label{Lm: lens spaces homotopy Lm3}
A map of lens spaces $f: L(p;q_1,\ldots,q_n)\ra L(p;q'_1,\ldots,q'_n)$ is a homotopy equivalence iff $\deg f=\pm 1$.
\end{lemma}

\begin{proof} The ``only if'' part is obvious. Let us prove the ``if'' part. Suppose $f$ sends $\sigma\in\pi_1(L)$ to $(\sigma')^m\in\pi_1(L')$ for some $m$. Then by Lemma \ref{Lm: lens spaces homotopy Lm2}, $q_1\cdots q_n\equiv\underbrace{\pm}_{\deg f} m^n q'_1\cdots  q'_n \mod p$. And thus there exists a map $g:L'\ra L$ of degree $\deg g=\deg f=\pm 1$ sending $(\sigma')^m\mapsto \sigma$. Compositions $g\circ f: L\ra L$ and $f\circ g:L'\ra L'$ are both maps of degree $1$ inducing identity on respective fundamental groups. Hence by Lemma \ref{Lm: lens spaces homotopy Lm1}, $g\circ f$ is homotopic to $\mr{id}_L$ and $f\circ g$ is homotopic to $\mr{id}_{L'}$. Thus $f$ is a homotopy equivalence.
\end{proof}

\begin{thm}[Classification of lens spaces up to homotopy]
There exists a homotopy equivalence of lens spaces $L(p;q_1,\ldots,q_n)$ and $L(p;q'_1,\ldots,q'_n)$ preserving the distinguished generator of $\pi_1$ iff
$$q'_1\cdots q'_n\equiv\pm q_1\cdots q_n\mod p$$
Ignoring the distinguished generator of $\pi_1$, some homotopy equivalence exists iff
$$m^n \, q'_1\cdots q'_n\equiv\pm  q_1\cdots q_n\mod p$$
for some $m\in \bZ_p$ coprime with $p$.
\end{thm}

\begin{proof} Follows immediately from Lemmas \ref{Lm: lens spaces homotopy Lm3} and \ref{Lm: lens spaces homotopy Lm2}.
\end{proof}

\begin{example} $L(5,1)=L(5,1,1)$ and $L(5,2)=L(5,1,2)$ are not homotopy equivalent, since $m^2\,1\cdot 2\not\equiv 1\cdot 1\mod 5$ for any $m\in \bZ_5$, since the only quadratic residues modulo $5$ are $\pm 1$.\end{example}

\begin{example}\label{example: L(7,1) vs L(7,2) homotopy}
$L(7,1)=L(7,1,1)$ and $L(7,2)=L(7,1,2)$ are different special lens spaces and are not simple homotopy equivalent by Theorem \ref{thm: lens space sh classification}.\footnote{
Here are the numeric values of the $O(2)$-torsion $|\tau_\eta|^2$: for $L(7,1)$, $|\tau_\eta|^2=1.763,\; 0.167,\; 0.069$ for $\eta=e^{\pm \frac{1}{7} 2\pi i}, e^{\pm \frac{2}{7} 2\pi i}, e^{\pm \frac{3}{7} 2\pi i}$ respectively. For the lens space $L(7,2)$, the values are $|\tau_\eta|^2=0.349,\;0.543,\;0.108$.
}
However $2^2\cdot 1\cdot 2\equiv 1\cdot 1\mod 7$, hence there exists a homotopy equivalence $f:L(7,1,1)\ra L(7,1,2)$ which is by Lemma \ref{Lm: lens spaces homotopy Lm2} a map of degree $1$ sending $\sigma$ to $(\sigma')^2$, since we have $d\equiv m^2\cdot 1\cdot 1\cdot 1\cdot 2\mod 7$ for $d=1$, $m=2$. To construct the map $f$, we first construct the $(\sigma,(\sigma')^2)$-equivariant map (\ref{Lm: lens spaces homotopy Lm2 eq3}) between spheres:
$$\tilde F_0: (z_1,z_2)\mapsto (z_1^2,z_2^4)$$
which has degree $8$. Next we correct it, as in (\ref{lens spaces homotopy eq1}), by a degree $-1$ map $\Phi_{-1}$ from a small 3-ball relative to its boundary to $S^3$ relative to a point, to an equivariant map $\tilde F_{-1}: S^3\ra S^3$ of degree $8-7\cdot 1=1$. It induces the degree map $f$ on the lens spaces, which is by Lemma \ref{Lm: lens spaces homotopy Lm3} a homotopy equivalence.
\end{example}

\begin{rem}
Note that for $p$ fixed there is a unique special lens space $L(p;\cdots)$ iff $p\in\{1,2,3,4,6\}$. Cf. the fact that $Wh(\bZ_p)$ is trivial exactly for those values of $p$.
\end{rem}

\begin{rem} Returning to the Example \ref{example: L(7,1) vs L(7,2) homotopy}, consider the Whitehead torsion of the homotopy equivalence between $L=L(7,1)$ and $L'=L(7,2)$, $\tau(f)\in Wh(\bZ_7)$. For $h:\pi_1(L)\ra \bC^*$ sending $\sigma$ to $\eta$, a $7$-th root of unity, for the diagram
$$
\begin{CD} \pi_1(L) @>\sigma\mapsto \eta>h> \bC^* \\
@V\sigma\mapsto (\sigma')^2Vf_*V @| \\
\pi_1(L') @>h'>\sigma'\mapsto \eta'> \bC^*
\end{CD}
$$
to commute, we need $(\eta')^2=\eta$. Hence, we set $\eta'=\eta^4$. Then we have
$$\tau_{\eta'}(L')=h_* \tau(f)\cdot \tau_{\eta}(L)$$
where $h_*: Wh(\bZ_7)\ra \bar K_1(\bC)/\mr{im}(\pi_1)=\bC^*/\{\pm \eta^j\}_{j\in\bZ_7}$.
Hence, by (\ref{lens space torsion}),
\be h_*\tau(f)=\frac{(\eta-1)^2}{(\eta'-1)((\eta')^4-1)}= \frac{(\eta-1)^2}{(\eta^4-1)(\eta^2-1)}\qquad \in\bC^*/\{\pm \eta^j\}_{j\in\bZ_7} \label{tosion of the map L(7,1) to L(7,2)}\ee
For $\eta=e^{\pm \frac{1}{7} 2\pi i}, e^{\pm \frac27 2\pi i}, e^{\pm \frac37 2\pi i}$, the respective numeric values of $|\tau_h(f)|^2$ are:
$$0.061,\quad 2.088,\quad 7.851$$
\end{rem}

\begin{rem}
Whitehead group $Wh(\bZ_7)$ is the quotient of the ring of units of $\bZ[\bZ_7]=\bZ[x]/(x^7-1)$ by trivial units $\{\pm x^j\}_{j\in \bZ_7}$.\footnote{We are using the nontrivial fact that $SK_1(\bZ[\bZ_p])=\{1\}$ for any $p$.} It is a free abelian group generated by two nontrivial units $u$ and $v$:\footnote{For the general explicit construction producing nontrivial units of $\bZ[\bZ_p]$, see \cite{Cohen} (11.4) p.44.}
\begin{multline*}
u=-1+x+x^{-1},\quad u^{-1}=1+x+x^{-1}-x^3-x^{-3},\\
v=-1+x^2+x^{-2},\quad v^{-1}=1-x-x^{-1}+x^2+x^{-2}
\end{multline*}
In particular, Whitehead torsion of the map $f: L(7,1)\ra L(7,2)$ of Example \ref{example: L(7,1) vs L(7,2) homotopy} has to be of the form $\tau(f)=u^k\cdot v^l$ for some integers $k,l$.
\end{rem}

{\bf Exercise.} Find the exponents $k,l\in \bZ$ from (\ref{tosion of the map L(7,1) to L(7,2)}).

\begin{rem}
For $p$ prime, $p\equiv 1\mod 4$, 3-dimensional lens spaces $L(p,q)$ fall into two classes up to homotopy, depending on whether $q$ is a quadratic residue or not (since $-1$ is a quadratic residue for $p\equiv 1\mod 4$). If $p$ is prime and $p\equiv -1\mod 4$, then all lens spaces $L(p,q)$ are homotopy equivalent. For $p$ non prime there may be more homotopy classes of lens spaces $L(p,q)$ (e.g. for $p=105=3\cdot 5\cdot 7$, there are $4$ classes).
\end{rem}

\section{Some applications and properties of torsions}

\subsection{Milnor's counterexample to Hauptvermutung}
Fix $n\geq 3$. For $q\in \bZ_7$, construct an $(n+3)$-dimensional CW-complex $X_q$ from $L(7,q)\times B^n$ (with $B^n$ an $n$-dimensional ball) by adjoining a cone over the boundary $L(7,q)\times \dd B^n$, with cell decomposition induced from the standard cell decomposition of lens spaces. Milnor \cite{MilnorHaupt} proves the following.

\begin{thm}[Milnor, 1962]\label{thm: Milnor Hauptvermutung}
\begin{enumerate}[(i)]
\item \label{thm: Milnor Hauptvermutung (i)} For $n+3\geq 6$ the complex $X_1$ is homeomorphic to $X_2$.
\item \label{thm: Milnor Hauptvermutung (ii)} No finite cell subdivision of $X_1$ is isomorphic to a cell subdivision of $X_2$.
\end{enumerate}
\end{thm}

Let us prove (\ref{thm: Milnor Hauptvermutung (ii)}). Denote by $x_0\in X_q$ the vertex of the cone in $X_q$.  $X_q$ is a manifold except at the exceptional point $x_0$; $X_q-x_0$ is homeomorphic to $L(7,q)\times \bR^n$ and has fundamental group $\Pi\simeq \bZ_7$. Let $K_q$ be the single point compactification of the universal cover of $X_q-x_0$. Group $\Pi$ acts on $K_q$ with a single fixed point $k_0$. The quotient $K_q/\Pi$ is $X_q$ and any cell structure on $X_q$ induces a $\Pi$-equivariant cell structure on $K_q$. If cells of $X_q-x_0$ are taken to be the cells of $L(7,q)$ times $B^n$, then the cellular chain complex $C_\bt(K_q,k_0;\bZ)$ is $C_\bt(\widetilde{L(7,q)};\bZ)$ with degrees shifted by $n$, and is a complex of free modules over $\bZ[\Pi]$. Hence, for $h:\Pi\ra \bC^*$ a homomorphism, the Reidemeister torsion $\tau_h(X_q,x_0)$ is defined (as the torsion of the complex $\bC\otimes_{\bZ[\Pi]}C_\bt(K_q,k_0;\bZ)$ with cellular basis) and equal to $\tau_h(L(7,q))^{(-1)^n}$. Thus, $\tau_h(X_1,x_0)\neq \tau_h(X_2,x_0)$ and, by invariance of torsions with respect to cellular subdivision, $X_1$ and $X_2$ do not possess isomorphic cellular structures.

\lec{Lecture 12, 15.05.2014}

Statement (\ref{thm: Milnor Hauptvermutung (i)}) of Theorem \ref{thm: Milnor Hauptvermutung} follows, in the case $n+3\geq 7$, from the following general theorem of Mazur \cite{Mazur}.

\begin{thm}[Mazur, 1961]\label{thm: Mazur} Let $M_1$ and $M_2$ be two closed differentiable $k$-manifolds which are parallelizable and have the same homotopy type. Then for $n> k$, $M_1\times \bR^n$ is diffeomorphic to $M_2\times \bR^n$.
\end{thm}

In particular, this implies that for $n\geq 4$, manifolds $L(7,1)\times \bR^n$ and $L(7,2)\times \bR^n$ are diffeomorphic. Hence, their one-point compactifications $X_0$ and $X_1$ are homeomorphic.

\subsection{Torsion of the product}
\begin{thm} Let $X$, $Y$ be two finite CW-complexes with $Y$ simply connected, and let $h:\pi_1(X)\ra \bC$ be a group homomorphism such that $\tau_h(X)$ is defined. Then $\tau_h(X\times Y)$ is defined and equal to
$$\tau_h(X\times Y)=\tau_h(X)^{\chi(Y)}$$
where $\chi(Y)$ is the Euler characteristic of $Y$.
\end{thm}

\begin{proof} Let us order the cells of $Y$ so that the dimensions are nondecreasing: $e_1,\ldots,e_N$. Denote $Y_i=\cup_{j=1}^i e_i \subset Y$ is a CW-subcomplex (we also set $Y_0=\varnothing$). Chain complex
$C_\bt(X\times Y_i,X\times Y_{i-1})$ is isomorphic to $C_\bt(X)$ except for a degree shift by $\dim e_i$. Thus $\tau_h(X\times Y_i,X\times Y_{i-1})=\tau_h(X)^{(-1)^{\dim e_i}}$.
Thus
$$\tau_h(X\times Y)=\prod_{i=1}^N \tau_h(X\times Y_i,X\times Y_{i-1})=\tau_h(X)^{\sum_{i=1}^N \dim e_i}=\tau_h(X)^{\chi(Y)}$$
\end{proof}

A curious application of this, together with Mazur's Theorem \ref{thm: Mazur} is the following \cite{MilnorHaupt}.
\begin{thm} As manifolds with boundary, $M_1=L(7,1)\times B^n$ and $M_2=L(7,2)\times B^n$ are not diffeomorphic, however their interiors are diffeomorphic for $n\geq 4$.
\end{thm}

\begin{proof}
Indeed, torsions $\tau_h(M_q)=\tau_h(L(7,q))^{1=\chi(B^n)}$ distinguish between $M_1$ and $M_2$. (By Chapman's theorem this implies that $M_1$ and $M_2$ are actually not homeomorphic.) The interiors are diffeomorphic to $L(7,q)\times\bR^n$ and hence $int(M_1)$ and $int(M_2)$ are diffeomorphic by Theorem \ref{thm: Mazur}.
\end{proof}

Note also that for $n$ odd, $n\geq 5$, boundaries $\dd M_q= L(7,q)\times S^{n-1}$ are not homeomorphic by the torsion argument (since $\tau_h(\dd M_q)=\tau_h(L(7,q))^2$).

\subsection{Chapman's theorem}
Denote $I_j=[-1,1]$, $j=1,2,3,\ldots$ and  let $Q=\prod_{j=1}^\infty I_j$ be the Hilbert cube. Chapman proved the following \cite{Chapman74}.\footnote{We are citing after the appendix of \cite{Cohen} (p.102).}

\begin{thm}[Chapman, 1974]
If $X$ and $Y$ are finite CW-complexes then $f: X\ra Y$ is a simple homotopy equivalence if and only if $f\times\mr{id}_Q: X\times Q\ra Y\times Q$ is homotopic to a homeomorphism of $X\times Q$ onto $Y\times Q$.
\end{thm}

\begin{corollary}[Topological invariance of Whitehead torsion]
If $f:X\ra Y$ is a homeomorphism then $f$ is a simple homotopy equivalence.
\end{corollary}
\begin{proof} Indeed, $f\times\mr{id}_Q: X\times Q\ra Y\times Q$ is a homeomorphism, hence by Theorem $f$ is a simple homotopy equivalence.  \end{proof}

\begin{corollary} Finite CW-complexes $X$ and $Y$ have same simple homotopy type iff $X\times Q$ is homeomorphic to $Y\times Q$.
\end{corollary}
\begin{proof}
``If'' part: assume $F:X\times Q\xra{\simeq} Y\times Q$ is a homeomorphism. Denote by $f$ the composition
$$X\xra{\times 0} X\times Q\xra{F} Y\times Q \xra{\pi} Y$$
Then $f\times \mr{id}_Q$ is homotopic to $F$. Hence, by Theorem, $f$ is a simple homotopy equivalence.

For the ``only if'' part, it is sufficient to consider the case when $Y$ is an elementary expansion of $X$. By Theorem, for $\iota:X\hra Y$ the inclusion, the map $\iota\times \mr{id}_Q:X\times Q\ra Y\times Q$ is homotopic to a homeomorphism. In particular, homeomorphism $X\times Q\simeq Y\times Q$ exists.
\end{proof}

\subsection{$h$-cobordisms}
\begin{definition}
Let $M$ be a smooth compact connected manifold with boundary $\dd M=N\sqcup N'$ such that both $N$ and $N'$ are deformation retracts of $M$. Then the triple $(M;N,N')$ is called an {\it $h$-cobordism}.
\end{definition}
Note that for an $h$-cobordism $\pi_1(M)\simeq \pi_1(N)\simeq \pi_1(N')$.

\begin{thm}[Smale, 1962]
If $(M;N,N')$ is an $h$-cobordism, with $\pi_1$ trivial and $\dim M\geq 6$, then $M$ is diffeomorphic to the product $M\simeq N\times [0,1]$.
\end{thm}

\begin{thm}[Barden, Mazur, Stallings]
\begin{enumerate}[(i)]
\item (``$s$-cobordism theorem''.) Let $(M;N,N')$ be an $h$-cobordism, $\dim M \geq 6$. Then the torsion $\tau(M,N) \in Wh(\pi_1)$ is trivial iff $M$ is diffeomorphic to $N\times [0,1]$.
\item If $N$ is a closed smooth manifold of dimension $\geq 5$ and $\tau_0\in Wh(\pi_1(N))$ any element, then there exists an $h$-cobordism $(M;N,N')$ with $\tau(M,N)=\tau_0$.
\item If $(M;N,N')$ and $(M_1;N,N'_1)$ are two $h$-cobordisms of dimension $\geq 6$ with coinciding torsion $\tau(M,N)=\tau(M_1,N)\in Wh(\pi_1)$, then $M_1$ is diffeomorphic to $M$ relative to $N$.
\end{enumerate}
\end{thm}

{\bf Example.}\cite{MilnorHaupt} 7-manifolds $L(7,1)\times S^4$ and $L(7,2)\times S^4$ are $h$-cobordant but not diffeomorphic.

\subsection{Poincar\'e duality for torsions}
\subsubsection{Poincar\'e duality for cellular decompositions of manifolds}
\begin{definition} A CW-complex $X$ for which characteristic maps for cells $\phi_\alpha: B^k \ra X$ are homeomorphisms onto the image of $\chi$ on the full closed ball $B^k$ (as opposed to its interior), is called a {\it ball complex}.
\end{definition}
In particular, for a ball complex, matrix elements of the boundary map in cellular chains 
$\in\{0,\pm 1\}$ (in cellular basis).

Let $M$ be an oriented smooth $n$-manifold without boundary, and $X$ a cellular decomposition of $M$ which is a ball complex. Then there exists another cellular decomposition $X^\vee$ of $M$ such that for every $k$-cell $e^k$ of $X$ there is a {\it dual} $(n-k)$-cell $\check{e}^{n-k}=\varkappa(e^k)$ of $X^\vee$ intersecting $e^k$ transversally in a single point and 
contained in the open star of $e^k$ in $X$ (union of top cells containing $e^k$). Map $\varkappa$ is assumed to be a bijection between cells of $X$ and cells of $X^\vee$.
Thus we have an isomorphism of free abelian groups
\be\varkappa:\;C_k(X;\bZ)\xra{\sim} C_{n-k}(X;\bZ) \label{varkappa}\ee
We orient cells of $X$ somehow and induce an orientation on cells of $X^\vee$ by requiring that the intersection pairing is $e^k\cdot \check{e}^{n-k}=1$. Then we have the non-degenerate intersection pairing between chains
$$\cdot:\quad C_k(X;\bZ)\otimes C_{n-k}(X^\vee;\bZ)\ra \bZ$$
It induces an isomorphism between $k$-chains and $(n-k)$-cochains
$$C_k(X)\xra{\sim} C^{n-k}(X^\vee)=\mr{Hom}(C_{n-k}(X^\vee),\bZ)  $$
which gives a {\it chain} isomorphism $C_\bt(X)\xra{\sim} C^{n-\bt}(X^\vee)$ and in turn an isomorphism bewteen homology and cohomology
$$ H_k(M)\cong H_k(X)\xra{\sim} H^{n-k}(X^\vee)\cong H^{n-k}(M)$$
Note that $\varkappa$ in (\ref{varkappa}) is not a chain map.
For $e^k_i$ and $e^{k-1}_j$ a pair of cells of $X$, let $\dd_{ji}^X$ be the corresponding matrix element of $\dd^X: C_k(X)\ra C_{k-1}(X)$. Then the matrix element of $\dd^{X^\vee}: C^{n-k+1}(X^\vee)\ra C^{n-k}(X^\vee)$ between $\varkappa(e^{k-1}_j)$ and $\varkappa(e^k_i)$ is $\dd^{X^\vee}_{ij}=\dd_{ji}^X$.\footnote{
In fact, there is a sign: $\dd^{X^\vee}_{ij}=(-1)^k\dd_{ji}^X$. We ignore this sign, since it depends only on the degree $k$ and not on particular cells, and therefore is irrelevant for the torsion.
} Thus, matrices of $\dd^X$ and $\dd^{X^\vee}$ are mutually transpose.

The dual cell decomposition $X^\vee$ can be explicitly constructed from $X$ by passing to the barycentric subdivision $\beta(X)$ (which is a simplicial complex which combinatorially is the nerve of the partially ordered set of cells of $X$ by adjacency, with simplices of $\beta(X)$ corresponding to strictly increasing sequences of cells of $X$). Top cells of $X^\vee$ are constructed as stars of vertices of $X$ in $\beta(X)$. Cell $\varkappa(e^k)$ is constructed as the intersection of stars of vertices of the closed cell $E^k$ of $X$ in $\beta(X)$, intersected with the star of $e^k$ in $X$:
$$\varkappa(e^k)=\bigcap_{v\in \sk_0(E^k)} \mr{star}_{\beta(X)}(v)\quad \cap \mr{star}_X (e^k)$$

\subsubsection{Case of manifolds with boundary}\label{sec: Poincare duality with boundary}
For $M$ an oriented $n$-manifold with boundary $\dd M$ and $X$ a cellular decomposition of $M$, a dual cellular decomposition $X^\vee$ consists of an $(n-k)$-cell $\varkappa(e^k)\subset M-\dd M$ for every $k$-cell of $M$ and, additionally, for $e^k\subset \dd M$, an $(n-k-1)$-cell $\varkappa_\dd(e^k)$ of $\dd M$. Intersection pairing is a non-degenerate pairing
$$\cdot:\quad C_k(X)\otimes C_{n-k}(X^\vee,\dd M)\ra \bZ$$

More generally, if $\dd M=\dd_1 M\sqcup \dd_2 M$, one
introduces a new CW-complex $Y$ as $X^\vee$ minus images of cells of $\dd_1 M$ in $X$ under $\varkappa,\varkappa_\dd$. The underlying topological space of $Y$ is homeomorphic to $M$.
There is a non-degenerate pairing
$$\cdot:\quad C_k(X,\dd_1 M)\otimes C_{n-k}(Y,\dd_2 M)\ra \bZ$$
which induces an chain isomorphism between chains and cochains
$$C_\bt (X,\dd_1 M) \xra{\sim} C^{n-\bt}(Y,\dd_2 M)$$
and in turn on the level of (co-)homology:
$$\underbrace{H_\bt (X,\dd_1 M)}_{H_\bt(M,\dd_1 M)} \xra{\sim} \underbrace{H^{n-\bt}(Y,\dd_2 M)}_{H^{n-\bt}(M,\dd_2 M)}$$
-- the usual Poincar\'e-Lefschetz duality.

\subsubsection{Algebraic duality lemma}
\begin{definition}\label{def: dual module}
Let $a\mapsto \bar a$ be an anti-involution of a ring $A$. For $F$ a left 
$A$-module, we endow the $F^\vee=\mr{Hom}(F,A)$ with a left $A$-module structure by setting
$$(a\phi)(x)= \phi(x)\bar a$$
We call $F^\vee$ a {\it dual module} to $F$.
\end{definition}

Note that $a$ induces an anti-involution $m\mapsto m^\dagger$, $(m^\dagger)_{ij}=\bar m_{ji}$ on matrices over $A$ and, in turn an involution on $K_1(A)$.

\begin{lemma}\label{Lm: alg duality lemma}
Let $C_n\ra C_{n-1}\ra \cdots \ra C_1\ra C_0$ be an acyclic based chain complex of free left $A$-modules. Then the dual chain complex $C^\vee$ with $(C^\vee)_{n-k}=(C_k)^\vee$ and the boundary maps $\dd^\vee_{C^\vee_{n-k+1}\ra C^\vee_{n-k}}=(\dd_{C_k\ra C_{k-1}})^\dagger$ is also an acyclic based chain complex of free left $A$-modules and the torsions are related by
$$\tau(C^\vee)=\overline{\tau(C)}^{(-1)^{n+1}}\qquad\in \bar K_1(A)$$
\end{lemma}

\begin{proof} If $a_k\in GL(A)$ is the transition matrix from basis $b_k b_{k-1}$ to basis $c_k$ in the chain module $C_k$ (in the notations of (\ref{def: torsion})), then the transition matrix between dual bases $(b_k)^\vee (b_{k-1})^\vee$ and $(c_k)^\vee$ in the dual module $C^\vee_{n-k}$ is the inverse adjoint of $a_k$, i.e. $(a^\vee)_{n-k}=a_k^{-1\dagger}$. Projecting to $\bar K_1(A)$, we have $[(a^\vee)_{n-k}]=\overline{[a_k]}^{-1}$. Therefore, for the torsion we have
$$\tau(C^\vee)=\prod_k [(a^\vee)_k]^{(-1)^k}= \prod_k [(a^\vee)_{n-k}]^{(-1)^{n-k}}=\prod_k \overline{[a_k]}^{(-1)^{n-k+1}}=\overline{\tau(C)}^{(-1)^{n+1}} $$
\end{proof}

\subsubsection{Duality for the R-torsion of a closed manifold}
For $M$ a smooth closed oriented $n$-manifold, let $X$ be a cell decomposition of $M$ (which is assumed to be a ball complex); it lifts to a $\pi_1(M)$-equivariant cell decomposition $\tilde X$ of the universal cover $\tilde M$. Denote $\tilde X^\vee$ the dual cell decomposition of $\tilde M$ -- the equivariant lift of the $X^\vee$, which is  dual to $X$.

Chain complexes $C_\bt(\tilde X;\bZ)$ and $C_\bt(\tilde X^\vee;\bZ)$ can be viewed as complexes of free left modules over $\bZ[\pi_1(M)]$. Moreover, $C_{n-k}(\tilde X^\vee)$ is a dual $\bZ[\pi_1]$-module to $C_k (\tilde X)$ with respect to the anti-involution on $\bZ[\pi_1]$ induced from the map $\sigma\mapsto \sigma^{-1}$ on $\pi_1$.\footnote{
For $\tilde e$ a $k$-cell of $\tilde X$ and $\tilde e^\vee$ a $(n-k)$-cell of $\tilde X^\vee$, we set
$\langle \tilde e^\vee, \tilde e\rangle=\pm \sigma$ iff $\tilde e^\vee\cdot \sigma^{-1}(e^\vee)=\pm 1$ for an element $\sigma\in \pi_1(M)$ (equivalently, for general chains, $\langle x^\vee, y\rangle=\sum_{\sigma\in\pi_1}(x^\vee\cdot \sigma^{-1}(y))\, \sigma$). This definition implies that $\langle \tilde e^\vee ,\sigma' e^\vee  \rangle = \pm \sigma'\cdot \sigma = \sigma' \langle \tilde e^\vee ,e^\vee \rangle$. Thus, $\langle \tilde e^\vee , \bt \rangle : C_k(\tilde X)\ra \bZ[\pi_1] $ is indeed a morphism of $\bZ[\pi_1]$-modules. Note also that we have $\langle \sigma' \tilde e^\vee, \tilde e \rangle= \langle \tilde e^\vee, \tilde e\rangle \cdot (\sigma')^{-1}$. Hence, $C_{n-k}(\tilde X^\vee)$ is indeed a dual $\bZ[\pi_1]$-module to $C_k(\tilde X)$, as in Definition \ref{def: dual module}.
}
Also, if $a_{ij}$ is the matrix of $\dd: C_k(\tilde X)\ra C_{k-1}(\tilde X)$ (with entries in $\bZ[\pi_1]$), i.e. $\dd \tilde e^k_i=\sum_j a_{ij} \tilde e^{k-1}_j$, then $\dd^\vee \varkappa(\tilde  e^{k-1}_{j})=\sum_i \bar a_{ij} \varkappa(\tilde e^k_i)$,\footnote{
Conjugation of matrix elements here stems from the fact that if $\sigma(\tilde e^{k-1})$ is a part of topological boundary of $\tilde e^k$, then $\sigma^{-1}(\varkappa(\tilde e^k))\subset \dd^\vee \varkappa(\tilde e^{k-1})$. (Note that the map $\varkappa$ is $\pi_1$-equivariant.)
} i.e. $\dd^\vee_{C^\vee_{n-k+1}\ra C^\vee_{n-k}}=(\dd_{C_k\ra C_{k-1}})^\dagger$.


Let $h: \pi_1(M)\ra O(m)$ be an orthogonal representation, and let
$$C'_k=\bR^m\otimes_{\bZ[\pi_1]}C_k(\tilde X;\bZ),\qquad C^{'\vee }_k=\bR^m\otimes_{\bZ[\pi_1]}C_k(\tilde X^\vee;\bZ)$$
Then the real vector space $C'_{n-k}$ is canonically dual to $C'_k$.\footnote{
The explicit pairing between $C'_{n-k}$ and $C'_k$ is:
$\langle v\otimes x^\vee , w\otimes y\rangle= \sum_{\sigma\in \pi_1} \langle v, h(\sigma)w\rangle_{\bR^m}\, (x^\vee\cdot \sigma(y))$.
}
Therefore, by Lemma \ref{Lm: alg duality lemma},
\be\tau_h(M)=\tau_h(M)^{(-1)^{n+1}}\label{duality for R-torsion}\ee
(if defined).
Hence we obtained the following.
\begin{thm} For an even dimensional smooth closed oriented manifold $M$, every $R$-torsion which is defined (the corresponding complex is acyclic), is trivial $\tau_h(M)=1$.
\end{thm}


\subsubsection{Duality for Whitehead torsion of an h-cobordism}
\begin{thm}
Let $N\hra M \hla N'$ be an oriented $n$-dimensional h-cobordism. Then for the Whitehead torsion we have
$$\tau(M,N')=\overline{\tau(M,N)}^{(-1)^{n+1}}\qquad \in Wh(\pi_1(M))$$
\end{thm}

\begin{proof}
Indeed, one considers a cellular decomposition $X$ of $M$ and constructs a new one $Y$ as in Section \ref{sec: Poincare duality with boundary}, so that the intersection pairing $C_k(X,N)\otimes C_{n-k}(Y,N')\ra \bZ$ is non-degenerate. Passing to the universal cover, $C_{n-k}(\tilde Y,\tilde N')$ is a dual $\bZ[\pi_1]$-module to $C_k(\tilde X,\tilde N)$ and, by Lemma \ref{Lm: alg duality lemma}, the Theorem follows.
\end{proof}

For the  $R$-torsion
for an oriented $n$-manifold with boundary $\dd M=\dd_1 M\sqcup \dd_2 M$, we have the following generalization of (\ref{duality for R-torsion}):
$$\tau_h(M,\dd_1 M)=\tau_h(M,\dd_2 M)^{(-1)^{n+1}}$$


\lec{Lecture 13, 22.05.2014}

\section{Analytic torsion}

Let $M$ be a compact oriented smooth $n$-manifold without boundary with Riemannian metric $g$. Let $E$ be a rank $m$ real vector bundle over $M$ with positive definite inner product in fibers $(,)_E:E_x\otimes E_x\ra \bR$ for any $x\in M$, endowed with a flat connection $\nabla$ compatible with the fiber inner product (i.e. parallel transport by $\nabla$ preserves norms of vectors). Consider differential forms on $M$ with coefficients in $E$, $\Omega^k(M,E)=\Gamma(M,(\wedge^k T^*M)\otimes E)$. Connection $\nabla$ induces a differential operator (twisted de Rham operator)
$$d_E: \Omega^k(M,E)\ra \Omega^{k+1}(M,E)$$
Flatness of $\nabla$ is equivalent to $d_E^2=0$. Riemannian metric on $M$ together with fiberwise inner product on $E$ induce the positive definite Hodge inner product on $E$-valued forms $(,): \Omega^k(M,E)\otimes \Omega^k(M,E)\ra \bR$ given by
$$(\alpha,\beta)= \int_M (\alpha\stackrel{\wedge}{,} *\beta)_E,\qquad \alpha,\beta\in \Omega^k(M,E) $$
where $*: \Omega^k(M,E)\ra \Omega^{n-k}(M,E)$ is the Hodge star associated to the metric $g$, acting trivially on the $E$-coefficients. Pairing
$$(\bt \stackrel{\wedge}{,}\bt)_E: \Omega^k(M,E)\otimes_{C^\infty(M)} \Omega^l(M,E)\ra \Omega^{k+l}(M)$$
is the wedge product in forms accompanied by the inner product $(,)_E$ in $E$-coefficients. Denote
$$d^*_E:\Omega^k(M,E)\ra \Omega^{k-1}(M,E),\qquad \alpha \mapsto -(-1)^{(k+1)n}*d_E * \alpha$$
the adjoint of $d_E$ with respect to the Hodge inner product, i.e. $(\alpha,d_E\beta)=(d_E^*\alpha,\beta)$.
The Laplacians $\Delta^{(k)}
$ for $0\leq k\leq n$ are defined as
$$\Delta^{(k)}=d_E d^*_E+d^*_E d_E:\quad \Omega^k(M,E)\ra \Omega^k(M,E)$$
By a straightforward extension of Hodge decomposition theorem to coefficients in a flat bundle $E$, one has a decomposition of $k$-forms into sum of orthogonal subspaces
\be \Omega^k(M,E)=\underbrace{\Omega^k_\mr{Harm}(M,E)}_{\ker \Delta^{(k)}}\oplus \underbrace{\Omega^k_\mr{ex}(M,E)}_{d_E(\Omega^{k-1}(M,E))}\oplus \underbrace{\Omega^k_\mr{coex}(M,E)}_{d^*_E(\Omega^{k+1}(M,E))} \label{Hodge decomp}\ee
and the first term on the right is isomorphic to the $k$-th cohomology space of $d_E$.

{\bf Assumption.} We will assume that the twisted de Rham complex $\Omega^\bt(M,E), d_E$ is acyclic. Thus there is no first term in the Hodge decomposition (\ref{Hodge decomp}).


\subsection{Zeta regularized determinants}
For each $k=0,1,\ldots,n$ the operator $\Delta^{(k)}$ has positive discrete eigenvalue spectrum $\{\lambda^{(k)}_j\}$ with the only accumulation point at infinity. Spectral density of $\Delta^{(k)}$ at $\lambda\ra+\infty$ behaves as $\rho(\lambda)\sim C\cdot \mr{vol}(M)\cdot\lambda^{n/2-1} $

For the time being we will assume that some $k$ is fixed and we will suppress the superscripts in $\Delta^{(k)}$, $\lambda_j^{(k)}$.

Zeta function of $\Delta$ is defined as
$$\zeta_{\Delta}(s)=\sum_{j} (\lambda_j)^{-s}$$
The sum converges for $\mr{Re}(s)>n/2$ and also can be written as $$\zeta_{\Delta}(s)=\mr{tr}_{\Omega^k(M,E)}\Delta^{-s}$$

\subsubsection{Heat equation}
Consider the heat equation, i.e. the following parabolic PDE
\be\left(\frac{\dd}{\dd t}+\Delta\right)\phi(x,t)=0\label{heat equation}\ee
on a $t$-dependent $E$-valued $k$-form on $M$, $\phi\in \Omega^k(M,E)\hat\otimes C^\infty([t_0,+\infty))$
For any $\phi_0\in \Omega^k(M,E)$ there exists a unique solution $\phi$ of the heat equation (\ref{heat equation}) on $M\times [0,+\infty)$ satisfying Cauchy boundary condition $\phi|_{M\times \{0\}}=\phi_0$. Denote the linear operator taking $\phi_0$ to $\phi(\bt,t)$ as
$$e^{-t\Delta}:\; \Omega^k(M,E)\ra \Omega^k(M,E)$$
For $t>0$ it is an integral operator
\be e^{-t\Delta}:\quad \phi_0(x)\mapsto \phi_t(x)=\int_{M\ni y}K(x,y;t)\wedge\phi_0(y)\label{heat kernel}\ee
where $K(x,y;t)\in \wedge^k T^*_x M\otimes \wedge^{n-k}T^*_y M \otimes E\otimes E^*$ are the values of the integral kernel (=parametrix = fundamental solution of the heat equation= heat kernel). Wedge product sign in (\ref{heat equation}) implicitly contains canonical pairing between $E^*$ and $E$.
In other words, we have a $t$-dependent differential form $K_t\in \Omega^n(M\times M, \pi_1^* E\otimes \pi_2^* E^*)$ with $\pi_{1,2}:M\times M\ra M$ projections to the first and second factor respectively. In terms of $K_t$ we have
$$e^{-t\Delta}:\quad \phi_0\mapsto \phi_t=(\pi_1)_*(K_t\wedge \pi_2^* \phi_0)
$$
where $(\pi_1)_*:\Omega^{n+k}(M\times M,\pi_1^*E)\ra \Omega^k(M,E)$ stands for the fiber integral over fibers of $\pi_1$.

At $t\ra +\infty$, $K(x,y;t)\sim e^{-t\lambda_1}$ where $\lambda_1>0$ is the smallest eigenvalue of $\Delta$. At $t\ra 0$ the heat kernel behaves singularly on the diagonal $M_\mr{diag}\subset M\times M$. More precisely, the following result is due to Seeley \cite{Seeley}:
\begin{thm}[Seeley, 1966]\label{thm: Seeley}
There exists a sequence of forms $a_j\in \Omega^n(T,\pi_1^* E\otimes \pi_2^* E^*)$ on a tubular neighborhood $T$ of $M_\mr{diag}$ in $M\times M$, for $j=0,1,2,\ldots$, such that
for any $N\in\bN_0$ we have
$$K(x,y;t)=\frac{e^{-\frac{d(x,y)^2}{4t}}}{(4\pi t)^{n/2}}\;\;\sum_{j=0}^N  a_j(x,y)\cdot t^j\quad +r_N(x,y;t)$$
 for $(x,y)\in T$ with the remainder $r_N(x,y;t)$ behaving as $O(t^{N+1-n/2})$ at $t\ra 0$. Here $d(x,y)$ denotes  the geodesic distance between $x$ and $y$.
\end{thm}

Using the fact that for $\lambda>0$ one can write $\lambda^{-s}=\frac{1}{\Gamma(s)}\int_0^{\infty}dt\; t^{s-1}\; e^{-\lambda t}$ with $\Gamma(s)$ the Euler's gamma function, we rewrite the zeta function of the Laplacian as
\be \zeta_\Delta(s)=\frac{1}{\Gamma(s)}\int_0^\infty dt\; t^{s-1}\;\mr{tr}_{\Omega^k(M,E)}e^{-t \Delta}=\frac{1}{\Gamma(s)}\int_0^\infty dt\; t^{s-1}\;\int_M \mr{tr}_E K(x,x;t) \label{zeta via heat kernel}\ee
where $K(x,x;t)$ denotes the pullback of the $n$-form $K_t$ to the diagonal $M_\mr{diag}\subset M\times M$. The integrand exponentially decays at $t\ra 0$ and behaves as $t^{s-1-n/2}$ at $t\ra 0$. Thus the integral over $t$ converges absolutely iff $\mr{Re}(s)>n/2$. Using Theorem \ref{thm: Seeley}, we can write
\begin{multline}\label{zeta analytic continuation}
\zeta_\Delta(s)=
\frac{1}{\Gamma(s)}\left(\int_\epsilon^\infty dt\; t^{s-1}\;\int_M \mr{tr}_E K(x,x;t)+\right. \\
\left. +
(4\pi)^{-n/2}\sum_{j=0}^N\frac{\epsilon^{s-n/2+j}}{s-n/2+j}\cdot \int_M\mr{tr}_E a_j(x,x)
+\int_0^\epsilon dt\; t^{s-1}\;\underbrace{\int_M \mr{tr}_E r_N(x,x;t)}_{O\left(t^{N+1-n/2}\right)}\right)
\end{multline}
for any $\epsilon>0$. This expression coincides with (\ref{zeta via heat kernel}) for $\mr{Re}(s)>n/2$. But note that (\ref{zeta analytic continuation}) actually makes sense for $\mr{Re}(s)> n/2-N$, provided that $s\neq n/2-j$ for $j=0,1,\ldots,N$.

Thus, taking the direct limit $N\ra\infty$, we obtain an analytic continuation of $\zeta_\Delta(s)$ as meromorphic function on $s\in \bC$ with simple poles at $s=n/2-j$, $j=0,1,\ldots$. Note that for $n$ even, for $s=0,-1,-2,\ldots$ a non-posititive integer, the pole is cancelled by a zero of $1/\Gamma(s)$. Thus we have two cases:
\begin{itemize}
\item $n$ even, $\zeta_\Delta(s)$ has $n/2$ simple  poles at points $s=n/2,n/2-1,\ldots,1$ and is analytic everywhere else.
\item $n$ odd, $\zeta_\Delta(s)$ has infinitely many simple poles at $s=n/2,n/2-1,\ldots$ and is analytic everywhere else.
\end{itemize}
Observe that $s=0$ is always a regular point.

\begin{rem}
Note that for $n$ odd, 
the analytic continuation of $\int_0^\infty dt\;t^{s-1}\;K(x,x;t)$ is regular at $s=0$ for any point $x\in M$. (We will need this later for the proof of metric independence of the analytic torsion.)
\end{rem}

\begin{definition}
Zeta regularized determinant of $\Delta$ is defined as
$${\det}_{\Omega^k(M,E)}\Delta:=e^{-\zeta'_\Delta(s)}$$
where $\zeta'_\Delta(s)=\frac{\dd}{\dd s}\,\zeta_\Delta(s)$.
\end{definition}

The motivation for this definition is that for $L$ a strictly positive self-adjoint linear operator on a finite dimensional Euclidean vector $V$, one can introduce its zeta function
$\zeta_L(s)=\sum_{\lambda} \lambda^{-s}=\mr{tr}_V L^{-s}$ and then $\zeta'_L(s)=\sum_\lambda -\log\lambda\cdot \lambda^{-s}$. Therefore, $\zeta'(0)=-\sum_\lambda \log\lambda=-\log\det L$. So, $\det L=e^{-\zeta'_L(0)}$.

\subsection{Ray-Singer torsion}
\begin{definition}[Ray-Singer \cite{Ray-Singer}] Analytic torsion of a closed smooth oriented Riemannian manifold $(M,g)$ with an acyclic flat Euclidean bundle $(E,(,)_E,\nabla)$ over it is defined as
$$\tau_{RS}(M,E)=\prod_{k=0}^n \left({\det}_{\Omega^k(M,E)}\Delta^{(k)}\right)^{-\frac{(-1)^k\; k}{2}}\qquad \in \bR_+$$
\end{definition}

Denote
\begin{eqnarray*}
\Delta_\mr{coex}^{(k)}=d^*_E d_E:&& \Omega^k_\mr{coex}(M,E)\ra \Omega^k_\mr{coex}(M,E),\\
\Delta_\mr{ex}^{(k)}=d_E d^*_E:&& \Omega^k_\mr{ex}(M,E)\ra \Omega^k_\mr{ex}(M,E)
\end{eqnarray*}
Then $\Delta^{(k)}=\Delta^{(k)}_{coex}\oplus \Delta^{(k)}_{ex}$ (a ``block diagonal'' operator on $\Omega^k(M,E)=\Omega^k_\mr{coex}\oplus \Omega^k_\mr{ex}$).
Eigenvalue spectrum of $\Delta^{(k)}$ is the union of spectra of $\Delta^{(k)}_\mr{coex}$ and $\Delta^{(k)}_\mr{ex}$ (as sets with multiplicites): $\{\lambda^{(k)}_j\}=\{\lambda^{(k),\mr{coex}}_j\}\cup \{\lambda^{(k),\mr{ex}}_j\}$. Also, by virtue of $d_E:\Omega^k_\mr{coex}\xra{\sim} \Omega^k_\mr{ex}$ being in isomorphism commuting with Laplacians, we have $\{\lambda^{(k),\mr{coex}}_j\}=\{\lambda^{(k+1),\mr{ex}}_j\}$.

We can define the zeta functions
\be\zeta_{\Delta^{(k)}_\mr{coex}}(s)=\sum_{j}\left(\lambda^{(k),\mr{coex}}_j\right)^{-s},\qquad \zeta_{\Delta^{(k)}_\mr{ex}}(s)=\sum_{j}\left(\lambda^{(k),\mr{ex}}_j\right)^{-s}\label{zeta coex and ex}\ee
Due to the relations between spectra above, we have $$\zeta_{\Delta^{(k)}}(s)=\zeta_{\Delta^{(k)}_\mr{coex}}(s)+\zeta_{\Delta^{(k)}_\mr{ex}}(s)$$ and
\be\zeta_{\Delta^{(k)}_\mr{coex}}(s)=\zeta_{\Delta^{(k+1)}_\mr{ex}}(s)\label{zeta rel coex vs ex}\ee
These relations together imply \be\zeta_{\Delta^{(k)}}(s)=\zeta_{\Delta^{(k)}_\mr{coex}}(s)+\zeta_{\Delta^{(k-1)}_\mr{coex}}(s)\label{zeta relation coex}\ee
and the inverse relation
$$\zeta_{\Delta^{(k)}_\mr{coex}}(s)=\sum_{j=0}^k (-1)^j\zeta_{\Delta^{(k-j)}}(s)$$
which is true for $\mr{Re}(s)>n/2$ when series (\ref{zeta coex and ex}) converge, and can serve as a {\it definition} of $\zeta_{\Delta^{(k)}_\mr{coex}}(s)$ for general $s\in\bC$, using the analytic continuation of $\zeta_{\Delta^{(k)}}(s)$ constructed via the heat kernel.

Relation (\ref{zeta relation coex}) implies in particular
\be\sum_{k=0}^n \frac{(-1)^k}{2}\;\zeta_{\Delta^{(k)}_\mr{coex}}(s)=\sum_{k=0}^n -\frac{(-1)^k\cdot k}{2}\;\zeta_{\Delta^{(k)}}(s)\label{zeta rel telescopic}\ee
since $-\frac{(-1)^k\cdot k}{2}+-\frac{(-1)^{k+1}\cdot (k+1)}{2}=\frac{(-1)^k}{2}$.

Setting ${\det}_{\Omega^k_\mr{coex}(M,E)}\Delta^{(k)}_\mr{coex}:=e^{-\zeta'_{\Delta^{(k)}_\mr{coex}}(0)}$ and applying $-\left.\frac{\dd}{\dd s}\right|_{s=0}$ to (\ref{zeta rel telescopic}), we can rewrite the definition of Ray-Singer torsion as
$$\tau_{RS}(M,E)=\prod_{k=0}^{n} \left({\det}_{\Omega^k_\mr{coex}(M,E)}\Delta^{(k)}_\mr{coex}\right)^{(-1)^k/2}$$

\subsection{Poincar\'e duality, torsion of an even dimensional manifold}
Hodge star $*: \Omega^k_\mr{coex}(M)\leftrightarrow \Omega^{n-k}_\mr{ex}(M)$ commutes with Laplacians. In particular, for the eigenvalue spectra we have $\{\lambda^{(k)}_\mr{coex}\}=\{\lambda^{(n-k)}_\mr{ex}\}$. Together with (\ref{zeta rel coex vs ex}) this implies the following relation for zeta functions:
$$\zeta_{\Delta^{(k)}_\mr{coex}}(s)=\zeta_{\Delta^{(n-k-1)}_\mr{coex}}(s)$$
Hence for the Ray-Singer torsion we have
$$\tau_{RS}(M,E)=\tau_{RS}(M,E)^{(-1)^{n+1}}$$
Since the torsion is a positive real number, this implies the following.
\begin{thm}\label{thm: RS torsion of even dim mfd}
Let $M$ be an even dimensional closed oriented Riemannian manifold with a flat acyclic Euclidean vector bundle $E$. Then
$$\tau_{RS}(M,E)=1$$
\end{thm}

\subsection{Metric independence of Ray-Singer torsion}
\begin{thm}[Ray-Singer, \cite{Ray-Singer}]\label{thm: metric independence of RS torsion}
Torsion $\tau_{RS}(M,E)$ does not depend on the choice of Riemannian metric $g$ on $M$.
\end{thm}

\begin{proof}
Case of
$n=\dim M$ even is trivial by Theorem \ref{thm: RS torsion of even dim mfd}. Consider the case of $n$ odd. Let $g_u$ be a smooth family of Riemannian metrics on $M$ parameterized by $u\in [0,1]$. Laplacians $\Delta^{(k)}$ depend on the metric via the Hodge star. Denote
$$\alpha=*\dot{*}=*_u\frac{\dd}{\dd u}*_u:\Omega^k(M,E)\ra \Omega^k(M,E)$$
-- the algebraic (ultralocal) operator on forms. Note that $*^2=\mr{id}$, hence $0=\frac{\dd}{\dd u}(*^2)=\dot{*}*+*\dot{*}$. Therefore, $\dot{*}*=-\alpha$.

We have
$$\dot{\Delta}^{(k)}=d \dot{d}^*+\dot{d}^* d = -d \dot{*}d*- d*d\dot{*}-\dot{*}d*d-*d\dot{*}d=-d\alpha d^*+dd^*\alpha-\alpha d^*d+d^*\alpha d$$
(we suppress subscripts in $d_E$, $d^*_E$ for brevity). Thus for the trace of the heat kernel, we have
\begin{multline*}
\frac{\dd}{\dd u}\mr{tr}_{\Omega^k}e^{-t\Delta^{(k)}}=-t\;\mr{tr}_{\Omega^k}\left(e^{-t\Delta^{(k)}}\dot{\Delta}^{(k)}\right)= \\
= -t\left( -\mr{tr}_{\Omega^{k-1}}\left(e^{-t\Delta^{(k-1)}} d^*d \alpha\right) + \mr{tr}_{\Omega^{k}}\left(e^{-t\Delta^{(k)}} dd^* \alpha\right)- \mr{tr}_{\Omega^{k}}\left(e^{-t\Delta^{(k)}} d^*d\alpha\right) + \mr{tr}_{\Omega^{k+1}}\left(e^{-t\Delta^{(k+1)}} dd^* \alpha\right)\right)
\end{multline*}
using cyclic property of the trace and the fact that $d$ and $d^*$ commute with Laplacians.
This implies the following identity:
\begin{multline*}
\frac{\dd}{\dd u}\left(\sum_{k=0}^n -\frac{(-1)^k\cdot k}{2}\;\mr{tr}_{\Omega^k}\left(e^{-t\Delta^{(k)}}\right)\right)=\frac{t}{2}\sum_{k=0}^n (-1)^k\;\mr{tr}_{\Omega^k}\left(e^{-t\Delta^{(k)}} \underbrace{(d d^*+d^* d)}_{\Delta^{(k)}} \alpha \right)=\\
=-\frac{1}{2}t\frac{\dd}{\dd t}\sum_{k=0}^n (-1)^k\;\mr{tr}_{\Omega^k} \left(e^{-t\Delta^{(k)}}\alpha\right)
\end{multline*}
Therefore, for the metric variation of the appropriate linear combination of zeta functions, we have
\begin{multline}
\frac{\dd}{\dd u} \left(\sum_{k=0}^n -\frac{(-1)^k\cdot k}{2}\;\zeta_{\Delta^{(k)}}(s)\right)= -\frac{1}{2}\;\frac{1}{\Gamma(s)} \int_0^\infty dt\; t^s\;\frac{\dd}{\dd t}\left(\sum_{k=0}^n (-1)^k \;\mr{tr}_{\Omega^k} \left(e^{-t\Delta^{(k)}}\alpha\right)\right)=\\
=\frac{1}{2}\sum_{k=0}^n(-1)^k \frac{s}{\Gamma(s)}\int_0^\infty dt\; t^{s-1}\;\int_M \mr{tr}_E K^{(k)}(x,x;t)\cdot \alpha(x,x) \label{RS torsion metric independence eq}
\end{multline}
where we integrated by parts in $t$.
For a fixed point $x\in M$, the analytic continuation of $\int_0^\infty dt\;t^{s-1}\; K^{(k)}(x,x;t)$ has no pole at $s=0$ (since we took $n$ to be odd). Hence, expression (\ref{RS torsion metric independence eq}) has a zero of 2nd order at $s=0$, which implies that $\left. \frac{\dd}{\dd s}\right|_{s=0}$ of (\ref{RS torsion metric independence eq}) vanishes. Therefore, $\frac{\dd}{\dd u}\tau_{RS}(M,E)=0$. We have proven that $\tau_{RS}$ stays constant along the chosen path in the space of Riemannian metrics. Since we may choose any path and the space of metrics is contractible (and in particular connected), $\tau_{RS}$ is constant on the space of metrics.
\end{proof}

\begin{rem} For $\gamma: M\ra O(m)$ a fiberwise rotation of $E$ (a {\it gauge transformation}), we have a corresponding transformed flat connection $\nabla^\gamma$ such that the corresponding twisted de Rham operator $d_E^\gamma$ makes the following square commutative.
$$
\begin{CD}
\Omega^k(M,E) @>d_E>> \Omega^{k+1}(M,E) \\
@V\gamma_* VV  @VV\gamma_* V \\
\Omega^k(M,E) @>d^\gamma_E>> \Omega^{k+1}(M,E)
\end{CD}
$$
Then the Ray-Singer torsion for the flat bundle $(E,\nabla)$ is equal to the torsion for the flat bundle $(E,\nabla^\gamma)$. Thus the Ray-Singer torsion descends to a function on the moduli space of acyclic flat $O(m)$-connections in $E$.
\end{rem}

\subsection{Example: torsion of circle}
Consider the case $M=S^1=\bR/\bZ$ with coordinate $x$ defined $\mod 1$, with Riemannian metric $g=(dx)^2$. Let $E$ be the trivial rank $2$ Euclidean vector bundle $S^1\times \bR^2\ra S^1$ with flat connection determined by the $\mathfrak{so}(2)$-valued $1$-form
\be \label{RS torsion for circle connection}a\, dx=\left(\begin{array}{cc}0 & \psi \\ -\psi & 0 \end{array}\right)\cdot dx\quad \in \mathfrak{so}(2)\otimes\Omega^1(S^1) \ee
The corresponding Laplacians are:
\begin{eqnarray*}
\Delta^{(0)}=d^*_E d_E:\quad &f\mapsto &-(f''+(af)'+af'+a^2 f) \\
\Delta^{(1)}=d_E d^*_E:\quad &g\, dx\mapsto &-(g''+(ag)'+ag'+a^2 g)\, dx
\end{eqnarray*}
Assume that in (\ref{RS torsion for circle connection}), $\psi\in \bR$ is a constant.
Eigenvalue spectra of $\Delta^{(0)}$, $\Delta^{(1)}$ coincide and are $\{(2\pi k+\psi)^2\}_{k\in \bZ}$ where each eigenvalue has multiplicity $2$ (corresponding eigenforms are proportional to $e^{2\pi ikx}$). Note that the twisted de Rham complex is acyclic iff $\psi\not\equiv 0\mod2\pi$. Corresponding zeta function is
$$\zeta_\Delta(s)=2\sum_{k\in\bZ}(2\pi k+\psi)^{-2s}=2\cdot (2\pi)^{-2s}(\zeta(2s,\frac{\psi}{2\pi})+\zeta(2s,1-\frac{\psi}{2\pi}))$$
with $\zeta(s,p)=\sum_{k=0}^\infty (k+p)^{-s}$ the Hurwitz zeta function. One calculates
$$\det \Delta^{(0)}=\det \Delta^{(1)}=e^{-\zeta'_\Delta(0)}(0)=\left(2\sin\frac\alpha 2\right)^4$$
which implies the following answer for the Ray-Singer torsion
$$\tau_{RS}(S^1,E)=\left(\det \Delta^{(1)}\right)^{1/2}=\left(2\sin\frac\alpha 2\right)^2$$
Note that l.h.s. here can be written as 
$\det(e^a-\mr{id})=|e^{i\psi}-1|^2$.
Thus the Ray-Singer torsion for the circle coincides (up to taking an inverse) with the R-torsion for representation $h: \pi_1(S^1)\ra O(2) $ given by the holonomy of connection $\nabla$ in $E$.

Solution to the heat equation in this case can also be written explicitly in terms of Jacobi $\theta$-function.

\lec{Lecture 14, 29.05.2014}

\subsection{Non-acyclic case}
In case when the flat Euclidean bundle $E$ over $M$ is not acyclic, the Ray-Singer torsion is defined as follows.
Hodge inner product $\int_M (\bt\stackrel{\wedge}{,} *\bt)_E$ induces an inner product on cohomology
$$(,)_H:\quad H^k(M,E)\otimes H^k(M,E)\ra \bR $$
via representation of cohomology classes by harmonic forms. Denote $\mu_H$ the element in the determinant line of cohomology $\mu_H\in\Det\; H^\bt(M,E)/\{\pm 1\}$ corresponding to an arbitrary orthonormal basis in cohomology w.r.t. $(,)_H$. Ray-Singer torsion is defined then as
\be\tau_{RS}(M,E)=
\prod_{k=0}^n \left({\det}'_{\Omega^k(M,E)}\Delta^{(k)}\right)^{-\frac{(-1)^k\; k}{2}}
\cdot\mu_H\qquad \in  \Det\; H^\bt(M,E)/\{\pm 1\} \label{RS torsion non-acyclic}\ee
where $\det'$ is the zeta regularized determinant, where in the corresponding zeta functions we only sum over non-vanishing (positive) eigenvalues of respective Laplacians.

Theorem \ref{thm: metric independence of RS torsion} extends to the non-acyclic case. In general, the product of determinants in the r.h.s. of (\ref{RS torsion non-acyclic}) and $\mu_H$ are separately dependent on Riemannian metric $g$, but the whole torsion is not.

Poincar\'e duality argument (Theorem \ref{thm: RS torsion of even dim mfd}) yields in the non-acyclic case that, for $\dim M$ even, the product of determinants in (\ref{RS torsion non-acyclic}) is equal to $1$.

\subsection{Cheeger-M\"uller theorem}
Let $(M,g)$ be a smooth oriented closed Riemannian manifold endowed with a Euclidean flat bundle $(E,(,)_E,\nabla)$. Let also $X$ be a cellular decomposition of $M$. Choose a vertex $x_0$ of $X$ as a base point. For $\gamma$ a closed loop in $M$ from $x_0$ to $x_0$, the holonomy $\mr{hol}_\nabla(\gamma) \in \mr{Aut}(E_{x_0})\simeq O(m)$ is a norm-preserving endomorphism of the fiber of $E$ over $x_0$. Due to flatness of $\nabla$, the holonomy only depends on the class of $\gamma$ up to homotopy. Thus, we have the holonomy representation $$h=\mr{hol}_\nabla:\quad  \pi_1(M)=\pi_1(X)\ra O(m)$$

\begin{thm}[Cheeger-M\"uller]\label{thm: Cheeger-Muller}
Ray-Singer torsion of $(M,E)$ and $R$-torsion of $X$ with holonomy representation $h$ are equal:
\be \tau_{RS}(M,E)=\tau_h(X)\qquad \in \Det\; H^\bt(M,E)/\{\pm 1\} \label{Cheeger-Muller}\ee
\end{thm}

Implicit in the theorem is the quasi-isomorphism of cochain complexes
\be \Omega^\bt(M,E) \quad \stackrel{\mr{quasi-iso}}{\sim}\quad C^\bt(X,h)=\bR^m\otimes_{\bZ[\pi_1]}C^\bt(\widetilde X,\bZ) \label{Cheeger-Muller: forms vs cochains}\ee

\begin{rem} 
Torsion in the r.h.s. of (\ref{Cheeger-Muller}) is the {\it cochain} version of the $R$-torsion, based on twisted cochains $C^\bt(X,h)=\bR^m\otimes_{\bZ[\pi_1]}C^\bt(\widetilde X,\bZ)$, rather than twisted chains $C_\bt(X,h)=\bR^m\otimes_{\bZ[\pi_1]}C_\bt(\widetilde X,\bZ)$. Respective torsions are related by $\tau_\mr{cochain}=\tau_\mr{chain}^{-1}$ (Corresponding transition matrices between bases in chains/cochains in the definition of torsion (\ref{def: torsion}) are inverse transpose of each other).
Note that for the corresponding determinant lines, one has $\Det\; H^\bt = (\Det\; H_\bt)^{-1} $.
\end{rem}

Complex on the right side of (\ref{Cheeger-Muller: forms vs cochains}) can be realized in the following way: trivialize $E$ over the barycenter $\dot e$ (chosen point) of every cell $e$ of $X$. For cells $e'\subset e$, we have a group element $E_X(e',e)=\mr{hol}_{\gamma: \dot e'\ra \dot e}\in \mr{Hom}(E_{\dot e'},E_{\dot e})\simeq  O(m)$ -- the transition function for the {\it cellular local system} $E_X$ over $X$, corresponding to the flat bundle $E$. Here $\gamma$ is any path from $\dot e'$ to $\dot e$, staying inside the cell $e$. Then we construct the cochain complex of $X$ twisted by $E_X$ as $C^k(X,E_X)=\bR^m\otimes C^k(X)$ with differential
$$d_{E_X}:\quad v\otimes e^*\mapsto \sum_{e_i\subset X} a_i\cdot E_X(e_i,e) (v) \otimes e_i^*$$
for $v\in \bR^m$ any vector and $a_i\in\bZ$ the coefficients of the cochain $de^*=\sum_{e_i\subset X} a_i e_i^*$; $e^*$ denotes the basis cochain in $C^k(X,\bZ)$ associated to the cell $e$.

Then $C^\bt(X,E_X)$ is chain isomorphic to $C^\bt(X,h)$. Integration map
$$\alpha\mapsto \sum_{e\subset X} \left(\int_{e\ni x} \mr{hol}_\nabla(\dot e\ra x)_*\alpha\right)\cdot e^*$$
is an explicit quasi-isomorphism $\Omega^\bt(M,E)\ra C^\bt(X,E_X)$.

Torsion in the r.h.s. of (\ref{Cheeger-Muller}) can be written, using the inner product on cochain spaces $C^\bt(X,h)$ induced by declaring the cellular basis an orthonormal basis, as a product of powers of determinants of discrete Laplacians (\ref{Lm: alg RS formula eq}):
$$\tau(X,h)=\prod_{k=0}^n \left({\det}'_{C^k(X,h)}\Delta^{(k)}_{X,h}\right)^{-\frac{(-1)^k k}{2}}\cdot \mu_{H^\bt(X,h)}$$
with $\Delta^{(k)}_{X,h}=dd^*+d^*d: C^k(X,h)\ra C^k(X,h)$, with $d^*$ being the transpose of $d$ with respect to cellular inner product on cochains; $\mu_{H^\bt(X,h)}$ is the element in the determinant line of cohomology associated to the inner product on cohomology induced from cellular metric via harmonic representatives.

There exist several conceptually different proofs of Theorem \ref{thm: Cheeger-Muller}, all of them containing a fair bit of nontrivial analysis.
\begin{itemize}
\item Proof by Werner M\"uller \cite{Muller78}. The discussion above, exhibiting $R$-torsion as a combinatorial approximation for the analytic torsion is the starting point for this proof: knowing that the $R$-torsion is invariant under cellular subdivision, one analyzes the limit of dense subdivision for the $R$-torsion and how the resulting determinant of cellular Laplacians approach asymptotically the zeta-regularized determinants of Laplacians on differential forms. A crucial technical result in this treatment is the Dodziuk-Patodi theorem \cite{Dodziuk-Patodi} establishing convergence of eigenvalues of combinatorial Laplacian on simplicial cochains of a triangulation to eigenvalues of Laplacian on forms, in the limit of the mesh of the triangulation approaching zero.
\item Proof by Jeff Cheeger \cite{Cheeger79}. Following the idea of Hirzebruch's proof of signature theorem, Cheeger observes that the double of a manifold $2M=M\sqcup \bar M$ is always cobordant to a sphere $S^n$ and hence can be constructed out of $S^n$ by a sequence of surgeries along embedded spheres $S^k$. The proof of (\ref{Cheeger-Muller}) is then reduced to the check for $S^n$ with trivial local system and a (non-straightforward) proof that the equality of torsions is preserved by elementary surgeries.
\item Proof by Burghelea-Friedlander-Kappeler \cite{BFK} and proof by by Maxim Braverman \cite{Braverman} based on Witten's deformation $d\mapsto e^{-f}d e^f$ of the de Rham complex by a Morse-Smale function. Proof of \cite{BFK} also extensively relies on the Mayer-Vietoris-type gluing formula for zeta-regularized determinants by the same authors.
\end{itemize}

\subsection{Inclusion/exclusion property}
\begin{lemma}\label{Lm: inclusion-exclusion}
Let $Z$ be a CW-complex endowed with an $O(m)$-local system, and let  $X,Y$ be two subcomplexes such that $Z=X\cup Y$. Then for the $R$-torsions we have
\be \tau(Z,E)=\tau(X,E|_X)\cdot\tau(Y,E|_Y)\cdot \tau(X\cap Y,E|_{X\cap Y})^{-1} \label{Lm: inclusion-exclusion, eq}\ee
\end{lemma}
Left hand side lives in $\Det \; H^\bt(Z,E)/\{\pm1\}$ and right hand side lives in $\Det \; H^\bt(X,E)\otimes \Det \; H^\bt(Y,E)\otimes (\Det \; H^\bt(X\cap Y,E))^{-1}/\{\pm1\}$ (we are omitting appropriate restrictions of $E$ in the notation). These two objects are canonically isomorphic, since we have the Mayer-Vietoris long exact sequence (twisted by $E$):
$$ H^\bt(Z,E)\ra H^\bt(X,E)\oplus H^\bt(Y,E)\ra H^\bt(X\cap Y,E) \ra H^{\bt+1}(Z,E) $$
Applying Lemma \ref{Lm: Det(C)=Det(H)} (``determinant of a complex equals determinant of its homology''), we obtain, since Mayer-Vietoris sequence is {\it exact},
\begin{multline*}
\Det (\mbox{Mayer-Vietoris})=\\
=\Det\;H^\bt(Z,E)\otimes (\Det\;(H^\bt(X,E)\oplus H^\bt(Y,E)))^{-1}\otimes \Det\; H^\bt(X\cap Y,E)\cong \bR
\end{multline*}
This proves that l.h.s. and r.h.s. of (\ref{Lm: inclusion-exclusion, eq}) do indeed live in the same place. Equality (\ref{Lm: inclusion-exclusion, eq}) itself follows directly from applying multiplicativity lemma (Lemma \ref{Lm: multiplicativity of torsions, non-acyclic case}, cf. also Section \ref{sec: multiplicativity via det lines}) to the Mayer-Vietoris short exact sequence in cochains:
$$ 0\ra C^\bt(Z,E)\ra C^\bt(X,E)\oplus C^\bt(Y,E)\ra C^\bt(X\cap Y,E)\ra 0  $$

A related property is the following.
\begin{lemma}
Let $X$ be a CW-complex with an $O(m)$-local system $E$, an let $Y\subset X$ be a subcomplex. Then
\be \tau(X;E)=\tau(X,Y;E)\cdot\tau(Y;E) \label{tau(X) = tau(X,Y) tau(Y)}\ee
\end{lemma}
This follows from multiplicativity lemma applied to the short exact sequence of the pair $(X,Y)$:
$$0\ra C^\bt(X,Y;E)\ra C^\bt(X;E) \ra C^\bt(Y;E)\ra 0$$
and the corresponding long exact sequence in cohomology.

Even more generally, one can start with a triple $X\supset Y\supset Z$ and obtain the relation between torsions of the three pairs:
$$\tau(X,Z;E)=\tau(X,Y;E)\cdot\tau(Y,Z;E)$$
-- a version of Lemma \ref{Lm: multiplicativity of Whitehead torsions of pairs} adapted for the $R$-torsion (without requiring that pairs be deformation retracts).

\begin{corollary}
For $M$ a manifold with boundary $\dd M$,  $\dim M$ even, with $E$ a Euclidean flat bundle, one has
\be \tau(\dd M,E)=\tau(M,E)^2 \label{tau(bdry) = tau(bulk)^2}\ee
\end{corollary}
Indeed, by Poincar\'e duality for torsions,
$\tau(M)=\tau(M,\dd M)^{(-1)^{\dim M+1}}=\tau(M,\dd M)^{-1}$. On the other hand, by (\ref{tau(X) = tau(X,Y) tau(Y)}),
$\tau(M)=\tau(M,\dd M)\cdot \tau(\dd M)$, which implies (\ref{tau(bdry) = tau(bulk)^2}).

\begin{example}[Tosion for a pair of pants]
Take $M$ to be a pair of pants (2-sphere with $3$ disks removed) and $E$ an $O(2)$-local system. Since $\pi_1(M)=\langle \gamma_1,\gamma_2,\gamma_3\rangle/\gamma_1\cdot\gamma_2\cdot\gamma_3=1$ with $\gamma_i$ the generator corresponding to a based loop contractible to $i$-th boundary circle, $E$ is defined up to equivalence by $3$ holonomies, $g_1,g_2,g_3\in O(2)$, satisfying $g_1g_2g_3=1$, up to conjugation by a common element $h\in O(2)$. Assuming that $E$ is acyclic (which, by Poincar\'e-Lefschetz duality for homology happens iff $E|_\dd$ is acyclic, which we know to occur iff all $g_i$ are different from $1$), we have
$$\tau(\dd M=S^1\sqcup S^1\sqcup S^1,E|_\dd)=\det((1-g_1)(1-g_2)(1-g_3))$$
and, by (\ref{tau(bdry) = tau(bulk)^2}),
\be \tau(M,E)=\det((1-g_1)(1-g_2)(1-g_3))^{1/2} \label{tau(pants)}\ee
\end{example}
Interpretation of this number is as follows. Note that although we can choose the local system $E$ to be acyclic on the boundary $\dd M$, it cannot be acyclic over $M$, since the Euler characteristic $\chi(M)=-1$ is nonzero. More precisely, $H^2(M;E)$ is always zero and $H^0(M;E)$ can be made zero by a choice of $E$.
Vanishing of $H^\bt(\dd M;E)$ implies that the long exact sequence of the pair $(M,\dd M)$ gives an isomorphism $H^\bt(M,\dd M;E)\xra{\sim}H^\bt(M;E) $. Together with Poincar\'e-Lefschetz non-degenerate pairing $H^\bt(M,\dd M;E)\otimes H^{2-\bt}(M;E)\ra \bR$, this gives a non-degenerate inner product $H^1(M;E)\otimes H^1(M;E)\ra\bR$. Choosing the element $\mu\in \Det\; H^\bt(M;E)$ corresponding to an orthonormal basis in $H^1$, we can compare the $R$ torsion with it to get a number. (\ref{tau(pants)}) is this number.

\subsection{Torsion and the moduli space of flat bundles}
Fix a Lie group $G$ and an orthogonal representation $G\ra O(m)$. For $M$ a manifold, consider the $R$ torsion of $M$ as a function of a flat $G$-bundle on $P$ (via the associated Euclidean flat bundle $E_P=P\times_G \bR^m$). $R$-torsion $\tau(M,P)\in \Det\; H^\bt(M,E_P)/\{\pm 1\}$ becomes a section of a real line bundle $\mc{L}$ over the moduli space of flat $G$-bundles over $M$,
$$\mc{M}(M,G)= \mr{Hom}(\pi_1(M),G)/G$$
with fiber $\Det\; H^\bt(M,E_P)$ over a point corresponding to flat $G$-bundle $P$. (More precisely, since torsion is defined up to sign, $\tau^2$ is a section of $\mc{L}^{\otimes 2}\ra\mc{M}$).

Consider the special case when $G\xra{\mr{Ad}} O(m=\dim G) $ is the adjoint representation, so that $E_P=\mr{ad}(P)$.
Then the tangent space (where it exists) to the moduli space $T_P \mc{M}(M,G)=H^1(M,\mr{ad}(P))$. Let $G$ be semisimple. Semisimplicity implies that {\it generic} flat $G$-bundle is {\it irreducible}, i.e. $H^0(M,\mr{ad}(P))=0$ (i.e. 
holonomy representation $\pi_1(M)\ra G\xra{\mr{Ad}} O(m)$ does not have proper invariant subspaces in $\bR^m$).

\subsubsection{2-dimensional case}
In case of $M$ a closed oriented 2-dimensional surface, using Poincar\'e duality, we have $H^0=H^2=0$ for a generic flat bundle, hence
$$\Det\; H^\bt(M,\mr{ad}(P))= (\Det\; H^1(M,\mr{ad}(P)))^{-1}=\wedge^\mr{top}T^*\mc{M}_P(M,G)$$
Thus $\mc{L}$ is the bundle of volume forms on $\mc{M}(M,G)$. There is a natural section of this line bundle -- the Liouville volume form $\frac{\omega^{\wedge\dim \mc{M}}}{(\dim \mc{M})!}$ associated to the Atiyah-Bott symplectic structure $\omega\in \Omega^2(\mc{M}(M,G))$. This symplectic structure is in fact just the Poincar\'e pairing
$$\underbrace{H^1(M,\mr{ad}(P))\otimes H^1(M,\mr{ad}(P))}_{T_P\mc{M}(M,G)}\ra \bR$$
Torsion $\tau(M,P)$ is also a section of $\mc{L}\ra \mc{M}$ and, again by Poincar\'e duality, the same section:
\be \tau(M,P)=\frac{\omega^{\wedge\dim \mc{M}}}{(\dim \mc{M})!}\quad \in \Omega^{\dim \mc{M}}(\mc{M}) \label{tau(surface) = symp vol element}\ee
cf. \cite{Witten91}.

\subsubsection{3-dimensional case}
For a closed oriented $3$-manifold $M$, we have, by Poincar\'e duality, $H^2(M,E)\simeq (H^1(M,E))^*$. Thus, for an irreducible flat bundle $P$, $\Det\; H^\bt(M,\mr{ad}(P))= (\Det\;H^1(M,\mr{ad}(P)))^{-2}=\left(\wedge^{\mr{top}}T^*_P\mc{M}(M,G)\right)^{\otimes 2}$. Thus, the {\it square root} of torsion is a volume form on the moduli space of flat bundles, and can be integrated, cf. the appearance of the square root of torsion in the 1-loop partition function of Chern-Simons theory \cite{Witten89}.

\subsubsection{Symplectic volume of the moduli space of flat bundles on a surface}
Using (\ref{tau(surface) = symp vol element}) and the gluing (inclusion/exclusion) property of $R$-torsion, Witten \cite{Witten91} calculates the symplectic volume of the moduli space of flat $G$-bundles (with $G$ compact semisimple) on a closed oriented surface $\Sigma$:
$$\mr{Vol}(\mc{M}(\Sigma,G))=\int_{\mc{M}(\Sigma,G)}\frac{\omega^{\wedge\dim \mc{M}}}{(\dim \mc{M})!}=\quad \frac{\# Z(G)\cdot (\mr{Vol}(G))^{2g-2}}{(2\pi)^{\dim\mc{M}}}\cdot\sum_R\frac{1}{(\dim R)^{2g-2}}$$
where $g\geq 2$ is the genus of $\Sigma$ and $R$ runs over irreducible representations of $G$ modulo isomorphism; $Z(G)$ is the center of $G$.
E.g. in case $G=SU(2)$, one has
$$\mr{Vol}(\mc{M}(\Sigma,SU(2)))=\frac{2}{(2\pi^2)^{g-1}}\cdot\sum_{n=1}^\infty \frac{1}{n^{2g-2}}$$
The basic technique behind this formula is cutting the surface into pairs of pants and using the Mayer-Vietoris gluing formula for torsions to assemble the result for the surface from torsions for individual pairs of pants. Sum over irreducible representations originates from Peter-Weyl theorem (cf. Schur's orthogonality relations for characters of finite groups).

One can construct a unitary complex line bundle $L$ over $\mc{M}(\Sigma,G)$ with a connection $\nabla_L$ of curvature $2\pi i\omega$. Moduli space $\mc{M}(\Sigma,G)$ inherits a complex structure from any choice of complex structure on $\Sigma$. 
Together with the Atiyah-Bott symplectic structure, this makes $\mc{M}(\Sigma,G)$ a K\"ahler manifold. The space of {\it holomorphic} sections of $L^{\otimes k}$ arises as the space of states associated to a boundary surface for Chern-Simons theory at {\it level} $k$ on a 3-manifold with boundary, cf. \cite{Witten89}. By Hirzebruch-Riemann-Roch theorem,\footnote{
Integral in the r.h.s. actually calculates the Euler characteristic of the Dolbeault complex twisted by $L^{\otimes k}$. However, higher (degree $\geq 1$) cohomology vanishes by Kodaira vanishing theorem.
}
\be \dim H^0_{\bar\dd}(\mc{M}(\Sigma,G),L^{\otimes k})=\int_{\mc{M}(\Sigma,G)} Td(\mc{M})\cdot e^{k\omega} \label{Verlinde space, RRH}\ee
with $Td$ the Todd class of the tangent bundle of the moduli space. Since $Td=1+(\mbox{higher forms})$, the right hand side is a polynomial in $k$ of degree $\dim\mc{M}/2$, with the top coefficient being the symplectic volume of $\mc{M}$. The object on the right hand side is related to the representation theory of the affine Lie algebra $\hat{\mathfrak{g}}$, due to Verlinde. E.g. for $G=SU(2)$, Verlinde formula is:
\begin{multline*} \dim H^0_{\bar\dd}(\mc{M}(\Sigma,SU(2)),L^{\otimes k})= \left(\frac{k+2}{2}\right)^{g-1}\sum_{j=0}^{k-1}\frac{1}{\left(\sin\frac{\pi (j+1)}{k+2}\right)^{2g-2}} =
\\\underbrace{=}_{g\geq 2}
\mr{Vol}(\mc{M}(\Sigma,SU(2)))\cdot k^{3g-3}+ \mbox{polynomial of degree }3g-4\mbox{ in }k
\end{multline*}

\end{document}